\documentclass{amsart}
\usepackage[left=2cm, right=2cm]{geometry}
\usepackage[utf8]{inputenc}
\usepackage{amsfonts, amsthm, amssymb, mathtools,etoolbox}
\usepackage{blindtext}
\usepackage[colorlinks=true, urlcolor=blue, linkcolor=blue, citecolor=blue]{hyperref}
\usepackage{tikz,tikz-cd}
\usepackage{array}
\usepackage{pifont}
\usepackage{cleveref}
\usepackage[style=alphabetic,sorting=nyt, maxnames=6]{biblatex}
\renewbibmacro{in:}{}
\tikzset{pullback/.style={minimum size=1.2ex,path picture={
\draw[opacity=1,black,-,#1] (-0.5ex,-0.5ex) -- (0.5ex,-0.5ex) -- (0.5ex,0.5ex);%
}}}
\usetikzlibrary{positioning}
\usetikzlibrary{calc}

\theoremstyle{plain}
\newtheorem{theorem}{Theorem}[section]
\newtheorem{proposition}[theorem]{Proposition}
\newtheorem{lemma}[theorem]{Lemma}
\newtheorem{corollary}[theorem]{Corollary}

\newtheorem{fact}[theorem]{Fact}

\theoremstyle{definition}
\newtheorem{example}[theorem]{Example}
\newtheorem{definition}[theorem]{Definition}
\newtheorem{notation}[theorem]{Notation}
\newtheorem{question}[theorem]{Question}

\theoremstyle{remark}
\newtheorem{remark}[theorem]{Remark}

\newcommand{\dq}[1]{``#1"}
\newcommand{\memo}[1]{\textcolor{red}{memo: #1}}
\newcommand{\invmemo}[1]{}
\newcommand{\revmemo}[1]{}

\newcommand{\N}{\mathbb{N}}
\newcommand{\Z}{\mathbb{Z}}

\newcommand{\R}{\mathbb{R}}
\newcommand{\C}{\mathcal{C}}
\newcommand{\comp}{\mathbb{C}}
\newcommand{\D}{\mathcal{D}}

\newcommand{\id}{\mathrm{id}}

\newcommand{\Set}{\mathbf{Set}}

\newcommand{\demph}[1]{\textbf{#1}}
\newcommand{\ob}{\mathrm{ob}}

\newcommand{\X}{\mathbb{X}}
\newcommand{\Y}{\mathbb{Y}}
\newcommand{\W}{\mathbb{W}}
\newcommand{\Star}{\mathbb{S}}
\newcommand{\gS}{\mathbb{S}}
\newcommand{\A}{\mathbb{A}}

\newcommand{\Pow}{\mathcal{P}}
\newcommand{\Pf}{\Pow_{\mathrm{fin}}}
\newcommand{\Gs}{\mathbf{Game}}
\newcommand{\Graphs}{\mathbf{Graphs}}
\newcommand{\nsum}{\oplus}
\newcommand{\Alg}[1]{\mathbf{Alg}_{#1}}
\newcommand{\Coalg}[1]{\mathbf{Coalg}_{#1}}
\newcommand{\RecCoalg}[1]{\mathbf{RecCoalg}_{#1}}

\newcommand{\rel}{\to}

\newcommand{\acc}{\succeq}
\newcommand{\accneq}{\succ}

\newcommand{\str}{\theta}
\newcommand{\Image}{\mathrm{Im}}
\newcommand{\Nim}[1]{\mathrm{Nim}_{#1}}

\newcommand{\ElM}{\mathrm{ElM}}
\renewcommand{\H}{\mathbb{H}}
\newcommand{\epi}{twoheadrightarrow}
\newcommand{\mono}{rightarrowtail}

\newcommand{\B}{\mathcal{B}}
\newcommand{\ADJ}[4]
    {
    \begin{tikzcd}[ampersand replacement = \&, column sep = small]
        {#1}
        \ar[rr, shift right=1.3ex, "{#2}"']
        \&\perp\&
        {#3}
        \ar[ll, shift right=1.3ex,"{#4}"']
    \end{tikzcd}
    }

\newcommand{\oc}{\mathsf{Outcome}}
\newcommand{\Moc}{\mathsf{MOutcome}}
\newcommand{\rd}{\mathsf{\xi}}

\newcommand{\NP}{\mathsf{NP}}
\newcommand{\np}{\mathsf{np}}
\newcommand{\bin}{\mathsf{bin}}
\newcommand{\MNP}{\mathsf{MNP}}
\newcommand{\mnp}{\mathsf{mnp}}
\newcommand{\End}{\mathsf{End}}
\newcommand{\Empty}{\mathsf{Empty}}
\newcommand{\emptyFunc}{\mathsf{empty}}
\newcommand{\True}{\mathsf{True}}
\newcommand{\FinTime}{\mathsf{FinTime}}
\newcommand{\mex}{\mathsf{mex}}
\newcommand{\xem}{\mathsf{xem}}
\newcommand{\Mex}{\mathsf{Mex}}
\newcommand{\Xem}{\mathsf{Xem}}
\newcommand{\val}{\mathsf{v}}
\newcommand{\hylo}{\mathsf{hylo}}
\newcommand{\G}{\mathsf{Grundy}}
\newcommand{\Gb}{\mathsf{b}}
\newcommand{\BirthDay}{\mathsf{BirthDay}}
\newcommand{\Remoteness}{\mathsf{Remoteness}}
\newcommand{\rbmult}{{\ast_{\H}}}
\DeclareMathOperator*{\colim}{colim}
\newcommand{\pr}[1]{{\color{red}\{ #1 \}}}
\newcommand{\pg}[1]{{\color{green!60!black}\{ #1 \}}}
\newcommand{\pb}[1]{{\color{blue}\{ #1 \}}}
\newcommand{\pd}[1]{{\color{black}\{ #1 \}}}

\title{Games as recursive coalgebras\\
A categorical view on the Nim-sum}
\author{Ryuya Hora}
\address{Graduate School of Mathematical Sciences, University of Tokyo, Tokyo, Japan}
\email{hora@ms.u-tokyo.ac.jp}
\date{\today}

\begin{document}
\begin{abstract}
    In 1901, Bouton proved that a winning strategy for the game of Nim is given by the bitwise XOR, called the nim-sum. But why does such a weird binary operation work?
    Led by this question, this paper introduces a categorical reinterpretation of combinatorial games and the nim-sum. The main categorical gadget used here is \textit{recursive coalgebras}, which allow us to redefine games as ``graphs on which we can conduct recursive calculations'' in a concise and precise way.

    For game-theorists, we provide a systematic framework to decompose an impartial game into simpler games and synthesize the quantities on them, which generalizes the nim-sum rule for the Conway addition. To read the first half of this paper, the categorical preliminaries are limited to the definitions of categories and functors.

    For category theorists, this paper offers a nicely behaved category of games $\mathbf{Game}$, which is a locally finitely presentable symmetric monoidal closed category comonadic over $\mathbf{Set}$ admitting a subobject classifier!

    As this paper has several ways to be developed, we list seven open questions in the final section.
\end{abstract}
\maketitle
\tableofcontents
Throughout this paper, $\N$ denotes the set of all non-negative integers $\N = \{0,1,2, \dots\}$.
\section{Introduction}\label{sec:Introduction}
\subsection{Why does the nim-sum appear?}
\invmemo{Context in combinatorial game theory}
In 1901, Bouton discovered the remarkable (and now very famous) winning strategy of the game \demph{nim} in \cite{bouton1901nim}.
In $n$-heap nim, there are $n$ heaps of stones, with $a_1, \dots, a_n$ stones in each heap. Two players take turns removing stones, but on each turn, they must remove stones from one heap (and \dq{removing $0$ stones} is not allowed.) The player who can no longer make a move loses\footnote{Usually, the rule is stated as \dq{the player who takes the last stone wins,} but it does not make sense in the particular situation where $a_1 = a_2 = \cdots = a_n = 0$ from the start.}. For example, \Cref{fig:NimPlay} is a typical play from the state $(a_1,a_2,a_3)=(2,3,3)$, where $A$ wins.
\begin{figure}[ht]
\centering
\begin{tikzpicture}[>=Latex, thick, scale=0.7]
  \def\rowgap{3.0} 
  \def\r{0.8}      
  \def\dx{2.0}     

  \coordinate (R0) at (0, 0*\rowgap);
  \path (R0) ++(-\dx,0) coordinate (R0H1);
  \path (R0) ++( 0,  0) coordinate (R0H2);
  \path (R0) ++( \dx,0) coordinate (R0H3);

  \draw (R0H1) circle (\r);
  \draw (R0H2) circle (\r);
  \draw (R0H3) circle (\r);

  \fill ($(R0H1)+(-0.25,0)$) circle (2pt);
  \fill ($(R0H1)+( 0.25,0)$) circle (2pt);
  \fill ($(R0H2)+(-0.30, 0.20)$) circle (2pt);
  \fill ($(R0H2)+( 0.00,-0.25)$) circle (2pt);
  \fill ($(R0H2)+( 0.30, 0.20)$) circle (2pt);
  \fill ($(R0H3)+(-0.25,0)$) circle (2pt);
  \fill ($(R0H3)+( 0.25,0)$) circle (2pt);

  \node[right] at ($(R0H3)+(\r+0.6,0)$) {$(2,3,2)$};

  \coordinate (R1) at (0, -1*\rowgap);
  \path (R1) ++(-\dx,0) coordinate (R1H1);
  \path (R1) ++( 0,  0) coordinate (R1H2);
  \path (R1) ++( \dx,0) coordinate (R1H3);

  \draw (R1H1) circle (\r);
  \draw (R1H2) circle (\r);
  \draw (R1H3) circle (\r);

  \fill ($(R1H1)+(-0.25,0)$) circle (2pt);
  \fill ($(R1H1)+( 0.25,0)$) circle (2pt);
  \fill ($(R1H2)+(-0.30, 0.20)$) circle (2pt);
  \fill ($(R1H2)+( 0.00,-0.25)$) circle (2pt);
  \fill ($(R1H2)+( 0.30, 0.20)$) circle (2pt);
  \fill ($(R1H3)+(0,0)$) circle (2pt);

  \node[right] at ($(R1H3)+(\r+0.6,0)$) {$(2,3,1)$};
  \draw[->] ($(R0)+(0,-1.1)$) -- ($(R1)+(0,1.1)$) node[midway,right] {$A$};

  \coordinate (R2) at (0, -2*\rowgap);
  \path (R2) ++(-\dx,0) coordinate (R2H1);
  \path (R2) ++( 0,  0) coordinate (R2H2);
  \path (R2) ++( \dx,0) coordinate (R2H3);

  \draw (R2H1) circle (\r);
  \draw (R2H2) circle (\r);
  \draw (R2H3) circle (\r);

  \fill ($(R2H1)+(0,0)$) circle (2pt);
  \fill ($(R2H2)+(-0.30, 0.20)$) circle (2pt);
  \fill ($(R2H2)+( 0.00,-0.25)$) circle (2pt);
  \fill ($(R2H2)+( 0.30, 0.20)$) circle (2pt);
  \fill ($(R2H3)+(0,0)$) circle (2pt);

  \node[right] at ($(R2H3)+(\r+0.6,0)$) {$(1,3,1)$};
  \draw[->] ($(R1)+(0,-1.1)$) -- ($(R2)+(0,1.1)$) node[midway,right] {$B$};

  \coordinate (R3) at (0, -3*\rowgap);
  \path (R3) ++(-\dx,0) coordinate (R3H1);
  \path (R3) ++( 0,  0) coordinate (R3H2);
  \path (R3) ++( \dx,0) coordinate (R3H3);

  \draw (R3H1) circle (\r);
  \draw (R3H2) circle (\r);
  \draw (R3H3) circle (\r);

  \fill ($(R3H1)+(0,0)$) circle (2pt);
  \fill ($(R3H3)+(0,0)$) circle (2pt);

  \node[right] at ($(R3H3)+(\r+0.6,0)$) {$(1,0,1)$};
  \draw[->] ($(R2)+(0,-1.1)$) -- ($(R3)+(0,1.1)$) node[midway,right] {$A$};

  \coordinate (R4) at (0, -4*\rowgap);
  \path (R4) ++(-\dx,0) coordinate (R4H1);
  \path (R4) ++( 0,  0) coordinate (R4H2);
  \path (R4) ++( \dx,0) coordinate (R4H3);

  \draw (R4H1) circle (\r);
  \draw (R4H2) circle (\r);
  \draw (R4H3) circle (\r);

  \fill ($(R4H1)+(0,0)$) circle (2pt);

  \node[right] at ($(R4H3)+(\r+0.6,0)$) {$(1,0,0)$};
  \draw[->] ($(R3)+(0,-1.1)$) -- ($(R4)+(0,1.1)$) node[midway,right] {$B$};

  \coordinate (R5) at (0, -5*\rowgap);
  \path (R5) ++(-\dx,0) coordinate (R5H1);
  \path (R5) ++( 0,  0) coordinate (R5H2);
  \path (R5) ++( \dx,0) coordinate (R5H3);

  \draw (R5H1) circle (\r);
  \draw (R5H2) circle (\r);
  \draw (R5H3) circle (\r);

  \node[right] at ($(R5H3)+(\r+0.6,0)$) {$(0,0,0)$};
  \draw[->] ($(R4)+(0,-1.1)$) -- ($(R5)+(0,1.1)$) node[midway,right] {$A$};
\end{tikzpicture}
\caption{A play of $3$-heap Nim, where $A$ wins.}
\label{fig:NimPlay}
\end{figure}
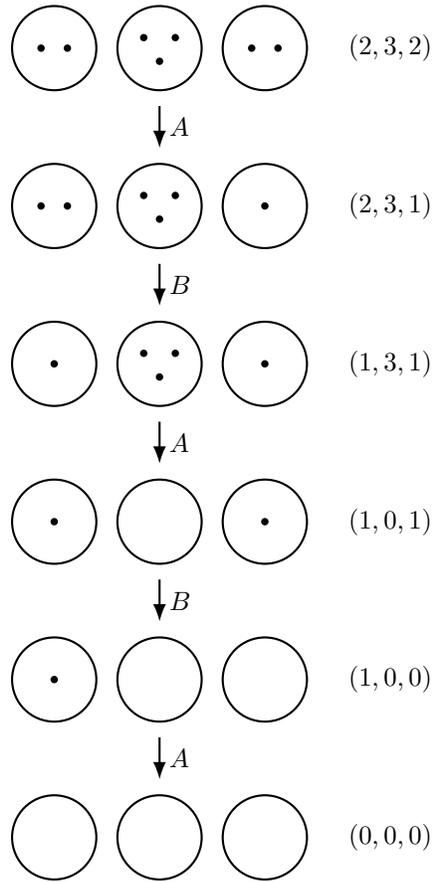

Bouton's winning strategy is based on a nicely designed group structure on $\N$ called \demph{nim-sum}. (We will define the nim-sum again in a more formal way in \Cref{def:NimSum2}.)
\begin{definition}\label{def:NimSum}
    The \demph{nim-sum} of two natural numbers is the bit-wise excluded disjunction.
\end{definition}
For example, the nim-sum of $3$ and $5$ is $6$ (\Cref{fig:NimSum}).
\begin{figure}[ht]
\centering
\begin{tikzpicture}[thick,>=Latex,scale=1]
  \def\w{0.6}   
  \def\h{0.8}   

  \coordinate (C2) at (0,0);
  \coordinate (C1) at (\w,0);
  \coordinate (C0) at (2*\w,0);

  \newcommand{\digit}[3]{%
    \node at ($(#1)+(0.5*\w,-#3*\h-0.5*\h)$) {$#2$};
  }

  \node[anchor=east] at ($(C2)+(-0.5*\w,-0.5*\h)$) {$3=$};
  \digit{C2}{0}{0}
  \digit{C1}{1}{0}
  \digit{C0}{1}{0}

  \node[anchor=east] at ($(C2)+(-0.5*\w,-1.5*\h)$) {$5=$};
  \digit{C2}{1}{1}
  \digit{C1}{0}{1}
  \digit{C0}{1}{1}


  \draw[very thick] ($(C2)+(-1.8*\w,-2*\h)$) -- ($(C0)+(\w,-2*\h)$);

  \node[anchor=east] at ($(C2)+(-0.5*\w,-2.5*\h)$) {$6=$};
  \digit{C2}{1}{2}
  \digit{C1}{1}{2}
  \digit{C0}{0}{2}

  \node[anchor=east] at ($(C2)+(-\w,-1*\h)$) {$\nsum$};
\end{tikzpicture}
\caption{Nim-sum calculation: $3\oplus 5=6$ since $011\oplus 101=110$ in binary expression}
\label{fig:NimSum}
\end{figure}

Bouton's theorem \cite{bouton1901nim} claims that the winning strategy of nim is to make move so that the resulting state $(a_1, \dots, a_n)$ has $0$ as its nim-sum $a_1 \nsum \cdots \nsum a_n=0$. One can check that the player $A$ conduct the winning strategy in \Cref{fig:NimPlay}. Later, Nim-sum is proven to be important not only for Nim, but also for general combinatorial game theory. (See textbooks on combinatorial game theory, including \cite{siegel2013combinatorial}.)


Once you know the strategy, it is not very difficult to prove that it really is a winning strategy. However, when one shows it, people often react in a similar way: \dq{I understand that it works, but \demph{why does it work with this weird operation called nim-sum?}}
The initial motivation for this paper was to answer this question concerning a conceptual origin of the nim-sum, and in fact, we provide a categorical characterization of it (\Cref{prop:CategoricalCharacterizationOfNim-Sum}). Although we admit that our characterization is not yet fully satisfying to call a \dq{conceptual understanding} (cf. \Cref{ssec:missingPiece}), we expect that our framework serves as a theoretical foundation for further research\footnote{A follow-up paper, which is a joint work with Ryo Suzuki, will be uploaded later.}.

\subsection{Recursions in game theory and coalgebra theory}
In order to obtain a categorical understanding of the nim-sum, we first need to seek a natural category of impartial games. While there are already several categories of games known in category theory community, represented by Joyal's one \cite{joyal1977remarques}, our approach is a categorical theory of recursions (see also \Cref{sssec:relatedWorks}).

Combinatorial game theory is inherently recursive (or inductive). Most of the important notions in game theory, including Grundy numbers, Conway addition, birthday, outcome, and even the notion of a game itself(!), are usually defined in a recursive (or inductive) way (see \cite{siegel2013combinatorial}). In the appendix of his book \textit{On numbers and games}
\cite[][]{conway2000numbers}, John H. Conway says that
\begin{quote}
    [...] all that is needed to justify the induction is the principle:
    \dq{If $P$ is some proposition that holds for $x$ whenever it holds for all $x^L$ and $x^R$, then $P$ holds universally.} \cite[Appendix to Part Zero,][]{conway2000numbers}
\end{quote}
in the context of the foundations of mathematics.

In category theory, this kind of recursion scheme is addressed by \demph{recursive coalgebras}, which we will recall in \Cref{ssec:PreliminariesOnCoalgebraicMethod}. The notion of recursive coalgebra first appears in \cite[Section 6,][]{osius1974categorical} in the context of categorical (or topos-theoretic) set theory, and then was generalized and developed in \cite{taylor1999practical}. In this paper, we call recursive $\Pf$-coalgebras \textit{games} (\Cref{thm:main1:GamesAsRecursiveCoalgebras}), and develop a categorical theory of games (cf. \Cref{tab:coalgebra_game_theory_comparison}). 


\subsection{Our contribution}
We introduce the category of games $\Gs$ in \Cref{sec:GamesAsRecursiveCoalgebras} as the category of recursive $\Pf$-coalgebras and reinterpret impartial game theory from that point of view (\Cref{tab:coalgebra_game_theory_comparison}). For example, we interpret the notion of game values, such as Grundy numbers, remoteness, outcomes, and so on, in terms of hylomorphisms. To the best of the author’s knowledge, this is the first paper that interprets game theory explicitly using recursive coalgebras (see also related papers in \Cref{sssec:relatedWorks}).

Then, in \Cref{sec:NimSumTypeTheoremSchema}, we introduce the notion of \demph{Bouton monoid} (\Cref{def:BoutonMonoidAsUniversal}), which generalizes the nim-sum (\Cref{prop:CategoricalCharacterizationOfNim-Sum}). Our definition of the Bouton monoid is based on its universality; the universal monoid that decomposes a game into smaller games and synthesizes the game values from the smaller ones to the original.
Our main theorem (\Cref{thm:existenceTheoremOfBoutonMonoid}) proves that, for any decomposition rule ($=$ monoidal structure) and any target game value (with some mild conditions), there exists a Bouton monoid.

As we believe that this is a fruitful way to pursue,
in the last section (\Cref{sec:OpenQuestions}), we list some open questions along with some new propositions.

For category theorists, in \Cref{app:CategoricalPropertiesOfGames}, we observe that the category $\Gs$ has many pleasant properties; $\Gs$ is a locally finitely presentable symmetric monoidal closed category comonadic over $\Set$, admitting an (Epi,Mono)-orthogonal factorization system and a subobject classifier!


\subsubsection{Related works}\label{sssec:relatedWorks}
There are already studies of related categories. For example, \cite{honsell2009conway} utilizes a terminal coalgebra to study \textit{hypergames}. They consider (hyper)games as elements of the (class size) terminal coalgebra. 

A recent study \cite{bavsic2024categories}, which is indepent of our study, is from a more game-theoretic community and closer to our approach. While they do not use the coalgebra theory explicitly, their approach is very recursive-coalgebra-theoretic, and their categories are quite similar to our category.
In fact, what they call the category of \textit{rulegraphs} $\mathbf{RGph}$ is equivalent to the category of recursive coalgebras $\RecCoalg{\Pow}$ for the covariant powerset functor $\Pow \colon \Set \to \Set$, which is the prototypical example of the theory of recursive coalgebras. Furthermore, the notion of hylomorphism implicitly appears in \cite[Section 4]{bavsic2024categories}, which we will explain in terms of recursive coalgerbras in \Cref{ssec:GameValuesAsPfAlgebras}.

While our category of games $\Gs \simeq \RecCoalg{\Pf}$ is a full subcategory of theirs $\mathbf{RGph}\simeq \RecCoalg{\Pow}$, their categorical properties are different. For example, the content of \Cref{sec:NimSumTypeTheoremSchema} needs the terminal object, which does not exists in $\RecCoalg{\Pow}$ but exists in $\Gs$.
Our category $\Gs$ is locally finitely presentable, and hence equivalent to an category of models of an essentially algebraic theory. This might contribute to their idea of the analogy between the category of games and the category of algebras, emphasized both in \cite{bavsic2024categories} and \cite{baltushkin2025isomorphism}.



The category of $\Pow$-coalgebra has been studied also in the context of modal logic. In particular, a recent work \cite{de2024profiniteness} studies the category of locally finite Kripke frames, which subsumes our category $\Gs$ as its full subcategory. Some of the content of \Cref{app:CategoricalPropertiesOfGames} have already appeared in \cite{de2024profiniteness}. See \Cref{rem:KripkeFrame} for more details.


\textbf{Acknowledgment}
The author would like to thank his supervisor, Ryu Hasegawa, for helpful discussions and suggestions, in particular for suggesting the use of hylomorphisms in this research.
His appreciation extends to Tomoaki Abuku, Koki Suetsugu, Paul Taylor, Takeshi Tsukada, Kazuyuki Asada, Syoei Suzuki, Ryo Suzuki, Kyosuke Higashida, Ivan Di Liberti, and Keisuke Hoshino for their valuable discussions, as well as Math Space topos and Tokyo Kanda Laboratory (NII) for offering him a place for discussion.
The content of this paper was first announced at \href{https://sites.google.com/view/jcgtw/%E7%A0%94%E7%A9%B6%E9%9B%86%E4%BC%9A?#h.avbqzhxax0hj}{Japan Combinatorial Game Theory Mini-Workshops with Urban Larsson} in May 2023. The author would like to thank the organizers.




The author was supported by JSPS KAKENHI Grant Number JP24KJ0837 and FoPM, WINGS Program, the University of Tokyo. The research was supported by ROIS NII Open Collaborative Research 2025-251M-22791.



\section{Games as graphs}\label{sec:GamesasGraphs}
In this section, for those who are not familiar with combinatorial game theory, we recall the game-theoretic notions and phenomena that will be reinterpreted and generalized in the following two sections. For more details, see the standard textbooks on combinatorial game theory, including \cite{siegel2013combinatorial}.

\subsection{Games and Outcomes}
Since the notion of a game is very fundamental, it has been studied from many perspectives, and many different mathematical formulations have been proposed.
In this paper, we begin our discussion with a graph-theoretic\footnote{Actually, this is not quite graph-theoretic in the sense of \Cref{rem:NotGraphStructureAndProperties}} formulation of \textit{impartial combinatorial games}. Other kinds of games will be mentioned in \Cref{rem:HasColimitsAndGameWithStartingPoint}, \Cref{ssec:PartisanGames}, and \Cref{sssec:OtherFunctors}.
The idea of the formulation is quite simple; a vertex is a state of the game, and an edge is a possible move. 
\begin{definition}[(impartial) games]\label{def:game}
    A \demph{game} $\X$ is a pair $\X = (X, \rel)$ of a (possibly infinite) set $X$ and a relation ${\rel}\subset X \times X$ that satisfies two finiteness conditions:
    \begin{enumerate}
        \item (finite options) For any $x \in X$, the number of options $\# \{x' \in X \mid x\rel x'\}$ is finite.
        \item (finite time) There is no infinite path. $x_0 \rel x_1 \rel x_2 \rel \dots$
    \end{enumerate}
\end{definition}
For example, \Cref{fig:FiniteGame} and \Cref{fig:InfiniteGame} are games, but \Cref{fig:LoopNonGame}, \Cref{fig:InfiniteNonGame}, and \Cref{fig:InfiniteOptionNonGame} are not games.

\begin{figure}[ht]
\centering
\begin{tikzpicture}[>=Latex, thick]

  \node[circle, fill=black, inner sep=3pt, label=above:{}] (A) at (0,0) {};
  \node[circle, fill=black, inner sep=3pt, label=above:{}] (B) at (2,0) {};
  \node[circle, fill=black, inner sep=3pt, label=above:{}] (C) at (4,0) {};

  \node[circle, fill=black, inner sep=3pt, label=below:{}] (D) at (0.5,-1.2) {};
  \node[circle, fill=black, inner sep=3pt, label=below:{}] (E) at (2,-1.2) {};
  \node[circle, fill=black, inner sep=3pt, label=below:{}] (F) at (3.5,-1.2) {};

  \node[circle, fill=black, inner sep=3pt, label=below:{}] (G) at (1,-2.4) {};
  \node[circle, fill=black, inner sep=3pt, label=below:{}] (H) at (3,-2.4) {};

  \node[circle, fill=black, inner sep=3pt, label=below:{}] (T1) at (0.5,-3.6) {};
  \node[circle, fill=black, inner sep=3pt, label=below:{}] (T2) at (3.5,-3.6) {};

  \draw[->] (A) -- (D);
  \draw[->] (A) -- (E);
  \draw[->] (B) -- (E);
  \draw[->] (B) -- (F);
  \draw[->] (C) -- (F);

  \draw[->] (D) -- (G);
  \draw[->] (E) -- (G);
  \draw[->] (E) -- (H);
  \draw[->] (F) -- (H);

  \draw[->] (G) -- (T1);
  \draw[->] (G) -- (T2);
  \draw[->] (H) -- (T2);

  \draw[->] (A) to[bend right] (G);
  \draw[->] (C) to[bend left] (T2);

\end{tikzpicture}
\caption{An example of a finite game}
\label{fig:FiniteGame}
\end{figure}
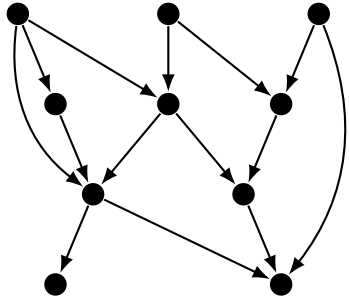

\begin{figure}[ht]
\centering
\begin{tikzpicture}[>=Latex, thick]

  \node[circle, fill=black, inner sep=3pt, opacity=0] (Linf) at (-6.3,0) {};
  \node at (-6.5,0) {$\cdots$};

  \node[circle, fill=black, inner sep=3pt] (L2) at (-4,0) {};
  \node[circle, fill=black, inner sep=3pt] (L1) at (-2,0) {};
  \node[circle, fill=black, inner sep=3pt] (O)  at ( 0,0) {};
  \node[circle, fill=black, inner sep=3pt] (R1) at ( 2,0) {};

  \draw[->] (Linf) -- (L2);
  \draw[->] (L2) -- (L1);
  \draw[->] (L1) -- (O);
  \draw[->] (O)  -- (R1);

\end{tikzpicture}
\caption{An example of an infinite game}
\label{fig:InfiniteGame}
\end{figure}
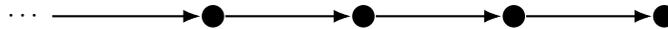

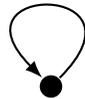
\begin{figure}[ht]
\centering
\begin{tikzpicture}[>=Latex, thick]

  \node[circle, fill=black, inner sep=3pt, label=above:{}] (v) at (0,0) {};

  \draw[->] (v) to[out=45, in=135, looseness=20] (v);


\end{tikzpicture}
\caption{An example of non-game with infinite path}
\label{fig:LoopNonGame}
\end{figure}

\begin{figure}[ht]
\centering
\begin{tikzpicture}[>=Latex, thick]

  \node[circle, fill=black, inner sep=3pt, opacity=0] (L3) at (-6,0) {};
  \node[circle, fill=black, inner sep=3pt, opacity=0] (R3) at ( 6,0) {};

  \node at (-6.5,0) {$\cdots$};
  \node at ( 6.5,0) {$\cdots$};

  \node[circle, fill=black, inner sep=3pt] (L2) at (-4,0) {};
  \node[circle, fill=black, inner sep=3pt] (L1) at (-2,0) {};
  \node[circle, fill=black, inner sep=3pt] (O)  at ( 0,0) {};
  \node[circle, fill=black, inner sep=3pt] (R1) at ( 2,0) {};
  \node[circle, fill=black, inner sep=3pt] (R2) at ( 4,0) {};

  \draw[->] (L3) -- (L2);
  \draw[->] (L2) -- (L1);
  \draw[->] (L1) -- (O);
  \draw[->] (O)  -- (R1);
  \draw[->] (R1) -- (R2);
  \draw[->] (R2) -- (R3);

\end{tikzpicture}
\caption{Another example of non-game with infinite path}
\label{fig:InfiniteNonGame}
\end{figure}
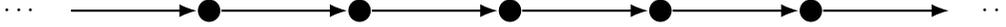

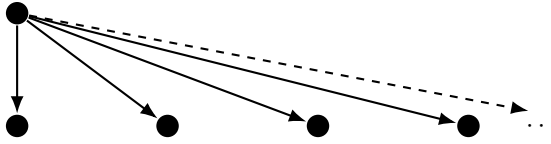
\begin{figure}[ht]
\centering
\begin{tikzpicture}[>=Latex, thick]

  \node[circle, fill=black, inner sep=3pt, label=above:{}] (U) at (0,1.5) {};

  \node[circle, fill=black, inner sep=3pt, label=below:{}] (V1) at (0,0) {};
  \node[circle, fill=black, inner sep=3pt, label=below:{}] (V2) at (2,0) {};
  \node[circle, fill=black, inner sep=3pt, label=below:{}] (V3) at (4,0) {};
  \node[circle, fill=black, inner sep=3pt, label=below:{}] (V4) at (6,0) {};
  \node at (7,0) {$\cdots$};

  \draw[->] (U) -- (V1);
  \draw[->] (U) -- (V2);
  \draw[->] (U) -- (V3);
  \draw[->] (U) -- (V4);

  \draw[->, dashed] (U) -- (6.8,0.2);

\end{tikzpicture}
\caption{An example of non-game with infinite options}
\label{fig:InfiniteOptionNonGame}
\end{figure}

The ultimate goal of combinatorial game theory is to win a given game.
To win a game is almost the same thing as knowing \demph{outcome}, \dq{$N$-states} (Next-player-winning states) or \dq{$P$-states} (Previous-player-winning states), defined as follows. 
\begin{definition}[Outcome]\label{def:outcome}
    For a game $\X=(X, \rel)$ and its state $x\in X$,
    the \demph{outcome} of $x$ is defined by
    \[
    \oc_{\X}(x)\coloneqq
    \begin{cases}
        N& \textrm{(if the next player wins from the state $x$)}\\
        P& \textrm{(if the previous player wins from the state $x$),}
    \end{cases}
    \]
    or equivalently, it is recursively defined by
    \[
    \oc_{\X}(x)\coloneqq
    \begin{cases}
        N& \textrm{(if there exists $x\rel x'$ with $\oc_{\X}(x')=P$)}\\
        P& \textrm{(otherwise)}
    \end{cases}
    \]
\end{definition}
The latter recursive definition is not circuler, thanks to the \dq{finite time} condition.

\begin{remark}[How can we win with outcome?]
    For those who are not familiar with combinatorial game theory, let us clarify the connection between an actual ways to win a game $\X=(X, {\rel})$ and the outcome function $\oc\colon X \to \{N,P\}$. Suppose it is your turn and the current position is an $N$-state. Then \demph{the winning strategy is simply to always move to a $P$-state.} By the definition of $P$-states, your opponent cannot reply with another $P$-state. The opponent must either move to an $N$-state, or lose immediately. In the latter case, by the definition of $N$-states, you can move to a $P$-state again. Due to the finite time condition, this procedure must terminate, and therefore you are guaranteed to win.
\end{remark}

\begin{remark}[Why do we assume the \dq{finite options} condition?]\label{rem:finiteOptions}
    While we do not assume the finiteness of the underlying set, we include the \dq{finite options} condition in \Cref{def:game}. This kind of \dq{local finiteness} condition is a difference with \cite{bavsic2024categories} and a similarity with \cite{de2024profiniteness}. We assume it not only because many games of interest satisfy the condition, but also we need it in the following sections. Let us explain some (of many) reasons why we need it. One game theoretic reason is that it is necessary to define the Grundy number while dealing with games with infinite states. (More precisely, we need it to keep Grundy number to be finite ordinal.) A categorical reason is that the considered functor $\Pf\colon \Set \to \Set$ is finitary (\cite[Example 2.5, Example 3.18]{adamek2007recursive}) and hence our category of games admits a lot of pleasant properties (\Cref{app:CategoricalPropertiesOfGames}). In particular, thanks to the finite option condition, we have the terminal game (\Cref{prop:TerminalGameAndInitialAlgebraAreHereditarilyFiniteSets}), which plays the central role in our characterization of nim-sum.
\end{remark}

\begin{example}[Winning strategy of the subtraction nim]\label{exmp:SubtractionNim}Let us give a famous example of the outcome function.
    A {subtraction nim} (for $S=\{1,2,3\}$) is $(\N, {\rel})$, where
    \[
    n\rel m \iff n=m+1, m+2, \textrm{or }m+3.
    \]
    We can (recursively) prove that the outcome function is given by
    \[
    \oc(n)=
    \begin{cases}
        N &(n\not\equiv 0 \mod 4)\\
        P &(n\equiv 0 \mod 4).
    \end{cases}
    \]
    So the winning strategy is to move to the multiples of $4$.
\end{example}

Theoretically, the most important game might be the following game, which we will call the binary exponent nim, or later, the terminal game (\Cref{rem:AckermanInterpretationAndTerminalGame}).
\begin{example}[Binary exponent nim, or the terminal game]\label{exmp:BinaryExponentNimOrTerminalGame}
    The underlying set of \demph{the binary exponent nim} is $\N$, and the relation $n \rel m$ is defined by
    \[n \rel m \iff 2^m \textrm{ appears in the binary expansion of }n.\]
    For example, when $n=10000$, since
    \[
    n=10000=2^{13}+2^{10}+2^{9}+2^{8}+2^{4},
    \]
    there are $5$ possible moves, namely $10000\rel4,8,9,10,13$.
    The state $n=10000$ is $N$-state, and the winning way it to move to $8$ or $10$ (\Cref{fig:BinaryExponentNim}).
    \end{example}

\begin{figure}[ht]
\centering
\begin{tikzpicture}[>=Latex, thick, node distance=12mm and 14mm]

  \node[circle, fill=red, inner sep=3pt, label=above:{10000}] (N10000) at (0,1) {};

  \node[circle, fill=red, inner sep=3pt, label=left:{13}] (N13) at (-4,-1) {};
  \node[circle, fill=blue, inner sep=3pt, label=left:{10}] (N10) at (-2,-1) {};
  \node[circle, fill=red, inner sep=3pt, label=left:{9}]  (N9)  at ( 0,-1) {};
  \node[circle, fill=blue, inner sep=3pt, label=right:{8}]  (N8)  at ( 2,-1) {};
  \node[circle, fill=red, inner sep=3pt, label=right:{4}]  (N4)  at ( 4,-1) {};

  \node[circle, fill=red, inner sep=3pt, label=right:{3}]  (N3)  at (0,-3) {};
  \node[circle, fill=blue, inner sep=3pt, label=right:{2}]  (N2)  at (0,-5) {};

  \node[circle, fill=red, inner sep=3pt, label=right:{1}]  (N1)  at (0,-7) {};
  \node[circle, fill=blue, inner sep=3pt, label=right:{0}]  (N0)  at (0,-9) {};

  \draw[->] (N10000) -- (N13);
  \draw[->] (N10000) -- (N10);
  \draw[->] (N10000) -- (N9);
  \draw[->] (N10000) -- (N8);
  \draw[->] (N10000) -- (N4);

  \draw[->] (N13) -- (N3);
  \draw[->] (N13) -- (N2);
  \draw[->] (N13) to[bend right] (N0);

  \draw[->] (N10) -- (N3);
  \draw[->] (N10) to[bend right] (N1);

  \draw[->] (N9) -- (N3);
  \draw[->] (N9) to[bend left] (N0);

  \draw[->] (N8) -- (N3);

  \draw[->] (N4) -- (N2);

  \draw[->] (N3) to[bend left] (N1);
  \draw[->] (N3) to[bend right] (N0);

  \draw[->] (N2) -- (N1);

  \draw[->] (N1) -- (N0);

\end{tikzpicture}
\caption{The binary exponent nim, below $10000$, with {\color{blue} $P$-states} and {\color{red} $N$-states}.}
\label{fig:BinaryExponentNim}
\end{figure}
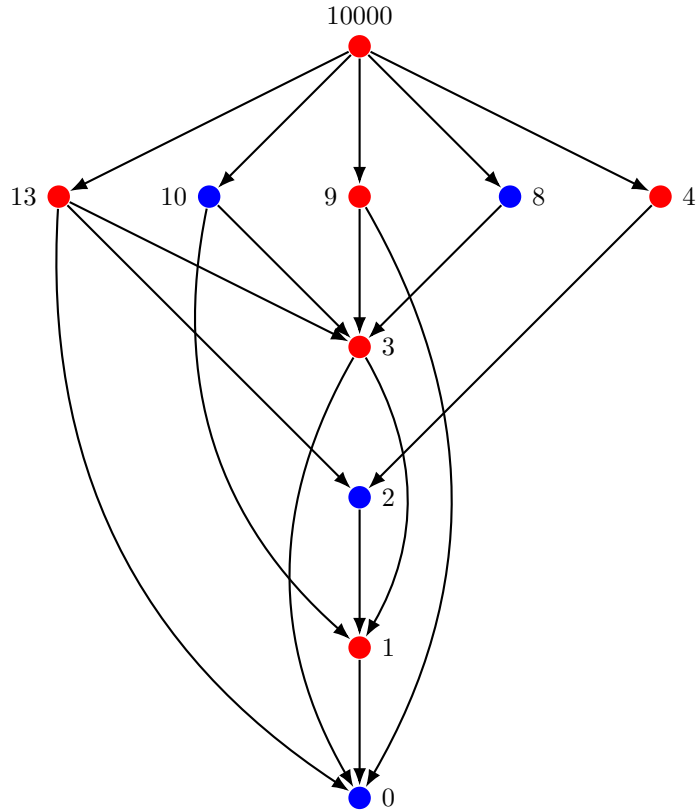

\begin{example}[\dq{Effeuiller la marguerite}]\label{exmp:EffeuillerLaMarguerite}
    Let us consider another game whose underlying set is $\N$. This time, we define $n\rel m$ by \[n\rel m \iff n=m+1.\] In each phase, players have at most one option.
    Obviously, the outcome of a state $n$ of this boring game is determined only by the parity of $n$; a state $n$ is a $P$-state if and only if $n$ is even.
    We will write $\ElM$ for this game since this game is conventionally called \dq{effeuiller la marguerite.}
    Althogh this game $\ElM=(\N, -1)$ is practically quite boring, it sometimes plays a theoretical role (\Cref{exmp:RemotenessAsBoutonMonoid}, \Cref{exmpl:remotenessAndEffeuillerLaMarguerite}).
\end{example}

We conclude this subsection by expressing Bouton's theorem in our formulation (\Cref{def:game}) as follows.
\begin{example}[Nim]\label{exmp:nimAsGraph}
    In our formulation, the $n$-heap nim $\Nim{n}$ is $(\N^{n}, \rel)$, where $(a_1, \dots a_n)\rel (b_1, \dots b_n)$ is defined by
    \[
    (a_1, \dots a_n)\rel (b_1, \dots b_n) \iff \textrm{there exists $1\leq i \leq n$ such that $a_i > b_i$ and for any $j \neq i$, $a_j = b_j$}
    \]
\end{example}
\begin{theorem}[Bouton's theorem \cite{bouton1901nim}]\label{thm:Bouton}
    A state $(a_1,\dots , a_n)$ of the $n$-heap nim $\Nim{n}$ is $P$-state if and only if $a_1 \nsum \cdots \nsum a_n=0$.  
\end{theorem}



\subsection{Divide difficulties with the Conway addition and Grundy number}
In the last subsection, we have seen that once you know the outcome of a given game, you know the winning strategy. But how can we effectively calculate the outcome? (For example, how can we know that the outcome of a nim state $(a_1, \dots , a_n)$ is $P$ if and only if $a_1 \nsum \dots \nsum a_n =0$?)

One promising strategy is to divide the difficulty into smaller parts!
What we will do in this subsection (and the whole \Cref{sec:NimSumTypeTheoremSchema}) is to decompose a game into smaller parts and synthesize the properties of smaller games:
\begin{itemize}
    \item We want to know the outcome of a game $\X$.
    \item We devide the game $\X$ into a \demph{Conway sum} (\Cref{def:ConwayAddition}) of smaller games $\X = \Y + \Z$.
    \item We calculate the \demph{Grundy numbers} (\Cref{def:GrundyNumber}) of the games $\Y,\Z$.
    \item We calculate the {Grundy numbers} of $\X$ from those of $\Y,\Z$ (\Cref{thm:GeneralizedBoutonTheoremNimSumRule}).
    \item We calculate the outcome of $\X$ from the Grundy number of $\X$ (\Cref{prop:GrundynumberIsMoreInformativeThanOutcome}).
\end{itemize}
Let us clarify again that all the content in this subsection is well-known in combinatorial game theory. For more details, see standard textbooks including \cite{siegel2013combinatorial}.

\invmemo{In the previous subsection, we have seen the nim-sum $\nsum$ provides a winning strategy of the nim game. But which part of the nim game rule made us to consider nim-sum? The answer is classically known in combinatorial game theory, as the generalized nim-sum rule for \demph{Conway addition} and \demph{Grundy number}. (See \cite{siegel2013combinatorial} for more detail.) \invmemo{Can't catch the meaning}}

First, we recall the notion of the \demph{Conway addition} of games, which is visualized in \Cref{fig:ConwayAddition}.
\begin{definition}[Conway addition of games]\label{def:ConwayAddition}
For two games $\X = (X,\rel_{\X}), \Y = (Y,\rel_{\Y})$, their \demph{Conway sum} (or just sum) $\X + \Y$ is the game $(X\times Y, \rel_{\X + \Y})$, where $(x,y) \rel_{\X + \Y} (x',y')$ if and only if 
$(x \rel_{\X} x' \land y=y')$ or $(x=x' \land y \rel_{\Y}y')$.
\end{definition}

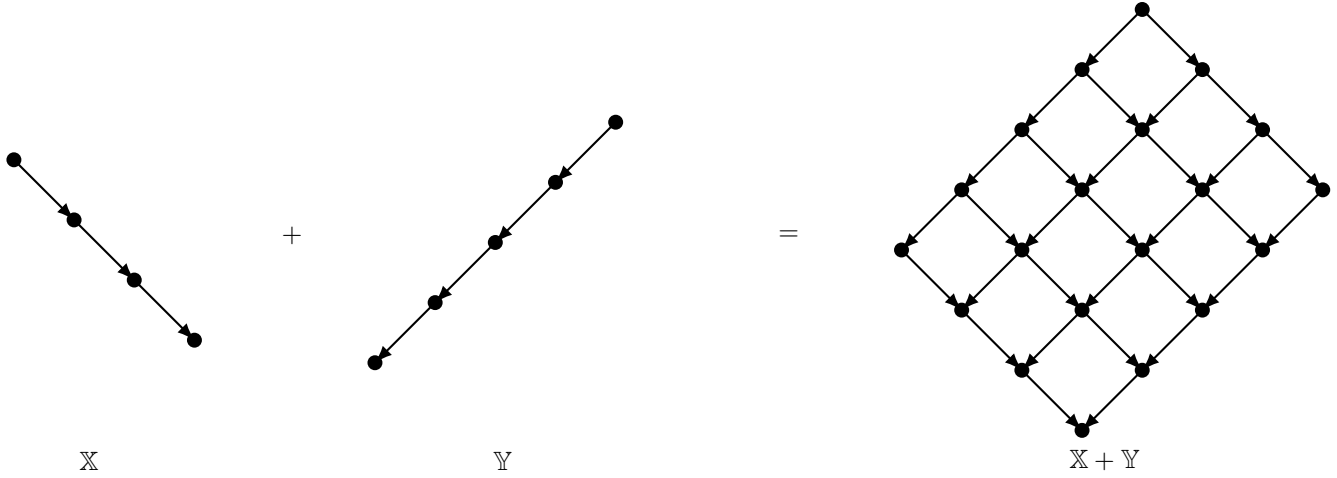
\begin{figure}[ht]
\centering
\begin{tikzpicture}[>=Latex, thick]

  \newcommand{\slantedgrid}[5]{%
    \def\dx{0.8}   
    \def\dy{-0.8}  
    \def\ex{-0.8}  
    \def\ey{-0.8}  
    \foreach \i in {0,...,#1}{%
      \foreach \j in {0,...,#2}{%
        \coordinate (#5-\i-\j) at (#3+\i*\dx+\j*\ex, #4+\i*\dy+\j*\ey);
        \node[circle, fill=black, inner sep=2pt] at (#5-\i-\j) {};
      }%
    }%
    \ifnum#1>0
      \pgfmathtruncatemacro{\imax}{#1-1}%
      \foreach \i in {0,...,\imax}{%
        \foreach \j in {0,...,#2}{%
          \draw[->] (#5-\i-\j) -- (#5-\the\numexpr\i+1\relax-\j);
        }%
      }%
    \fi
    \ifnum#2>0
      \pgfmathtruncatemacro{\jmax}{#2-1}%
      \foreach \i in {0,...,#1}{%
        \foreach \j in {0,...,\jmax}{%
          \draw[->] (#5-\i-\j) -- (#5-\i-\the\numexpr\j+1\relax);
        }%
      }%
    \fi
  }

  \slantedgrid{3}{0}{-8}{0}{X}
  \node at (-7,-4) {$\X$};

  \node at (-4.3,-1) {$+$};

  \slantedgrid{0}{4}{0}{0.5}{Y}
  \node at (-1.5,-4) {$\Y$};

  \node at (2.3,-1) {$=$};

  \slantedgrid{3}{4}{7}{2}{Z}
  \node at (6.5,-4) {$\X + \Y$};

\end{tikzpicture}
\caption{An example of Conway addition.}
\label{fig:ConwayAddition}
\end{figure}

\begin{remark}[Terminology and notation: Why is it called \dq{sum}?]
    The Conway sum is conventionally called \dq{sum} and denoted by $\X + \Y$, while some readers might think \dq{product} is a better name. In fact, the corresponding monoidal structure in graph theory is called the box \dq{product} (see, for example, \cite{kapulkin2024closed}, which is also related to \Cref{ssec:ClassificationProblemOfMonoidalStructures}). But there is a decent reason to call it \dq{sum} in game theory, related to the surreal number \cite{conway2000numbers}.
\end{remark}

A typical usage of the Conway addition is to decompose a complicated game into smaller games. The prototypical example is the following decomposition of nim.
\begin{example}[Decomposition of Nim]\label{exmp:DecompositionOfNim}
    The $n$-heaps nim $\Nim{n}$ is the Conway sum of $n$-copies of ($1$-heap) nim games.\[\Nim{n} = \Nim{1} + \dots + \Nim{1}\]
\end{example}

We want to utilize the Conway addition $+$ to calculate the outcome of a state $(x,y)\in \X+ \Y$. However, even if you know the outcomes $\oc_\X(x)$ and $\oc_\Y(y)$, it is generally impossible to calculate the outcome of the sum state $\oc_{\X+ \Y}(x,y)$. Thus we need to enrich the outcome into the Grundy number. As a preparation, we recall the notion of \textit{mex}, which stands for \demph{m}inimum \demph{ex}cluded value.
\begin{definition}[mex]\label{def:mex}
    For a finite set of natural numbers $S\subset \N$, its \demph{mex}, denoted by $\mex(S)\in \N$, is the minimum natural number that does not belong to the subset $S$. In other words, mex of $S$ is the minimum element of the complement of $S$:
\[\mex(S) = \min S^{\mathrm{c}}.\]
\end{definition}
Notice that the complement $S^c$ is always non-empty since the set $S$ is assumed to be finite (cf. \Cref{rem:finiteOptions}).

\begin{definition}[Grundy number]\label{def:GrundyNumber}
    For a game $\X = (X, \rel)$ and a state $x\in X$, its \demph{Grundy number} $\G_{\X}(x)$ is recursively defined by
    \begin{equation}\label{eq:GrundyNumber}
        \G_{\X}(x)\coloneqq \mex(\{\G_{\X}(x')\mid x \rel x'\}).
    \end{equation}
\end{definition}
This recursive definition does work due to the two finiteness conditions in the definition of games (\Cref{def:game}). See \Cref{fig:RecursioveCalculationOfGrundyNumber} for an example of a recursive calculation of the Grundy numbers.
\begin{figure}[ht]
\centering
\begin{tikzpicture}[>=Latex, thick, scale=0.9]

  \tikzset{edge/.style={->, draw=black!35}}
  \newcommand{\grundypanel}[4]{%
    \begin{scope}[xshift=#1cm, yshift=#2cm]
      \def\stage{#3}%

      \node[circle, inner sep=3pt,
            fill={\ifnum\stage>0 black\else black!20\fi},
            label=left:{\ifnum\stage>0 $0$\fi}] (#4Z1) at (-2,-7.2) {};
      \node[circle, inner sep=3pt,
            fill={\ifnum\stage>0 black\else black!20\fi},
            label=left:{\ifnum\stage>0 $0$\fi}] (#4Z2) at ( 0,-7.2) {};
      \node[circle, inner sep=3pt,
            fill={\ifnum\stage>0 black\else black!20\fi},
            label=left:{\ifnum\stage>0 $0$\fi}] (#4Z3) at ( 2,-7.2) {};

      \node[circle, inner sep=3pt,
            fill={\ifnum\stage>1 black\else black!20\fi},
            label=left:{\ifnum\stage>1 $1$\fi}] (#4Y1) at (-2,-5.8) {};
      \node[circle, inner sep=3pt,
            fill={\ifnum\stage>1 black\else black!20\fi},
            label=left:{\ifnum\stage>1 $1$\fi}] (#4Y2) at ( 0,-5.8) {};
      \node[circle, inner sep=3pt,
            fill={\ifnum\stage>1 black\else black!20\fi},
            label=left:{\ifnum\stage>1 $1$\fi}] (#4Y3) at ( 2,-5.8) {};

      \node[circle, inner sep=3pt,
            fill={\ifnum\stage>2 black\else black!20\fi},
            label=left:{\ifnum\stage>2 $0$\fi}] (#4X1) at (-2,-4.4) {};
      \node[circle, inner sep=3pt,
            fill={\ifnum\stage>2 black\else black!20\fi},
            label=left:{\ifnum\stage>2 $2$\fi}] (#4X2) at ( 0,-4.4) {};
      \node[circle, inner sep=3pt,
            fill={\ifnum\stage>2 black\else black!20\fi},
            label=left:{\ifnum\stage>2 $0$\fi}] (#4X3) at ( 2,-4.4) {};

      \node[circle, inner sep=3pt,
            fill={\ifnum\stage>3 black\else black!20\fi},
            label=left:{\ifnum\stage>3 $1$\fi}] (#4W1) at (-2,-3.0) {};
      \node[circle, inner sep=3pt,
            fill={\ifnum\stage>2 black\else black!20\fi},
            label=left:{\ifnum\stage>2 $2$\fi}] (#4W3) at ( 0,-3.0) {};
      \node[circle, inner sep=3pt,
            fill={\ifnum\stage>3 black\else black!20\fi},
            label=right:{\ifnum\stage>3 $1$\fi}] (#4W2) at ( 2,-3.0) {};

      \node[circle, inner sep=3pt,
            fill={\ifnum\stage>4 black\else black!20\fi},
            label=left:{\ifnum\stage>4 $0$\fi}] (#4V1) at (-1,-1.6) {};
      \node[circle, inner sep=3pt,
            fill={\ifnum\stage>4 black\else black!20\fi},
            label=left:{\ifnum\stage>4 $3$\fi}] (#4V2) at ( 1,-1.6) {};

      \draw[edge, {\ifnum\stage>1 black\else black!20\fi}] (#4Y1) -- (#4Z1);
      \draw[edge, {\ifnum\stage>1 black\else black!20\fi}] (#4Y2) -- (#4Z2); \draw[edge, {\ifnum\stage>1 black\else black!20\fi}] (#4Y2) -- (#4Z3);
      \draw[edge, {\ifnum\stage>1 black\else black!20\fi}] (#4Y3) -- (#4Z3);

      \draw[edge, {\ifnum\stage>2 black\else black!20\fi}] (#4X1) -- (#4Y1); \draw[edge, {\ifnum\stage>2 black\else black!20\fi}] (#4X1) -- (#4Y2);
      \draw[edge, {\ifnum\stage>2 black\else black!20\fi}] (#4X2) -- (#4Y2); \draw[edge, {\ifnum\stage>2 black\else black!20\fi}] (#4X2) -- (#4Z1);
      \draw[edge, {\ifnum\stage>2 black\else black!20\fi}] (#4X3) -- (#4Y3);

      \draw[edge, {\ifnum\stage>3 black\else black!20\fi}] (#4W1) -- (#4X1); \draw[edge, {\ifnum\stage>3 black\else black!20\fi}] (#4W1) -- (#4X2);
      \draw[edge, {\ifnum\stage>3 black\else black!20\fi}] (#4W2) -- (#4X2); \draw[edge, {\ifnum\stage>3 black\else black!20\fi}] (#4W2) -- (#4X3);
      \draw[edge, {\ifnum\stage>2 black\else black!20\fi}] (#4W3) -- (#4Y1); \draw[edge, {\ifnum\stage>2 black\else black!20\fi}] (#4W3) -- (#4Z3);

      \draw[edge, {\ifnum\stage>4 black\else black!20\fi}] (#4V1) -- (#4W1); \draw[edge, {\ifnum\stage>4 black\else black!20\fi}] (#4V1) -- (#4W2); \draw[edge, {\ifnum\stage>4 black\else black!20\fi}] (#4V1) -- (#4W3);
      \draw[edge, {\ifnum\stage>4 black\else black!20\fi}] (#4V2) -- (#4W2); \draw[edge, {\ifnum\stage>4 black\else black!20\fi}] (#4V2) -- (#4X2); \draw[edge, {\ifnum\stage>4 black\else black!20\fi}] (#4V2) -- (#4X3);

      \node at (0,-8.6) {\small Step~#3};
    \end{scope}%
  }

  \grundypanel{0.0}{0.0}{1}{A}
  \grundypanel{6.0}{0.0}{2}{B}
  \grundypanel{12.0}{0.0}{3}{C}
  \grundypanel{0.0}{-9.0}{4}{D}
  \grundypanel{6.0}{-9.0}{5}{E}

  \begin{scope}[xshift=12cm, yshift=-9cm]
      \node[circle, inner sep=3pt,
            fill=blue] (GZ1) at (-2,-7.2) {};
      \node[circle, inner sep=3pt,
            fill=blue] (GZ2) at ( 0,-7.2) {};
      \node[circle, inner sep=3pt,
            fill=blue] (GZ3) at ( 2,-7.2) {};

      \node[circle, inner sep=3pt,
            fill=red] (GY1) at (-2,-5.8) {};
      \node[circle, inner sep=3pt,
            fill=red] (GY2) at ( 0,-5.8) {};
      \node[circle, inner sep=3pt,
            fill=red] (GY3) at ( 2,-5.8) {};

      \node[circle, inner sep=3pt,
            fill=blue] (GX1) at (-2,-4.4) {};
      \node[circle, inner sep=3pt,
            fill=red] (GX2) at ( 0,-4.4) {};
      \node[circle, inner sep=3pt,
            fill=blue] (GX3) at ( 2,-4.4) {};

      \node[circle, inner sep=3pt,
            fill=red] (GW1) at (-2,-3.0) {};
      \node[circle, inner sep=3pt,
            fill=red] (GW3) at ( 0,-3.0) {};
      \node[circle, inner sep=3pt,
            fill=red] (GW2) at ( 2,-3.0) {};

      \node[circle, inner sep=3pt,
            fill=blue] (GV1) at (-1,-1.6) {};
      \node[circle, inner sep=3pt,
            fill=red] (GV2) at ( 1,-1.6) {};

      \draw[edge, black] (GY1) -- (GZ1);
      \draw[edge, black] (GY2) -- (GZ2); \draw[edge, black] (GY2) -- (GZ3);
      \draw[edge, black] (GY3) -- (GZ3);

      \draw[edge, black] (GX1) -- (GY1); \draw[edge, black] (GX1) -- (GY2);
      \draw[edge, black] (GX2) -- (GY2); \draw[edge, black] (GX2) -- (GZ1);
      \draw[edge, black] (GX3) -- (GY3);

      \draw[edge, black] (GW1) -- (GX1); \draw[edge, black] (GW1) -- (GX2);
      \draw[edge, black] (GW2) -- (GX2); \draw[edge, black] (GW2) -- (GX3);
      \draw[edge, black] (GW3) -- (GY1); \draw[edge, black] (GW3) -- (GZ3);

      \draw[edge, black] (GV1) -- (GW1); \draw[edge, black] (GV1) -- (GW2); \draw[edge, black] (GV1) -- (GW3);
      \draw[edge, black] (GV2) -- (GW2); \draw[edge, black] (GV2) -- (GX2); \draw[edge, black] (GV2) -- (GX3);

      \node at (0,-8.6) {{\color{blue} $P$-states} and {\color{red} $N$-states}};
    \end{scope}%

\end{tikzpicture}
\caption{Recursive calculation of \demph{Grundy numbers} from bottom to top, and the outcome}
\label{fig:RecursioveCalculationOfGrundyNumber}
\end{figure}
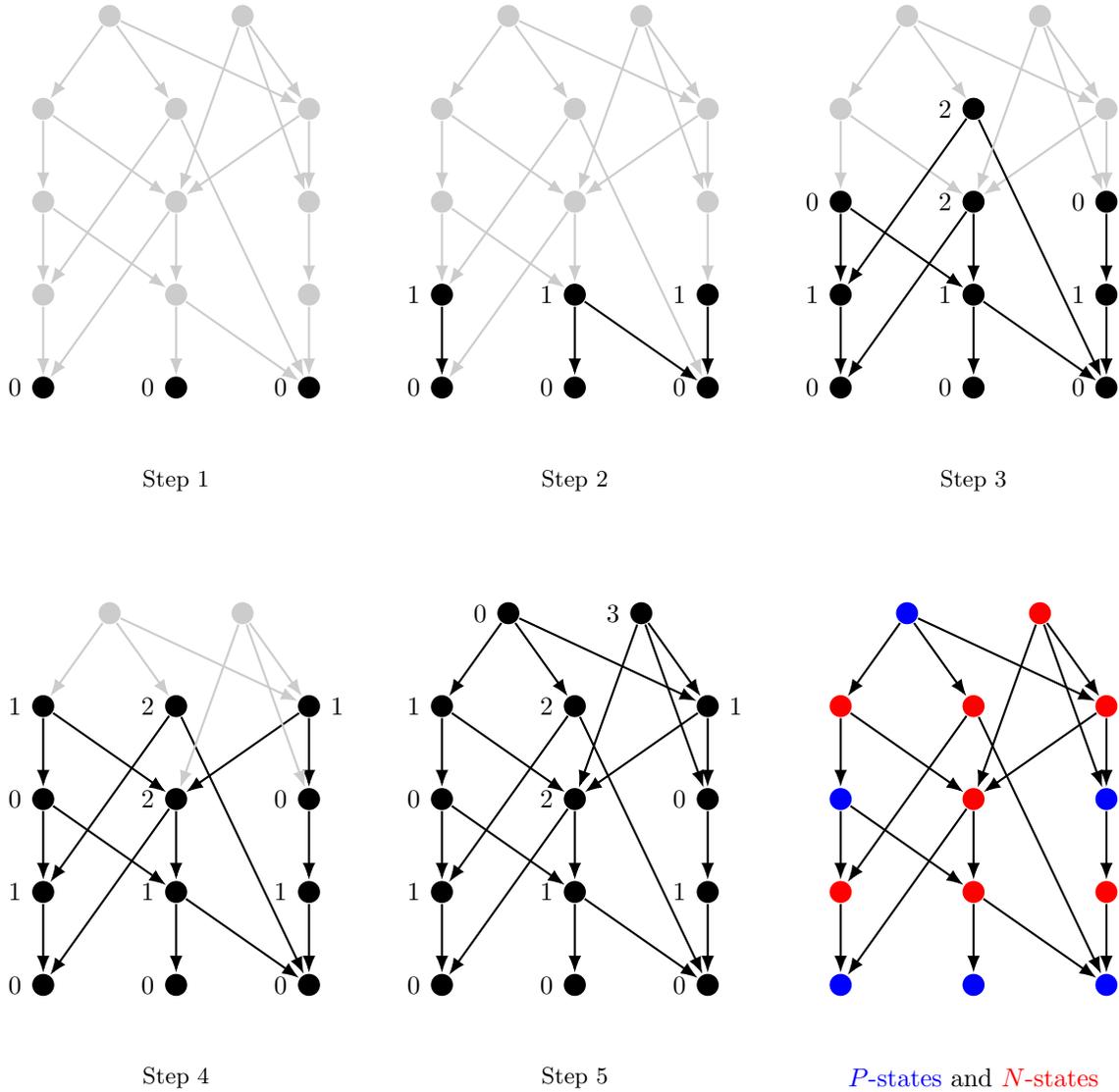
The importance of the Grundy number is due to the following proposition.
\begin{proposition}[the Grundy number is more informative than the outcome]\label{prop:GrundynumberIsMoreInformativeThanOutcome}
    For any game $\X=(X, {\rel})$ and any state $x\in X$, its outcome is $P$ if and only if its Grundy number $\G_{\X}(x)$ is $0$:
    \[
    \oc_{\X}(x)=P \iff \G_{\X}(x)=0.
    \]
\end{proposition}
\begin{proof}
One can easily prove this by induction, although we will see more conceptual proof later (\Cref{prop:AlgebraHomPreservesGameValue}). 
\end{proof}

The reason why we enrich the ouctome to Grundy number is that the Grundy number is compatible with the Conway addition (\Cref{thm:GeneralizedBoutonTheoremNimSumRule})!
Let us recall nim-sum again (\Cref{def:NimSum}).

\begin{definition}[Nim-sum]\label{def:NimSum2}
    The \demph{nim-sum} is the abelian group structure on $\N$, induced by the bijection $\N \to \bigoplus_{k=0}^{\infty} \Z/2\Z$ given by the binary expansion. In other words, the nim-sum is the digit-wise exclusive disjunction (xor) of the binary expansion.
\end{definition}
For example, $7\nsum 5 = (111)_{2} \nsum (101)_{2} = (010)_{2} = 2$.

\begin{theorem}[Nim-sum rule= generalized Bouton's theorem 
{\cite[II. Theorem4.7]{siegel2013combinatorial}}]
\label{thm:GeneralizedBoutonTheoremNimSumRule}
    For any pair of games $\X = (X,\rel_{\X}), \Y = (Y,\rel_{\Y})$ and any pair of states $x\in X$, $y\in Y$, the Grundy number of $(x,y)$ is given by the nim-sum as follows:
    \[\G_{\X+\Y}(x,y)= \G_{\X}(x)\nsum\G_{\Y}(y).\]
\end{theorem}

Combining \Cref{prop:GrundynumberIsMoreInformativeThanOutcome} and \Cref{thm:GeneralizedBoutonTheoremNimSumRule}, we can calculate $P$-states of a Conway sum and thus its winning strategy. In fact, for a given game $\X$ that can be decomposed into a Conway sum of two games $\X = \Y + \Z$, we have
\[
x=(y,z) \textrm{ is a $P$-state} {\iff} \G_{\Y+\Z}(y,z)=0 \iff \G_{\Y}(y) \nsum \G_{\Z}(z)=0.
\]
\begin{example}[Analysis of nim]\label{exmp:AnalysisOfNim}
    The classical Bouton theorem (\Cref{thm:Bouton}) is the typical example of the above observation.
    As the $n$-heap nim $\Nim{n}$ is the Conway sum of $n$-copies of the $1$-heap nim $\Nim{1}$ (\Cref{exmp:DecompositionOfNim}), in order to win $\Nim{n}$, it suffices to know $\G_{\Nim{1}}\colon \N \to \N$. By the easy induction, we can prove that $\G_{\Nim{1}}= \id_{\N}$ (cf. \Cref{app:GrundyNumberisLeftAdjointToNim}), and thus we obtain
    \[
    \G_{\Nim{n}}(a_1, \dots , a_n)= \G_{\Nim{1}}(a_1) \nsum \dots \nsum \G_{\Nim{1}}(a_n) = a_1 \nsum \dots \nsum a_n.
    \]
    Therefore, \Cref{prop:GrundynumberIsMoreInformativeThanOutcome} implies that a state $(a_1, \dots, a_n)$ is a $P$-state if and only if $a_1 \nsum \dots \nsum a_n=0$, which is exactly what the Bouton theorem states (\Cref{thm:Bouton}). 
\end{example}

\section{Games as recursive coalgebras}\label{sec:GamesAsRecursiveCoalgebras}
In this section and the next section, we will reinterpret the content in \Cref{sec:GamesasGraphs} in terms of recursive coalgebras.
Here is the spoiler:
\begin{table}[ht]
    \centering
    \begin{tabular}{|l|l|l|l|}
        \hline
        \textbf{Abstract Coalgebra Theory} &  &\textbf{Game Theoretic Notions}  &example\\ \hline
        Recursive coalgebra& \cref{ssec:GamesasrecursiveCoalgebras} &Game& $\Nim{}$ \\ \hline
        Coalgebra homomorphism&\cref{ssec:GameMorphismsAsCoalgebraMorphisms}  &Game morphism &\\ \hline
        Algebra&\cref{ssec:GameValuesAsPfAlgebras}  &recursion step of game value& mex\\ \hline
        Algebra homomorphism& \cref{ssec:GameValuesAsPfAlgebras}  &Transformation of game value& \Cref{prop:GrundynumberIsMoreInformativeThanOutcome}\\ \hline
 Coalgebra-algebra morphism& \cref{ssec:GameValuesAsPfAlgebras} & Game value& Grundy number \\\hline
 Terminal recursive coalgebra&\cref{ssec:terminalGame} & Binary exponent nim&\\\hline
 Monoidal structure& \cref{ssec:GameDecompotisionAsMonoidalStructure}& Game decomposition&Conway sum\\\hline
 Bouton monoid&\cref{ssec:SynthesizingWithBoutonMonoid} & &Nim-sum\\\hline
    \end{tabular}
    \caption{Comparison between coalgebra theory and game theory}
    \label{tab:coalgebra_game_theory_comparison}
\end{table}

Before delving into the reinterpretation part, 
we first observe a motivating phenomenon in \Cref{ssec:WhyRecursiveCoalgebra}, and recall the theory of recursive coalgebras in \Cref{ssec:PreliminariesOnCoalgebraicMethod}. 
\subsection{Why recursive coalgebras?}\label{ssec:WhyRecursiveCoalgebra}
Before explaining the categorical abstract nonsense in \Cref{ssec:PreliminariesOnCoalgebraicMethod}, let us observe one phenomenon motivating the latter contents.
The reader can skip this subsection since it does not include any theorems. However, we recommend to read it because this observation is the starting idea of this project\footnote{This observation can be found also in \cite[section 4]{bavsic2024categories}, which does not use coalgebra theory.}.

Our original motivation is to understand the nim-sum (\Cref{def:NimSum2}), and the reason why the nim-sum works is the Nim-sum rule (\Cref{thm:GeneralizedBoutonTheoremNimSumRule}) regarding the Grundy number (\Cref{def:GrundyNumber}). 
So our first task would be to understand the Grundy number categorically.

In the definition of Grundy number, we used a particular type of recursion \[\G_{\X}(x)=\mex(\{\G_{\X}(x')\mid x \rel x'\}) \text{  (\ref{eq:GrundyNumber})}.\]  

Let $\Pf(X)$ denote the set of all finite subset of a given set $X$ (\Cref{not:FinitePowersetFunctorPfin}). Then, the mex function is a function from $\Pf(\N)$ to $\N$: 
\[\mex \colon \Pf(\N) \to \N.\]

Similarly, we can rewrite games as a function between $\Pf(X)$ and $X$, but in the opposite direction!
For any directed graph $(X, {\rel}\subset X\times X)$, the corresponding \demph{neighborhood function} $\str\colon X \to \Pow(X)$ is defined by
\begin{equation}\label{eq:TheCorrespondence}
\str(x)= \{x'\in X\mid x\rel x'\}.
\end{equation}
If a graph $\X=(X, {\rel})$ satisfies the finite option condition in  \Cref{def:game}, in particular, if $\X$ is a game, the codomain $\Pow(X)$ can be reduced to $\Pf(X)$ and we obtain the function
\[
\str\colon X \to \Pf(X).
\]
Thus, we can rephrase the definition of Grundy number (\Cref{def:GrundyNumber} and \Cref{eq:GrundyNumber}) as the unique function $\G_{\X}\colon X \to \N$ that makes the following \dq{twisted} diagram 
    \begin{equation}\label{eq:GrundyNumberDiagram}
        \begin{tikzcd}[column sep = 50pt, row sep= 30pt]
        \Pf(X)\ar[r,"\Pf(\G_{\X})"]&\Pf(\N)\ar[d,"\mex"']\\
        X\ar[u,"\str"]\ar[r,"\G_{\X}"]&\N
        \end{tikzcd}
    \end{equation}
    commutative, where $\Pf(\G_{\X})$ denotes the direct image function! The \dq{finite time} condition ensures the unique existence of such a function $\G_{\X}\colon X \to A$.

    One may find it weird to consider such a twisted diagram. However this reflects a kind of \dq{twistedness} of the computation of the Grundy number;
    the recursive computation of Grundy numbers, as in \Cref{fig:RecursioveCalculationOfGrundyNumber}, is carried out by tracing the game \demph{backwards}. The somewhat unusual kind of \dq{twisted} commutativity in Diagram \ref{eq:GrundyNumberDiagram} corresponds to the nature of this recursive computation that goes against the flow of the game! Such \dq{twisted recursive computations} have a long history in category theory (or categorical computer science) under the name of \demph{recursive coalgebras} (or hylomorphisms), which we will recall in the next subsection.
    
\subsection{Preliminaries on coalgebras and recursive coalgebras}\label{ssec:PreliminariesOnCoalgebraicMethod}
This subsection aims to recall the basic notions in coalgebra theory, in particular, the definition and properties of recursive coalgebras. The reades are assumed to be familiar with the definitions of categories and functors, (and limits for \Cref{ssec:terminalGame}). For example, the first three sections of \cite{Maclane1998CWMcategories} suffice. For general theory and examples of coalgebras, see \cite{jacobs2017introduction}. For recursive coalgebras, see the book \cite{taylor1999practical}, papers including \cite{adamek2020well}, or papers cited therein. 


In order to separate the general theory from our particular context of game theory, we intentionally postpone the motivating examples untill the next subsection. So if the reader feels that it is too abstract, please refer to the next section for our game-theoretic examples.

\subsubsection{Algebras and coalgebras of an endofunctor}\label{ssec:AlgebrasAndCoalgebrasOfEndofunctor}
The definition of coalgebras (and algebras) of an endofunctor is surprisingly simple:
\begin{definition}[Coalgebras and algebras]\label{DefinitionCoalgebra}
    For an endofunctor $T\colon \C \to \C$ on a category $\C$, a \demph{$T$-coalgebra} $\X$ is a pair $\X=(X, \str)$ of an object $X \in \ob(\C)$ and a morphism $\str \colon X \to TX$. Dually, a \demph{$T$-algebra} $\A$ is a pair $\A=(A, \alpha)$ of an object $A \in \ob(\C)$ and a morphism $\alpha \colon TA \to A$.
\end{definition}

\begin{definition}[Coalgebra homomorphism and algebra homomorphism]
\label{DefinitionCoalgebraMorphism}
    For an endofunctor $T\colon \C \to \C$ on a category $\C$, a \demph{homomorphism of $T$-coalgebras} from $(X, \str)$ to $(X',\str')$ is a morphism $f\colon X \to X'$ in $\C$ such that the diagram
    \[
    \begin{tikzcd}
        TX\ar[r,"Tf"]&TX'\\
        X\ar[r,"f"]\ar[u,"\str"]&X'\ar[u,"\str'"]
    \end{tikzcd}
    \]
    commutes. Homomorohisms between $T$-algebras $(A, \alpha), (A', \alpha')$ are defined in the dual way:
    \[
    \begin{tikzcd}
        TA\ar[r,"Tf"]\ar[d,"\alpha"]&TA'\ar[d,"\alpha'"]\\
        A\ar[r,"f"]&A'.
    \end{tikzcd}
    \]
\end{definition}

\begin{notation}\label{not:AlgCoalgCategortOfCoAlgebras}
    For an endofunctor $T\colon \C \to \C$ on a category $\C$, the category of $T$-coalgebras and $T$-coalgebra homomorphisms is denoted by $\Coalg{T}$. The category of $T$-algebras and $T$-algebra homomorphisms is denoted by $\Alg{T}$.
\end{notation}

\subsubsection{Coalgebra-algebra morphisms and Recursive coalgebras}\label{ssec:RecursiveCoalgebras}
We will recall the notion of recursive coalgebra, due to \cite{osius1974categorical} and \cite{taylor1999practical}.
\begin{definition}
    For an endofunctor $T\colon \C \to \C$ on a category $\C$, a \demph{coalgebra-algebra morphism} from a $T$-coalgebra $\X=(X, \str)$  to  a $T$-algebra $\A=(A,\alpha)$ is a morphism $f\colon X\to A$ in the category $\C$ such that the diagram
    \[
    \begin{tikzcd}
        TX\ar[r,"Tf"]&TA\ar[d,"\alpha"']\\
        X\ar[u,"\str"]\ar[r,"f"]&A
    \end{tikzcd}
    \]
    commutes.
\end{definition}

\begin{definition}[Recursive coalgebra]\label{def:recursiveCoalgebras}
    A $T$-coalgebra $\X=(X,\str)$ is said to be \demph{recursive} if, for any $T$-alegebra $\A=(A,\alpha)$, there exists a unique coalgebra-algebra morphism from $\X$ to $\A$. The full subcategory of $\Coalg{T}$ that consists of all recursive $T$-coalgebras is denoted by $\RecCoalg{T}$.

    For a recursive coalgebra $\X=(X, \str)$ and a $T$-algebra $\A=(A, \alpha)$, the unique coalgebra-algebra morphism $X\to A$ is called the \demph{hylomorphism} and denoted by $\hylo_{\A, \X}\colon X \to A$.
\end{definition}

Let us emphasize the following \dq{stability} of hylomorphisms by compositions, which is theoretically almost trivial, but it provides a seemingly non-trivial consequences including \Cref{prop:AlgebraProvideGameValues}, \Cref{cor:gameMorphismsPreserveGrundyOutcomeAndOthers}, and \Cref{prop:AlgebraHomPreservesGameValue}, which generalizes \Cref{prop:GrundynumberIsMoreInformativeThanOutcome}.
\begin{lemma}[Stability of hylomorphisms by composition]\label{lem:StabilityOfHylomorphisms}
    Let $T\colon \C \to \C$ be an endofunctor on a category $\C$, $\X=(X, \str)$ be a recursive $T$-coalgebra, $\A=(A, \alpha)$ be a $T$-algebra, and $\hylo_{\A,\X}\colon X \to A$ be the hylomorphism between them. 
    \begin{itemize}
        \item For any recurisive $T$-coalgebra $\X'=(X',\str')$ and any $T$-coalgebra homomorphism $f\colon \X'\to \X$, the hylomorphism $\hylo_{\A,\X'}\colon X' \to A$ is given by $\hylo_{\A,\X'} = \hylo_{\A, \X}\circ f$.
        \[
        \begin{tikzcd}[row sep=5pt]
            X'\ar[rr, "\hylo_{\A,\X'}"]\ar[rd, "f"']&&A\\
            &X\ar[ru, "\hylo_{\A,\X}"']&
        \end{tikzcd}
        \]
        \item For any $T$-algebra $\A'=(A',\alpha')$ and any $T$-algebra homomorphism $g\colon \A\to \A'$, the hylomorphism $\hylo_{\A',\X}\colon X \to A'$ is given by $\hylo_{\A',\X} = g \circ \hylo_{\A, \X}$.
        \[
        \begin{tikzcd}[row sep=5pt]
            X\ar[rr, "\hylo_{\A',\X}"]\ar[rd, "\hylo_{\A,\X}"']&&A'\\
            &A\ar[ru, "g"']&
        \end{tikzcd}
        \]
    \end{itemize}
\end{lemma}
\begin{proof}
    The commutative diagmram
    \[
    \begin{tikzcd}[column sep=30pt]
        TX'\ar[r, "Tf"]&TX\ar[r, "T(\hylo_{\A,\X})"]&TA\ar[r, "Tg"]\ar[d, "\alpha"]&TA'\ar[d, "\alpha'"]\\
        X'\ar[r, "f"]\ar[u, "\str'"]&X\ar[u, "\str"]\ar[r, "\hylo_{\A,\X}"]&A\ar[r, "g"]&A'
    \end{tikzcd}
    \]
    completes the proof.
\end{proof}

\invmemo{
\begin{remark}[In which category is a coalgebra-algebra morphism actually a morphism?]
    It is just a profunctor. We can consider its collage/cograph. The above proposition is just a general phenomenon for profunctor.
\end{remark}
}

As we are interested in the terminal object of $\RecCoalg{T}$, let us recall the related result from \cite{lambek1968fixpoint} and \cite{capretta2006recursive}.
\begin{proposition}[Initial algebra is the terminal recursive coalgebra]\label{prop:TerminalRecursiveCoalgebraIsInitialAlgebra}
    For an endofunctor $T\colon \C \to \C$ on a category $\C$, if a $T$-algebra $\A=(A,\alpha)$ is the initial $T$-algebra, then
    \begin{itemize}
        \item the structure map $\alpha \colon TA \to A$ is an isomorphism \cite[Lemma 2.2]{lambek1968fixpoint}, and
        \item the coalgebra $(A, \alpha^{-1}\colon A \to TA)$ is the terminal recursive coalgebra \cite[Proposition 2]{capretta2006recursive}.
    \end{itemize}
\end{proposition}
\begin{proof}
Both are easily checked by diagram chases.
    The first assertion is known as Lambek's lemma. 
\end{proof}

\invmemo{The inverse direction is \cite[][Proposition 7]{capretta2006recursive}}

\begin{remark}[Relationship with well-founded coalgebras]
Actually, the recursive coalgebras we will consider (\Cref{ssec:GamesasrecursiveCoalgebras}) are also characterized as \textit{well-founed coalgebras}, which are known to be inherently related to recursice coalgebras. 
For their relationships, see \cite{taylor1999practical} and also a recent study \cite{adamek2020well}.
\end{remark}

\subsection{Games as recursive \texorpdfstring{$\Pf$}{Pf}-coalgebras}\label{ssec:GamesasrecursiveCoalgebras}
In the rest of \Cref{sec:GamesAsRecursiveCoalgebras}, we will specialize the general theory of recursive coalgebra (\Cref{ssec:PreliminariesOnCoalgebraicMethod}) to the case where the endofunctor $T\colon \C \to \C$ is the
the \demph{finite powerset functor}.
\begin{notation}[finite powerset functor]\label{not:FinitePowersetFunctorPfin}In this paper,
\begin{itemize}
    \item $\Pow\colon \Set \to \Set$ denotes the covariant \demph{powerset functor} that sends a set to its powerset and a function to its direct image function, and
    \item $\Pf:\Set \to \Set$ denotes \demph{the finite powerset functor}, which is the subfunctor of $\Pow:\Set \to \Set$ such that 
    $\Pf(X) = \{S \subset X\mid \# S < \infty\}$.
\end{itemize}
\end{notation}


In terms of coalgebras, any graph $\X=(X, {\rel})$ is associated with the corresponding $\Pow$-coalgebra \[\str\colon X \to \Pow(X)\] by \Cref{eq:TheCorrespondence}. The finite option condition in \Cref{def:game} is precisely saying that the codomain $\Pow(X)$ can be reduced to $\Pf(X)$. The next proposition states that the finite time condition is precicely the recursiveness (\Cref{def:recursiveCoalgebras}) and characterize games as the graphs that admits the \dq{twisted recursion} (Diagram \ref{eq:GrundyNumberDiagram}).
\begin{proposition}[Games as recursive coalgebras]\label{prop:GamesFiniteTimeIsRecursivenessAndGamesAsRecursiveCoalgebras}
    For any graph $\X=(X, {\rel})$ with the finite option condition (\Cref{def:game}), the corresponding $\Pf$-coalgebra
    \[
    \str\colon X \to \Pf(X)
    \]
    is a recursive coalgebra if and only if $\X$ is a game. In particular, for any set $X$, \Cref{eq:TheCorrespondence} provides a one-to-one correspondence between the game structures on $X$ and the recursive $\Pf$-coalgebra structures on $X$.
\end{proposition}
\begin{proof}
Let $\X=(X, {\rel})$ be a graph with the finite option condition. If the graph is a game, i.e., satisfies the finite time condition, we can recursively construct the unique coalgebra-algebra morphism to any $\Pf$-algebra, and hence we can prove that it is a recursive coalgebra. 

Conversely, take an arbitrary recursive coalgebra $\X= (X,\str)$.
We prove that the corresponding graph $(X, \rel)$ satisfies the finite time condition.
Consider a $\Pf$-algebra $\A=(\{\top, \bot\},\alpha)$ defined by
\[
\alpha(S) \coloneqq
\begin{cases}
    \top &( \bot \notin S)\\
    \bot &( \bot \in S).
\end{cases}
\]
On the one hand, the function $\FinTime_{\X}\colon X \to \{\top, \bot\}$ defined by
\[
\FinTime_{\X}(x) \coloneqq 
\begin{cases}
    \top &(\textrm{any path starting from $x$ terminates in finite steps})\\
    \bot & (\textrm{otherwise})
\end{cases}
\]
is an coalgebra-algebra morphism from $\X$ to $\A$. On the other hand, the constant function to $\top$, which is denoted by $\True_{\X}\colon X \to \{\top, \bot\}$, is also a coalgebra-algebra morphism from $\X$ to $\A$. 
Therefore, the assumption that $\X$ is recursive implies that $\FinTime_{\X}= \True_{\X}$, i.e., $\X$ satisfies the finite time condition.
    \revmemo{cite, general recursion schema?}
\end{proof}

\begin{example}[Nim as a recursive coalgebra]\label{exmp:NimCoalgebra}
    Let $(\N, \nu\colon \N \to \Pf(\N))$ denote the 
    $\Pf$-coalgebra corresponding to the ($1$-heap) nim game $\Nim{1}$ (\Cref{exmp:nimAsGraph}). This function $\nu\colon \N \to \Pf(\N)$ is given by 
    \[
    \mu(n)= \{0,1, \dots, n-1\},
    \]
    which looks like the von Neumann set-theoretic definition of natural numbers. This analogy becomes more precise in \Cref{exmp:vonNeumann}.
\end{example}

\subsection{Game morphisms as \texorpdfstring{$\Pf$}{Pf}-coalgebra homomorphismsm}\label{ssec:GameMorphismsAsCoalgebraMorphisms}
    As games are recursive coalgebras, it is reasonable to define \demph{game morphisms} as coalgebra homomorphisms. Notice that being a coalgebra homomorphism is much stronger than just being a graph morphism!
    For games $\X=(X, {\rel_\X})$ and $\Y=(Y, {\rel_\Y})$ and a function $f\colon X \to Y$, being a graph morphism $x_0\rel x_1 \implies f(x_0) \rel f(x_1)$ is equivalent to saying that $f(\str_{\X}(x))\subset \str_{\Y}(f(x))$, which is strictly weaker than being a coalgebra homomorphism $f(\str_{\X}(x))= \str_{\Y}(f(x))$. In order to ensure the opposite inclusion $f(\str_{\X}(x))\supset \str_{\Y}(f(x))$, we need to assume a kind of \dq{path-lifting property} as follows.
    \[\textrm{Being a game morphism }\iff
    \begin{tikzcd}[row sep = 20pt, column sep=40pt]
        \Pf(X)\ar[r,"\Pf(f)"]&\Pf(Y)\\
        &\\
        X\ar[uu, "\str_{\X}"]\ar[r,"f"]\ar[ruu, phantom, "\rotatebox{-45}{$\subset$}"]&Y\ar[uu, "\str_{\Y}"]\\
        {}\ar[r,"\textrm{(1) Graph morph.}", phantom ]&{}
    \end{tikzcd}
    \textrm{ and }
    \begin{tikzcd}[row sep = 20pt, column sep=40pt]
        \Pf(X)\ar[r,"\Pf(f)"]&\Pf(Y)\\
        &\\
        X\ar[uu, "\str_{\X}"]\ar[r,"f"]\ar[ruu, phantom, "\rotatebox{-45}{$\supset$}"]&Y\ar[uu, "\str_{\Y}"]\\
        {}\ar[r,"\textrm{(2) Path-lift.}", phantom ]&{}
    \end{tikzcd}
    \]
    
    \begin{definition}[Game morphism]\label{def:GameMorphism}
        A \demph{game morphism} from a game $\X = (X, \rel_{\X})$ to another game $\Y = (Y, \rel_{\Y})$ is 
        a function $f\colon X \to Y$ satsifying the following two conditions:
        \begin{description}
            \item[(1) Graph morphism] For any $x,x'\in X$, $x\rel_{\X} x'$ implies $f(x) \rel_{\Y} f(x')$. \label{ConditionGraphpreserving}
            \item[(2) Path-lifting] For any $x\in X$ and $y\in Y$, if $f(x) \rel_{\Y} y$, then there exists $x' \in \X$ such that $x\rel_{\X}x'$ and $f(x')= y$. (see \Cref{fig:PathLiftingProperty})\label{conditionLocallySurjective}
        \end{description}
    \end{definition}

    \begin{notation}\label{not:CategoryOfGames}
    The category of games and game morphisms is denoted by $\Gs$, and the canonical forgetful functor $\X=(X, {\rel}) \mapsto X$ is denoted by $U \colon \Gs \to \Set$.
\end{notation}

Summarizing the content of the previous and the current subsections, we arrive at the following so-called \dq{theorem,} which is in effect little more than a systematic restatement of the definitions introduced so far.
\begin{theorem}[Games as recursive coalgebras]\label{thm:main1:GamesAsRecursiveCoalgebras}
    The category of games is equivalent\footnote{In fact, they are isomorphic.} to the category of recursive $\Pf$-coalgebras.
    \[
    \Gs\simeq \RecCoalg{\Pf}
    \]
\end{theorem}
\begin{proof}

This follows from \Cref{prop:GamesFiniteTimeIsRecursivenessAndGamesAsRecursiveCoalgebras} and the argument above \Cref{def:GameMorphism}.
\end{proof}



Some readers might find this \dq{path-lifting} condition unnatural and hard to accept. On the one hand, some game theorists might think that it is a \dq{category-theory-biased} condition with little connection to game theoretic phenomena. On the other hand, some category theorists might worry that the path-lifting property could break pleasant properties of the category of graphs. I believe these concerns will be resolved in the rest of the paper, but let us close this subsection with a brief explanation of why the \dq{path-lifting property} is reasonable.

From a game theoretic perspective, the reason why we need the path-lifting property is simply because graph homomorphisms do not preserve game-theoretic data, such as the outcome and the Grundy number. For example, \Cref{fig:PathLiftingProperty} illustrates two graph morphisms: the left-hand one satisfies the path-lifting property and preserves the outcome, while the right-hand one does not. In fact, we will prove that all game morphisms preserve all \dq{recursively defined} game data, such as the outcome and the Grundy number (\Cref{prop:AlgebraProvideGameValues} and \Cref{cor:gameMorphismsPreserveGrundyOutcomeAndOthers}).

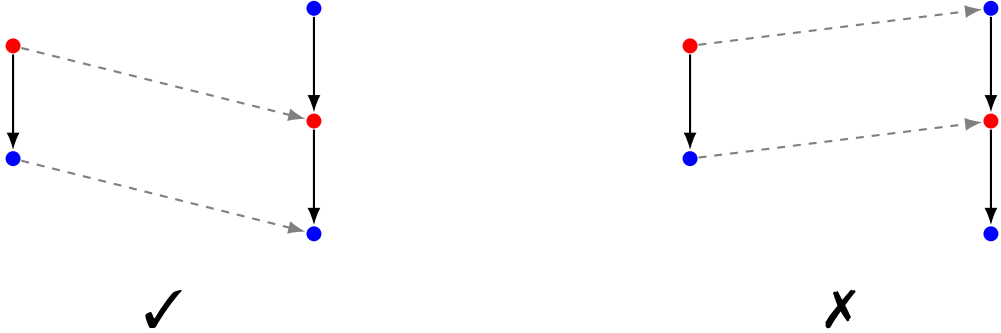
\begin{figure}[ht]
\centering
\begin{tikzpicture}[>=Latex, thick, scale=1.0]

  \begin{scope}[xshift=0cm]
    \node[circle, fill=red, inner sep=2pt] (LtopA) at (0,1.0) {};
    \node[circle, fill=blue, inner sep=2pt] (LbotA) at (0,-0.5) {};
    \draw[->] (LtopA) -- (LbotA);
  \end{scope}

  \begin{scope}[xshift=4cm]
    \node[circle, fill=blue, inner sep=2pt] (RtopA) at (0,1.5) {};
    \node[circle, fill=red, inner sep=2pt] (RmidA) at (0,0) {};
    \node[circle, fill=blue, inner sep=2pt] (RbotA) at (0,-1.5) {};
    \draw[->] (RtopA) -- (RmidA);
    \draw[->] (RmidA) -- (RbotA);
  \end{scope}

  \draw[dashed,->,draw=black!50] (LtopA) -- (RmidA);
  \draw[dashed,->,draw=black!50] (LbotA) -- (RbotA);

  \node[scale=2, black] at (2,-2.5) {\ding{51}};

  \begin{scope}[xshift=9cm]
    \node[circle, fill=red, inner sep=2pt] (LtopB) at (0,1.0) {};
    \node[circle, fill=blue, inner sep=2pt] (LbotB) at (0,-0.5) {};
    \draw[->] (LtopB) -- (LbotB);
  \end{scope}

  \begin{scope}[xshift=13cm]
    \node[circle, fill=blue, inner sep=2pt] (RtopB) at (0,1.5) {};
    \node[circle, fill=red, inner sep=2pt] (RmidB) at (0,0) {};
    \node[circle, fill=blue, inner sep=2pt] (RbotB) at (0,-1.5) {};
    \draw[->] (RtopB) -- (RmidB);
    \draw[->] (RmidB) -- (RbotB);
  \end{scope}

  \draw[dashed,->,draw=black!50] (LtopB) -- (RtopB);
  \draw[dashed,->,draw=black!50] (LbotB) -- (RmidB);

  \node[scale=2, black] at (11,-2.5) {\ding{55}};

\end{tikzpicture}
\caption{A game morphism (left side) and a graph morphism that does not satisfy the path-lifting property (right side) with {\color{blue} $P$-states} and {\color{red} $N$-states}}
\label{fig:PathLiftingProperty}
\end{figure}

From the categorical point of view,
the category of games $\Gs$ behaves surprisingly nice. Investigating the categorical properties of the category of games is not the main topic of this paper, but here I will list some facts.
\begin{theorem}[Categorical properties of the category of games]\label{thm:CategoricalPropertiesOfGames}
The category of games $\Gs$ has the following properties:
\begin{itemize}
    \item The category of games $\Gs$ is locally finitely presentable. In particular, it has all small limits and colimits.
    \item The forgetful functor $U\colon \Gs \to \Set$ is comonadic. In particular, 
    it preserves, reflects, and strictly creates all small colimits.
    \item  The category of games $\Gs$ 
    has a subobject classifier.
\end{itemize}
\end{theorem}
\begin{proof}
    The proofs are postponed until \Cref{app:CategoricalPropertiesOfGames}. 
\end{proof}

\begin{remark}[Game is not just a graph with properties.]\label{rem:NotGraphStructureAndProperties}
In \Cref{def:game}, we defined the notion of games as if they were just graphs with some properties. However, that way of speaking is not accurate in the spirit of the categorical distinction between stuff, structure, and properties, since
the canonical forgetful functor 
\[
\Gs \to \Graphs
\]
is not fully faithful. 
(see \cite[][2.4. Stuﬀ, structure, and properties]{baez2009lectures} for more explanation on the categorical distinction.) This fact is also mentioned in \cite[Section 5]{bavsic2024categories}.
In \cite{adamek2005introduction}, Ad\'{a}mek also points out that to regard $\Pf$-coalgebras as graphs is \dq{not a reasonable point of view} as follows:
    \begin{quote} {\cite[][Example 2.7.]{adamek2005introduction}}
        Sometimes one also identifies $Q$ with a finitely branching directed graph [...]
        However, this is often not a reasonable point of view because the coalgebra homomorphisms are much stronger than graph homomorphisms [...]
    \end{quote}
    Therefore, defining a game as a recursive coalgebra does not mean reducing it to graph theory. 
    Rather, it extracts the structure --- the \dq{neighborhood function} $\str$ --- from the underlying graph, and declares that this is the game-theoretic structure.
\end{remark}

\subsection{Game values as \texorpdfstring{$\Pf$}{Pf}-algebras}\label{ssec:GameValuesAsPfAlgebras}
In this subsection, we explain how game values, such as outcome and the Grundy number, are understood in our categorical framework.
First, we introduce the most na\"{i}ve notion of the game value: the quantity preserved by all game morphisms.
\begin{definition}\label{def:gameValueAsCocone}
    A \demph{game value} is a cocone under the forgetful functor $U \colon \Gs \to \Set$. In other words, a game value $\val$ is a set $A$ equipped with a (proper class size) family of functions from all games
    \[
    \{\val_{\X}\colon X\to A\}_{\X=(X, \str)\text{: game}}
    \]
    that is preserved by any game morphism $f\colon \X \to \Y$ in the sense that $\val_{\X}=\val_\Y \circ f$.
    \[
\begin{tikzcd}[row sep = 5pt]
    \X\ar[rd,"\val_\X"]\ar[dd, "f"']&\\
    &A\\
    \Y\ar[ru, "\val_{\Y}"']&
\end{tikzcd}
\]
\end{definition}

Thus, we obtain the following proposition as a special case of the abstract nonsense.
\begin{proposition}\label{prop:AlgebraProvideGameValues}
    For any $\Pf$-algebra $\A=(A, \alpha)$, the family of hylomorphisms \[\hylo_{\A}=\{\hylo_{\A,\X}\colon X\to A\}_{\X=(X, \str)\textrm{: game}}\] is a game value.
\end{proposition}
\begin{proof}
    This is just a special case of the first half of \Cref{lem:StabilityOfHylomorphisms}.
\end{proof}

Let us see some examples of $\Pf$-algebras and the induced game values. See \Cref{table:GameValuesAndAlgebras} for the correspondence list.
\begin{table}[ht]
    \centering
    \begin{tabular}{|c|c|} \hline 
         \textbf{Game value}&  $\Pf$-algebra\\ \hline 
         Outcome&  $\NP=(\{N,P\}, \np)$\\ \hline 
         Grundy number&  $\Mex=(\N,\mex)$\\\hline
 Ended states&$\Empty=(\{\top,\bot\},\emptyFunc)$\\\hline 
 Birthday& $\Xem=(\N,\xem)$\\\hline
 Misere outcome&$\MNP=(\N, \mnp)$\\\hline
 Remoteness&$(A, \alpha)$ in \Cref{exmp:remoteness}\\ \hline
    \end{tabular}
    \caption{Game values and the inducing $\Pf$-algebras}
    \label{table:GameValuesAndAlgebras}
\end{table}
First, let us observe that the outcome (\Cref{def:outcome}) and the Grundy number (\Cref{def:GrundyNumber}) are game values induced by $\Pf$-algebras.
\begin{example}[Outcome]
    The outcome (\Cref{def:outcome}) $\oc=\{\oc_\X\colon X \to \{N,P\}\}_{\X=(X, \str)\textrm{: game}}$ is induced by the $\Pf$-algebra $\NP=(\{N,P\}, \np)$, where
    \[
    \np\colon \Pf(\{N,P\})\to \{N,P\}\colon S \mapsto 
    \begin{cases}
        P &(P \notin S)\\
        N &(P\in S).
    \end{cases}
    \]
    So we have
    \[
    \oc=\hylo_{\NP}.
    \]
\end{example}

\begin{example}[Grundy number]
    The Grundy number $\G =\{\G_{\X}\colon X \to \N\}$ (\Cref{def:GrundyNumber}) is induced by the $\Pf$-algebra $\Mex=(\N, \mex)$ (\Cref{def:mex}).
    \[
    \G= \hylo_{\Mex}
    \]
\end{example}

Other important game values in combinatorial game theory are also induced by $\Pf$-algebras.

\begin{example}[Ended states]
Let us start with an easy example. For each state $x\in X$ of a game $\X=(X, {\rel})$, 
we define $\End_{\X}\colon X \to \{\top, \bot\}$ by 
\[\End_{\X}(x)=\top \iff \textrm{the game ends at $x$}\iff \str_{\X}(x)=\emptyset.\]
This game value is induced by the $\Pf$-algebra $\Empty=(\{\top, \bot\}, \emptyFunc)$

\[
     \emptyFunc \colon \Pf(\{\top,\bot\})\to \{\top,\bot\}\colon S \mapsto 
    \begin{cases}
        \top &(S=\emptyset)\\
        \bot &(S\neq \emptyset).
    \end{cases}
\]
So we obtain
\[
\End=\hylo_{\Empty}.
\]
\end{example}

\begin{example}[Birthday]
    For a game $\X=(X, {\rel})$ and a state $x\in X$, its \demph{birthday}\footnote{The notion of birthday is usually defined for a set-theoretically constructed game, as an element of $\H$. Our version of the birthday of $x\in X$ coincides with the birthday of $\rd_{\X}(x)\in \H$.}, denoted by $\BirthDay_\X(x)$, is defined to be the length of the longest possible play from the state $x$. This game value $\{\BirthDay_{\X}\colon X \to \N\}$ is induced by the $\Pf$-algebra $\Xem= (\N,\xem)$\footnote{The notation $\Xem$ is not conventional. We adopt it because $\Xem$ is a kind of dual to $\Mex$. See \Cref{app:GrundyNumberisLeftAdjointToNim} for more details.}, where
    \[
    \xem \colon \Pf(\N)\to \N\colon S \mapsto \min\{n \in \N\mid \textrm{for any }m\in S, m<n\}.
    \]
    So we have
    \[
    \BirthDay = \hylo_{\Xem}.
    \]
\end{example}

\begin{example}[Misere outcome]
    In combinatorial game theory, a \demph{misere game} is a game where the player who takes the last move loses. In our framework, a misere game is understood by the same notion of a game, but a different notion of the outcome: the outcome should be modified so that the ending state is an $N$-state.
    Defining a $\Pf$-algebra $\MNP=(\{N,P\},\mnp)$ by
    \[
     \mnp \colon \Pf(\{N,P\})\to \{N,P\}\colon S \mapsto 
    \begin{cases}
        P &(S=\{N\})\\
        N &(S\neq \{N\}),
    \end{cases}
    \]
    the corresponding game evaluation is the \demph{misere outcome} $\Moc=\hylo_{\MNP}$. 
\end{example}

\begin{example}[Remoteness {\cite{smith1966graphs}}]\label{exmp:remoteness}
If you do not just care about whether you win or lose, but also about how quickly you can win or how long you can postpone losing, then what you should look at is not the outcome but a more informative game value called the \demph{remoteness}.
The remoteness takes values in $A=\{P_{2k}\mid k\in\mathbb{N}\} \cup \{N_{2k+1}\mid k\in\mathbb{N}\}$, where
\begin{itemize}
    \item $\Remoteness(x)=P_{2k}$ means that, if you move to $x$, then you can win in at most $2k$ moves, and
    \item $\Remoteness(x) = N_{2k+1}$ means that, if you move to $x$, you can avoid losing for at least $2k+1$ moves.
\end{itemize}
This is a recuresively defined game value induced by the $\Pf$-algebra structure 
\[
\alpha(S) \coloneqq
\begin{cases}
    N_{\min\{n\mid P_{n}\in S\}+1}&(\textrm{if $S$ contains some $P$})\\
    P_{\xem\{n\mid N_n\in S\}}&(\textrm{if $S$ contains only $N$}).
\end{cases}
\]
\invmemo{I feel the construction of Bouton monoid is strongly related to the order structure, or semi lattice structures or $\Pf$-algebra structure for a monad $\Pf$.}
\end{example}

\begin{corollary}\label{cor:gameMorphismsPreserveGrundyOutcomeAndOthers}
    Any game morphism preserves the outcome, the Grundy number, ended states, the birthday, the misere outcome, and the remoteness.
\end{corollary}
\begin{proof}
These are just examples of \Cref{prop:AlgebraProvideGameValues}, and hence examples of the stability of hylomorphisms (\Cref{lem:StabilityOfHylomorphisms}).
\end{proof}

\invmemo{There is the other game value on the $2$-element set.}

Using the latter part of \Cref{lem:StabilityOfHylomorphisms}, we obtain the following statement:
\begin{proposition}\label{prop:AlgebraHomPreservesGameValue}
    For any $\Pf$-algebra homomorphism $g\colon \A\to \A'$ from $\A=(A, \alpha)$ to $\A'=(A',\alpha')$, the game value $\hylo_{\A'}$ can be calculated via $\hylo_{\A}$, i.e., we have $\hylo_{\A',\X} = g \circ \hylo_{\A, \X}$.
        \[
        \begin{tikzcd}[row sep=5pt]
            X\ar[rr, "\hylo_{\A',\X}"]\ar[rd, "\hylo_{\A,\X}"']&&A'\\
            &A\ar[ru, "g"']&
        \end{tikzcd}
        \]
\end{proposition}
\begin{proof}
    This is just the latter half of \Cref{lem:StabilityOfHylomorphisms}.
\end{proof}
For example, the fact that $\oc$ is calculated by $\G$ (\Cref{prop:GrundynumberIsMoreInformativeThanOutcome}) immediately follows from \Cref{prop:AlgebraHomPreservesGameValue} and the fact that $g\colon \N \to \{N,P\}$ defined by
\[
g(n) \coloneqq \begin{cases}
    N &(n>0)\\
    P &(n=0)
\end{cases}
\]
is a $\Pf$-algebra homomorphism $\Mex \to \NP$. We can do similar arguments for the following $\Pf$-algebra homomorphisms as well:
\begin{itemize}
    \item $\End$ is calculated by $\BirthDay$.
    \item $\End$ is calculated by $\Remoteness$.
    \item $\oc$ is calculated by $\Remoteness$.
\end{itemize}

\begin{remark}[Comparison with \cite{bavsic2024categories}]
 As explained in \Cref{sec:Introduction}, this subsection has a lot of intersection with \cite{bavsic2024categories}. 
 Our \Cref{prop:AlgebraProvideGameValues} corresponds to their notion of \textit{valuations}  \cite[Definition 4.11]{bavsic2024categories} and \cite[Proposition 4.17]{bavsic2024categories}. Many examples of game values in this subsection also appear in \cite[Section 4]{bavsic2024categories}.
\end{remark}

\subsection{The terminal game consists of hereditarily finite sets}\label{ssec:terminalGame}
This subsection aims to describe the terminal game, which plays the central role in the next section (\Cref{sec:NimSumTypeTheoremSchema}). First, we will describe it categorically (Diagram \ref{eq:TerminalGameHereditarilyFiniteAdamek}), and then we will write it down in a set-theoretic way (\Cref{prop:TerminalGameAndInitialAlgebraAreHereditarilyFiniteSets}).

The calculation of the terminal recursive $T$-coalgebra ($=$ the terminal game, in our case) is a priori nontrivial. However, thanks to \Cref{prop:TerminalRecursiveCoalgebraIsInitialAlgebra}, it is reduced to the classical theorem of transfinite construction of an initial $T$-algebra \cite{pohlova1973sums, adamek1974free}, which is known as \textit{Adamek's fixed point theorem}. (A concise explanation of the theorem, which suffices for our purpose, is found in \cite[][Proposition 4.6.1 (1)]{Jacobs_2016a}.) The construction states that the initial algebra of $\Pf\colon \Set \to \Set$, which also provides the terminal object in $\Gs\simeq \RecCoalg{\Pf}$ (\Cref{prop:TerminalRecursiveCoalgebraIsInitialAlgebra}), is given by the following colimit in $\Set$
\begin{equation}\label{eq:TerminalGameHereditarilyFiniteAdamek}
    \begin{tikzcd}
            \emptyset \ar[r]& \Pf(\emptyset) \ar[r]& \Pf(\Pf(\emptyset)) \ar[r]&\Pf(\Pf(\Pf(\emptyset))) \ar[r]&\cdots \textrm{The Terminal Game!}
        \end{tikzcd}
\end{equation}
Here, we utilize the finite option condition in \Cref{def:game} to ensure the assumption of Adamek's fixed point theorem; the functor $\Pf\colon \Set \to \Set$ is finitary and hence preserves this colimit (cf. \Cref{rem:finiteOptions}).

Let us write down this colimit. 
(A canonical choice of) this colimit is known as the set of \demph{hereditarily finite sets} in set theory. For its set-theoretic context and the terminology, see textbooks including \cite[][section I.10]{kunen2013set}.
\begin{definition}[Hereditarily finite sets]\label{def:HereditarilyFiniteSets}
A \demph{hereditarily finite set} is recursively defined as a finite set of hereditarily finite sets\footnote{Rigorously speaking, a hereditarily finite set is a set that is ensured to be hereditarily finite by this recursive definition.}.
     The set of all hereditarily finite sets is denoted by $\H$.
\end{definition}
As this recursive definition might look confusing at first glance, let us give several examples. 
\begin{example}\label{exmp:HereditarilyFiniteSets}
The following sets are hereditarily finite sets:
\begin{itemize}
    \item The empty set $\emptyset$ is trivially hereditarily finite since it has no element.
    \item Then, the set $\{\emptyset\}$ is also hereditarily finite. 
    \item By induction, every (von Neumann's formulation of) natural number $n=\{0,1, \dots n-1\}$ is hereditarily finite. In other words, we have
    \[\N = \{0=\emptyset,\ 1=  \{\emptyset\},\ 2= \{\emptyset,\{\emptyset\}\},\ 3=\{\emptyset,  \{\emptyset\}, \{\emptyset,\{\emptyset\}\}\},\; \dots\;\} \subset \H.\] 
    \item The set $\{\{\emptyset\}\}$ is a hereditarily finite set that is not a natural number.
    \item Informally speaking, hereditarily finite sets are sets described by a finite number of parentheses, like $\{\{\},\{\{\{\}\}\},\{\{\}\}\}$.
    \item While all hereditarily finite sets are finite, the converse does not hold. For example, while the set $\{\R\}$ is a singleton, but it is not hereditarily finite.
\end{itemize}
\end{example}

\begin{remark}[von Neumann hierarchy]
    $\H$ is often denoted by $V_{\omega}$ in order to emphasize that it is the $\omega$-th of the von Neumann hierarchy.
\end{remark}
By definition, $\Pf(\H)$ is set-theoretically identical to $\H$. Therefore, $\H$ admits 
\begin{align*}
    \textrm{the canonical $\Pf$-algebra structure  }\;& \id_{\H}:\H \to \Pf(\H)\textrm{, and}\\
\textrm{the canonical $\Pf$-coalgebra structure  }\;&\id_{\H}:\Pf(\H) \to \H.
\end{align*}
\begin{proposition}[Terminal object of $\Gs$]\label{prop:TerminalGameAndInitialAlgebraAreHereditarilyFiniteSets}
The set $\H$ is the initial $\Pf$-algebra with the $\Pf$-algebra structure $\id_{\H}:\H \to \Pf(\H)$. Therefore, $\H$ is also the terminal object of $\Gs \simeq \RecCoalg{\Pf}$ with the $\Pf$-coalgebra structure $\id_{\H}:\Pf(\H) \to \H$.
\end{proposition}
\begin{proof}
 The former part follows from Adamek's fixed point theorem (see Diagram \ref{eq:TerminalGameHereditarilyFiniteAdamek} and the discussion above). The latter follows from \Cref{prop:TerminalRecursiveCoalgebraIsInitialAlgebra}.
\end{proof}
By a mild abuse of notation, we will write $\H$ for the set of hereditarily finite sets, the initial $\Pf$-algebra, and the terminal recursive $\Pf$-coalgebra.

It is worth writing down the relation $\rel_{\H}$ of the terminal game $\H$.
\begin{definition}[Terminal game]\label{DefinitionUniversalGame}
    The \demph{terminal game} $\H=(\H,\rel_{\H})$ is the game whose underlying set is the set of all hereditarily finite sets $\H$ and whose relation $\rel_{\H}$ is given by 
    \[x \rel_{\H} y \iff y \in x.\]
\end{definition}

\begin{remark}[Ackerman's interpretation and the binary exponent nim]\label{rem:AckermanInterpretationAndTerminalGame}
As we promised, the terminal game is isomorphic to the binary exponent nim in \Cref{exmp:BinaryExponentNimOrTerminalGame}!
Consider the following $\Pf$-algebra structure on $\N$
    \[\bin: \Pf(\N)\to \N \colon S \mapsto \sum_{s\in S} 2^s,\]
    which is just the binary expression and hence bijective.
    The unique $\Pf$-algebra homomorphism $\mathsf{Ack}\colon (\H, \id_{\H}) \to (\N, \bin)$ is called \demph{Ackerman's interpretation} \cite{ackermann1937widerspruchsfreiheit} and known to be bijective.
    Therefore, this $\Pf$-algebra $(\N,\bin)$ is isomorphic to the initial $\Pf$-algebra $(\H, \id_{\H})$ and hence initial as well.
    Consequently, the inverse function $\bin^{-1}\colon \N \to \Pf(\N)$, which coincides with the binary exponent nim, is also a terminal object in $\Gs \simeq \RecCoalg{\Pf}$. 
\end{remark}


\begin{notation}\label{not:ReductionMap}
    For each game $\X$, the unique game morphism to the terminal game $\H$ is denoted by $\rd_{\X}\colon \X \to \H$.
\end{notation}

\revmemo{the notation}

\begin{example}\label{exmp:vonNeumann}
    For the $1$-heap nim $\Nim{1}=(\N, \str)$, the game morphism $\rd_{\Nim{1}}\colon \N \to \H$ is nothing other than the von Neumann formulation of natural numbers. (cf. \Cref{exmp:NimCoalgebra})
\end{example}
\begin{lemma}[Characterization of game values]\label{lem:GameValuesCharacterization}
    For a (class-size) family of functions $\val=\{\val_\X\colon X \to B\}_{\X=(X, \str)\textrm{: game}}$, the following conditions are equivalent.
    \begin{enumerate}
        \item The family $\val$ is a game value (\Cref{def:gameValueAsCocone}), i.e., any game morphism preserves $\val$.
        \item The equality $\val_{\X}=\val_{\H}\circ \rd_{\X}$ holds for any game $\X$. \revmemo{If A is recursively defined, this coincides with the hylo morephism decomposition. Cite it.}
        \item There exists a $\Pf$-algebra $\A=(A, \alpha)$ and a function $g\colon A \to B$ such that $\val_{\X}=g\circ \hylo_{\A, \X}$ for any game $\X$.
    \end{enumerate}
\end{lemma}
\begin{proof}
    The implication $(1)\implies (2)$ follows since $\rd_{\X}\colon \X \to \H$ is a game morphism. The implication $(2)\implies (3)$ follows since we can take the initial $\Pf$-algebra $\H$ as $\A$, and $\val_{\H}$ as $g$ using the terminal game $\H$
(cf. \Cref{prop:TerminalGameAndInitialAlgebraAreHereditarilyFiniteSets}). 
    The last implication $(3)\implies (1)$ follows from \Cref{prop:AlgebraProvideGameValues}.
\end{proof}

\begin{remark}[$\H$ as colimit, and a game with a starting state.]\label{rem:HasColimitsAndGameWithStartingPoint}
    For category theorists, \Cref{lem:GameValuesCharacterization} is an immediate corollary of an almost trivial fact that $\H$ is the colimit of the forgetful functor $U\colon \Gs \to \Set$, and the morphisms $\{\rd_{\X}\colon X \to \H\}_{\X=(X, \str)\textrm{: game}}$ is the colimit cocone of this large colimit. Since the category $\Gs$ has the terminal object $\H$, the colimit of this diagram is just the image of the terminal object $\H$. As game values are defined as cocones of $U$ (\Cref{def:gameValueAsCocone}), one can say that $\H$ is the \dq{universal game value} (but in a bit boring way).

    This phenomenon is loosely related to the relationships between different definitions of games. In some context, an impartial game is defined as an element of $\H$. From our point of view, each element $x\in X$ defines a \dq{game} $\rd_{\X}(x)\in \H$ in that sense. In slightly different words, if we consider games $\X$ equipped with a \dq{starting state} $x\in \X$, we can define the category of pointed games $\Gs_{*}$, whose connected components are in one-to-one correspondence with hereditarily finite sets:
   \[
   \pi_0\left(\Gs_{*}\right) = \pi_0\left(\int U\right) = \colim \left(U\colon \Gs\to \Set\right)= \H.
   \]
\end{remark}

\begin{remark}\label{rem:notAllGameValueAreInducedByPfAlgebras}
    As the condition (3) in \Cref{lem:GameValuesCharacterization} infers, some game values are not induced by $\Pf$-algebras. In fact, while the number of $\Pf$-algebra structures on the $2$-element set $\{\top, \bot\}$ is $\#\{\top,\bot\}^{\Pf(\{\top,\bot\})}=16$, there are exactly $\#\{\top,\bot\}^{\H}=2^{\aleph_0}$ distinct game values on it due to \Cref{lem:GameValuesCharacterization}.
    \revmemo{The other 2-element algebra}
\end{remark}





\section{Application: Nim-sum type theorem schema and the Bouton monoids}\label{sec:NimSumTypeTheoremSchema}
As an example of usages of our framework, this section provides a theorem schema that generalizes the nim-sum rule (\Cref{thm:GeneralizedBoutonTheoremNimSumRule}). We firstly define the notion of a Bouton monoid (\Cref{def:BoutonMonoidAsUniversal}), and then prove that it always exists (\Cref{thm:existenceTheoremOfBoutonMonoid}).


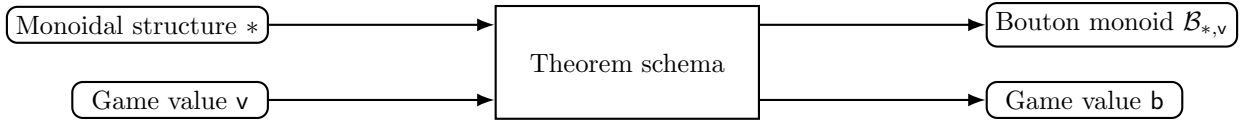
\begin{figure}[ht]
\centering
\begin{tikzpicture}[>=Latex, thick, node distance=10mm and 15mm]

  \node[draw, minimum width=35mm, minimum height=15mm, align=center] (box) {Theorem schema};

  \node[draw, rounded corners=4pt, left=of box.west, xshift=-15mm, yshift=5mm, anchor=east, minimum width=26mm, align=center] (mono) {Monoidal structure $\ast$};
  \node[draw, rounded corners=4pt, left=of box.west, xshift=-15mm, yshift=-5mm, anchor=east, minimum width=26mm, align=center] (pf) {Game value $\val$};

  \node[draw, rounded corners=4pt, right=of box.east, xshift=15mm, yshift=5mm, anchor=west, minimum width=26mm, align=center] (bouton) {Bouton monoid $\B_{{\ast}, \val}$};
  \node[draw, rounded corners=4pt, right=of box.east, xshift=15mm, yshift=-5mm, anchor=west, minimum width=26mm, align=center] (gb) {Game value $\Gb$};

  \coordinate (boxWup) at ([yshift=5mm]box.west);
  \coordinate (boxWdn) at ([yshift=-5mm]box.west);
  \coordinate (boxEup) at ([yshift=5mm]box.east);
  \coordinate (boxEdn) at ([yshift=-5mm]box.east);

  \draw[->] (mono.east) -- (boxWup);
  \draw[->] (pf.east)   -- (boxWdn);
  \draw[->] (boxEup)  -- (bouton.west);
  \draw[->] (boxEdn)  -- (gb.west);
\end{tikzpicture}
\caption{How theorem schema works}
  \label{fig:HowTheoremSchemaWorks}
\end{figure}
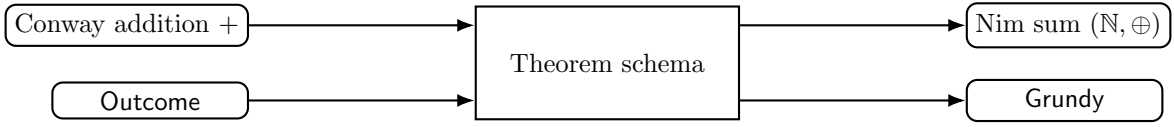
\begin{figure}[ht]
\centering
\begin{tikzpicture}[>=Latex, thick, node distance=10mm and 15mm]

  \node[draw, minimum width=35mm, minimum height=15mm, align=center] (box) {Theorem schema};

  \node[draw, rounded corners=4pt, left=of box.west, xshift=-15mm, yshift=5mm, anchor=east, minimum width=26mm, align=center] (mono) {Conway addition ${+}$};
  \node[draw, rounded corners=4pt, left=of box.west, xshift=-15mm, yshift=-5mm, anchor=east, minimum width=26mm, align=center] (pf) {$\oc$};

  \node[draw, rounded corners=4pt, right=of box.east, xshift=15mm, yshift=5mm, anchor=west, minimum width=26mm, align=center] (bouton) {Nim sum $(\N, \nsum)$};
  \node[draw, rounded corners=4pt, right=of box.east, xshift=15mm, yshift=-5mm, anchor=west, minimum width=26mm, align=center] (gb) {$\G$};

  \coordinate (boxWup) at ([yshift=5mm]box.west);
  \coordinate (boxWdn) at ([yshift=-5mm]box.west);
  \coordinate (boxEup) at ([yshift=5mm]box.east);
  \coordinate (boxEdn) at ([yshift=-5mm]box.east);

  \draw[->] (mono.east) -- (boxWup);
  \draw[->] (pf.east)   -- (boxWdn);
  \draw[->] (boxEup)  -- (bouton.west);
  \draw[->] (boxEdn)  -- (gb.west);
\end{tikzpicture}
    \caption{The nim-sum as an example of the theorem schema}
    \label{fig:NimSumasSchemaWorks}
\end{figure}

\subsection{Decomposing games with monoidal structures}\label{ssec:GameDecompotisionAsMonoidalStructure}
The goal of this subsection is to find a systematic way to decompose a complex game into simpler games. An established categorical tool to decompose objects is a \demph{monoidal structure} on a category.
For those who are not familiar with this notion, a standard reference is \cite{Maclane1998CWMcategories}.

In this subsection, we will observe several monoidal structures on the category $\Gs$.

\begin{remark}\label{rmk:nonUnital}
    In this section, we also consider \demph{non-unital monoidal category}, which does not require the unit object. \revmemo{cite... Lurie?} Everything in this section works by replacing monoidal by non-unital monoidal, and monoid by semigroup. If the reader is not familiar with non-unital monoidal structure, one can restrict arguments to the unital cases, while considering remoteness will require non-unital one (\Cref{exmp:RemotenessAsBoutonMonoid}).
\end{remark}

Hisorically, the Conway addition is regarded as the canonical choice of monoidal structure on categories of games (for example, \cite{joyal1977remarques}). It also provides a monoidal structure on our category.
\begin{proposition}\label{prop:conwayAddition}
    The Conway addition is a symmetric monoidal closed structure on $\Gs$ with the monoidal unit $1=(\{*\}, \rel_{1}=\emptyset)$.
\end{proposition}
\begin{proof}
    Being symmetric monoidal is easily proven. The closedness, which we will not use in this section, is proven in \Cref{Cor:closedness}.
\end{proof}


Let us see some other examples.
\begin{example}[Selective sum \cite{smith1966graphs}]\label{exmp:SelctiveSum}
For a finite number of games $\X_1, \dots, \X_n$ with $n\geq0$, their \demph{selective sum} $\X_1\lor  \dots\lor \X_n$ is the game in which each player selects at least one games out of the $n$ games and makes one move for each selected game. More foramlly spealing, the selective sum $\X\lor \Y$ of two games $\X=(X, \rel_\X), \Y=(Y,\rel_\Y)$ is defined as follows:
    \begin{itemize}
        \item Its underlying set is $X\times Y$.
        \item $ (x,y) \rel_{\X\lor\Y} (x',y')$ is defined by
        \[
    (x,y) \rel_{\X\lor\Y} (x',y') \iff (x\to_\X x' \land y=y')\lor(x\to_\X x' \land y\to_\Y y')\lor(x=x' \land y\to_\Y y')
    \]
    \end{itemize}
    The selective sum defines a (unital, symmetric, and closed) monoidal structure on $\Gs$ with the monoidal unit $1=(\{*\}, \rel_{1}=\emptyset)$.
\end{example}

\begin{example}[Conjunctive sum \cite{smith1966graphs}]\label{exmp:ConjunctiveSum}
For a finite and positive number of games $\X_1, \dots, \X_n$ with $n>0$, their \demph{conjunctive sum} $\X_1\land \dots\land \X_n$ is the game in which each player makes a move for all of the $n$ games. More foramlly spealing, the conjunctive sum $\X\land \Y$ of two games $\X=(X, \rel_\X), \Y=(Y,\rel_\Y)$ is defined as follows:
    \begin{itemize}
        \item Its underlying set is $X\times Y$.
        \item $ (x,y) \rel_{\X\land\Y} (x',y')$ is defined by
        \[
    (x,y) \rel_{\X\lor\Y} (x',y') \iff (x\to_\X x' \land y\to_\Y y')
    \]
    \end{itemize}
    The conjuctive sum defines a non-unital\footnote{The conjunctive sum defines a monoidal structure on $\Coalg{\Pf}$ with a monoidal unit (\Cref{fig:LoopNonGame}), which fails to be a game.} (but symmetric and closed) monoidal structure on $\Gs$.
\end{example}

\subsection{Synthesizing game values with Bouton monoids}\label{ssec:SynthesizingWithBoutonMonoid}

\subsubsection{Bouton monoid}\label{sssec:BoutonMonoid}
In the story of the Bouton theorem (\Cref{sec:GamesasGraphs}), the abelian group of the nim-sum $(\N, \nsum)$ (\Cref{def:NimSum2}) and the notion of Grundy number $\G=\{\G_\X\colon X \to \N\}_{\X=(X, \str)\textrm{: game}}$ plays a central role. More concretely, there are two characterizing properties of them:
\begin{description}
    \item[(1) Informative (\Cref{prop:GrundynumberIsMoreInformativeThanOutcome})] For any game $\X$ and any element $x\in X$, the Grundy number is more informative than the outcome in the sense that
    \[\oc_{\X}(x)=P \iff \G_{\X}(x)=0.\]
    \item[(2) Decomposable (\Cref{thm:GeneralizedBoutonTheoremNimSumRule})]
    For any games $\X=(X, \str_{\X}), \Y=(Y,\str_{\Y})$ and any states $x\in X, y\in Y$, the Grundy number of the sum state $(x,y)\in \X+\Y$ is decomposable in the sense that 
    \[\G_{\X+ \Y}(x,y) = \G_{\X}(x) \nsum \G_{\Y}(y).\]
\end{description}


We will define the notion of a Bouton monoid as the minimum monoid that satisfies the two conditions above. The existence will be proven later (\Cref{thm:existenceTheoremOfBoutonMonoid}). 

In the follwoing definition, we will use the notion of lax monoidal functors. For those who are not familiar with that notion, one can just consider monoidal structure $\ast$ on $\Gs$ that makes $U$ strictly monoidal, i.e., $U\X \times U\Y =U(\X\ast \Y)$. In fact, all  the three monoidal structures in \Cref{ssec:GameDecompotisionAsMonoidalStructure} make $U\colon \Gs \to \Set$ strict monoidal.
\begin{notation}[Notations around lax monoidal functors]\label{not:laxmonoidal}
    If one consider a monoidal structure $\ast$ on the category $\Gs$ that makes $U\colon \Gs \to \Set$ lax monoidal\footnote{Rigorously speaking, as the comparison map is a part of the lax monoidal functor, saying that \dq{$U$ is lax monoidal} is somewhat misleading. Similarly, to be rigorous, we need to specify the associativity and the unitality maps of the monoidal structure. However, we intentionally hide these data by a conventional abuse of notation, partially because of \Cref{rmk:nonUnital}.} (with respect to the cartesian monoidal structure on $\Set$), we will write 
\[
\mu_{\X,\Y}\colon U\X \times U\Y \to U(\X\ast \Y)
\]
for the chosen comparison map. Furthermore, for a pair of states $(x,y)\in U\X\times U\Y$, we will write $(x,y)\in U(\X\ast \Y)$ for the state $\mu_{\X,\Y}(x,y)$. Notice that it is compatible with the notations in the previous sections, for example, in \Cref{thm:GeneralizedBoutonTheoremNimSumRule}.
\end{notation}

Now we are ready to state the definition of the Bouton monoid. To aid in understanding the definition, \Cref{tab:correspondenceTableBetweenTheGeneralCasesAndTheNimSumSase} shows the correspondence between the general case and the specific case of the nim-sum.

\begin{definition}\label{def:BoutonMonoidAsUniversal}
Let $\ast$ be a monoidal structure on $\Gs$ such that the forgetful functor $U\colon \Gs \to \Set$ is lax monoidal, and let $\val=\{\val_{\X}\colon X \to A\}_{\X=(X, \str)\textrm{: game}}$ be a game value.
The \demph{Bouton monoid} of $\ast$ and $\val$ is 
\begin{itemize}
    \item a game value $\Gb=\{\Gb_{\X}\colon X \to \B\}$ equipped with
    \item a function $a\colon \B \to A$ and
    \item a monoid structure $\circledast\colon \B \times \B \to \B$
\end{itemize}
that satisfies the following three conditions.
\begin{description}
    \item[(1) Informative] For any game $\X$, we have
    \[\val_{\X}= a\circ \Gb_{\X}.\]
    \item[(2) Decomposable]
    For any games $\X=(X, \str_{\X}), \Y=(Y,\str_{\Y})$ and any states $x\in X, y\in Y$, we have 
    \[\Gb_{\X\ast \Y}(x,y) = \Gb_{\X}(x) \circledast \Gb_{\Y}(y).\]
    \item[(3) Minimum] It is minimum among such data. More concretely, for any data of the same type $(\Gb'=\{\Gb'_{\X}\colon X \to \B'\}, a'\colon \B' \to A, \circledast' \colon \B'\times \B' \to \B')$ satisfying (the appropriate substitutes of) the above two conditions, there exsits a surjective monoid homomorphism $q\colon \bullet \twoheadrightarrow \B$ from a submonoid 
    $\bullet \subset\B'$
    such that the following diagram
    \[
    \begin{tikzcd}[row sep=10pt]
            &\bullet\ar[r, rightarrowtail, "\subset"]\ar[dd, twoheadrightarrow, "q"]&\B'\ar[rd, "a'"]&\\
        X\ar[ru, "\Gb'_{\X}"]\ar[rd, "\Gb_{\X}"']&&&A\\
            &\B\ar[rru, "a"', bend right =17pt]&&
    \end{tikzcd}
    \]
    commutes for any game $\X=(X,\str)$.
\end{description}
\end{definition}
\begin{table}[ht]
    \centering
    \begin{tabular}{c|c}
         \textbf{General case}& \textbf{Nim-sum case}\\ \hline
         Monoidal structure $\ast$& Conway addition $+$\\ \hline
         (Target) game value $\val$&$\oc$\\ \hline
         $\mu_{\X,\Y}$& $\id_{X\times Y}$\\ \hline
         $\Gb_{\X}$& $\G_{\X}$\\ \hline
         $\B$& $\N$\\ \hline
         $\circledast$& Nim-sum $\nsum$\\ \hline
         $a\colon \B \to A$& $\N \to \{N,P\}\colon 0\mapsto P$\\ \hline
         $\val_{\X}= a\circ \Gb_{\X}$ & \Cref{prop:GrundynumberIsMoreInformativeThanOutcome} \\ \hline
         $\Gb_{\X\ast \Y}(x,y) = \Gb_{\X}(x) \circledast \Gb_{\Y}(y)$& \Cref{thm:GeneralizedBoutonTheoremNimSumRule}\\ \hline
    \end{tabular}
    \caption{The correspondence table between the general cases and the nim-sum case}\label{tab:correspondenceTableBetweenTheGeneralCasesAndTheNimSumSase}
\end{table}

\begin{remark}[Is this minimality condition weird?]
    Some readers may be unfamiliar with this kind of minimality condition. But actually it is very common especially in the automata theory. For example, the minimality of automata and syntactic monoids (\Cref{exmp:syntacticMonoid}) are stated in this form. See textbooks including  \cite{pin2020mathematical}, and see also \cite{colcombet2020automata} for an elegant categorical framework. This minimality condition says that the Bouton monoid $\B$ is a subquotient of any others, which informally means that $\B$ has neither redundant elements ($=$ no proper submonoids) nor redundunt fineness ($=$ no proper quotient monoids).
\end{remark}

The goal of this subsection is to prove the existence of the Bouton monoid.
\begin{theorem}[Existence of Bouton monoid]\label{thm:existenceTheoremOfBoutonMonoid}
    For any monoidal structure $\ast$ on $\Gs$ such that the forgetful functor $U\colon \Gs \to \Set$ is lax monoidal, and any game value $\val=\{\val_{\X}\colon X \to A\}_{\X=(X, \str)\textrm{: game}}$, the Bouton monoid exists.
\end{theorem}

\subsubsection{Step 1: Miniature monoid}\label{sssec:MinitureMonoid}
We start with constructing an approximation of the Bouton monoid only by the information of the monoidal structure $\ast$ forgetting the terget game value $\val$. We will call the approximation by a temporary name: \demph{the miniture monoid} of the monoidal structure $\ast$.

The following trivial proposition is usually very boring due to the triviality of the terminal object. But it is worth considering in a category of coalgebras, in which a terminal object tends to be non-trivial.
\begin{lemma}
    For any monoidal category 
    with a terminal object $1$, the object $1$ has a unique monoid structure with respect to the monoidal structure.
\end{lemma}
\begin{proof}
In a monoidal category 
    $(\C, I, \otimes)$
    with a terminal object $1$, 
    the multiplication map $m\colon 1\otimes 1 \to 1$ (and the unit map $e\colon I \to 1$ if you consider a unital structure) is uniquely determined since $1$ is terminal. All required commutativities are trivial since $1$ is the terminal object. 
\end{proof}
In our case of the category of games $\Gs$, the unique game morphism $\rd_{\H \ast\H}\colon \H \ast \H \to \H$ provides the unique $\ast$-monoid structure on the terminal game $\H$. Then, we can convert the $\ast$-monoid $(\H, \rd_{\H \ast\H}\colon \H \ast \H \to \H)$ to a usual monoid by the following well-known fact.
\begin{fact}\revmemo{cite}
    Any lax monoidal functor $F\colon (\C, I_{\C}, \otimes_{\C})\to (\D,I_{\D},\otimes_{\D} )$ sends a monoid object in $\C$ to a monoid object in $\D$.
\end{fact}

By the assumption that $U\colon \Gs \to \Set$ is lax monoidal, we obtain the induced monoid whose multiplacation function is given by
\[
U\H\times U\H \xrightarrow{\mu_{\H,\H}} U(\H\ast \H) \xrightarrow{U(\rd_{\H\ast \H})} U\H,
\]
where $\mu_{\X,\Y}\colon U(\X)\times U(\Y)\to U(\X\ast \Y)$ denotes the comparison map of the lax monoidal functor $U$.
We will call this monoid structure on the set $\H$ \demph{the miniature monoid} of the monoidal structure $\ast$, and write $\rbmult \colon\H\times \H \to \H$ for the multiplication function.

\begin{example}[Miniture monoids for some monoidal structures]\label{exmp:rawBoutonMonoidConway}
    For the Conway addition $+$, the miniature monoid structure $+_{\H} $ on $\H$ is given by
    \[
    A+_{\H}  B = \{a+_{\H}  B \mid a\in A\} \cup \{A+_{\H}  b\mid b\in B\}.
    \]
    In usual context like \cite{siegel2013combinatorial}, this is just called the Conway addition.
    Similarly, the miniature monoid structure of $\lor$ (\Cref{exmp:SelctiveSum}) and $\land$ (\Cref{exmp:ConjunctiveSum}) are given by
\begin{align*}
    A \lor_{\H} B  &=  \{a\lor_{\H}  B \mid a\in A\} \cup \{a\lor_{\H}  b\mid a\in A, b\in B\} \cup \{A\lor_{\H}  b\mid b\in B\}\text{ and}\\
    A \land_{\H} B  &=  \{a\land_{\H}  b\mid a\in A, b\in B\}.
\end{align*}
\end{example}

\begin{lemma}[Miniature monoid reflects the ambient monoidal structure]\label{lem:miniatureMonoidReflectsTheMonoidalStructure}
For any pair of games $\X=(X,\rel_{\X})$ and $\Y=(Y,\rel_{\Y})$, we have the following equation in the miniature monoid:
    \[
    \rd_{\X\ast\Y}(x,y) = \rd_{\X}(x) \ast_{\H} \rd_{\Y}(y).
    \]
\end{lemma}
\begin{proof}
    Since $\H$ is the terminal game, we have the following commutative diagram in $\Gs$:
    \begin{equation}\label{eq:LSCalgebra}
        \begin{tikzcd}
        &\X\ast \Y\ar[rd,"\rd_{\X\ast \Y}"],\ar[d,"\rd_{\X}\ast \rd_{\Y}"']&\\
        &\H\ast\H\ar[r,"\rd_{\H\ast \H}"']&\H.
    \end{tikzcd}
    \end{equation}
    By sending this diagram by $U$ and utilizing the naturality of $\mu$, we have the following commutative diagram in $\Set$:
    \[
    \begin{tikzcd}
        X\times Y=U\X\times U\Y \ar[d,"U\rd_{\X}\times U\rd_{\Y}"']\ar[r,"\mu_{\X,\Y}"]&U(\X\ast \Y)\ar[rd,"U(\rd_{\X\ast \Y})"],\ar[d,"U(\rd_{\X}\ast \rd_{\Y})"']&\\
        \H \times \H=U\H \times U\H\ar[r,"\mu_{\H,\H}"]\ar[rr, bend right , "\ast_{\H}"']&U(\H\ast\H)\ar[r,"U(\rd_{\H\ast \H})"]&U(\H)=\H.
    \end{tikzcd}
    \]
    This completes the proof.
\end{proof}

\subsubsection{Step 2: Minimum monoid factorization}\label{sssec:MinimumMonoidFactorization}
The miniture monoid $(\H, \ast_{\H})$ depends only on the monoidal structure $\ast$, and is independent of the target game value $\val$. So we want to minimize $\H$ keeping the ability to calculate the target game value $\val$. We introduce a very simple gadget for minimizing a monoid. See \Cref{app:MinimumMonoidFactorization} for related contexts, examples, and the proof of the existence.
\begin{definition}[Minimum monoid factorization]\label{def:minimumMonoidFactorization}
    For a monoid $M$ and a function $f\colon M \to A$ to a set $A$, the \demph{minimum monoid factorization} of $f\colon M \to A$ is a monoid $M_f$ equipped with a factorization of the function $f$ into a monoid homomorphism $q_f\colon M \to M_f$ followed by a function $a_f\colon M_f\to A$
    \[
    \begin{tikzcd}
        M\ar[rr, "f"', bend right]\ar[r, "q_f"]&M_f\ar[r, "a_f"]&A
    \end{tikzcd}
    \]
    that satisfies the following minimality condition.
    
    \begin{description}
        \item[Minimum] For any monoid $M'_f$ equipped with a factorization $f= M \xrightarrow{q'_f} M'_f \xrightarrow{a'_f} A$ into a monoid homomorphism $q'_f \colon M \to M'_f$ followed by a function $a'_f\colon M'_f \to A$, there exsits a surjective monoid homomorphism $q\colon \bullet \twoheadrightarrow M_f$ from a submonoid 
    $\bullet \subset M'_f$
    such that the following diagram commutes.
    \[
    \begin{tikzcd}[row sep=10pt]
            &\bullet\ar[r, rightarrowtail, "\subset"]\ar[dd, twoheadrightarrow, "q"]&M'_f\ar[rd, "a'_f"]&\\
        M\ar[ru, "q'_f"]\ar[rd, "q_f"']&&&A\\
            &M_f\ar[rru, "a_f"', bend right =17pt]&&
    \end{tikzcd}
    \]
    \end{description}
\end{definition}
One can check that the minimum monoid factorization always uniquely exists (up to canonical isomorphisms. See \Cref{prop:UniqueExistenceOfMinimumMonoidFactorization} for a detailed proof and a construction. If one considers a non-unital monoidal structure (\Cref{rmk:nonUnital}), one can use the semigroup version in \Cref{rem:minimumAlgebraFactorization}.

Let us return to the context of the Bouton monoid.
Since $\H$ has the miniture monoid structure (\Cref{sssec:MinitureMonoid}), we can consider the minimum monoid factorization of the function $\val_{\H}\colon \H \to A$. 
Let $\H \xrightarrow{q_{\val_\H}}\B\xrightarrow{a_{\val_{\H}}} A$ be the minimum monoid factorization of the function $\val_{\H}\colon \H \to A$ with respect to the miniature monoid structure $\ast_{\H}\colon \H \times \H \to \H$. 
Let $\circledast \colon \B \times \B \to \B$ denote the induced monoid structure on $\B$.
For any game $\X=(X, \str)$, due to \Cref{lem:GameValuesCharacterization}, we obtain the following decomposition of the function $\val_{\X}\colon X \to A$.
\[
\begin{tikzcd}
    X\ar[r, "\rd_{\X}"] \ar[rrr, bend right =50, "\val_{\X}"']&\H\ar[r, "q_{\val_{\H}}", twoheadrightarrow]\ar[rr, bend right , "\val_{\H}"']&\B\ar[r, "a_{\val_{\H}}"]&A
\end{tikzcd}
\]
We define $\Gb_{\X}\colon X \to \B$ by $\Gb_{\X}\coloneqq q_{\val_\H} \circ \rd_{\X}$.
\[
\begin{tikzcd}
    X\ar[r, "\rd_{\X}"] \ar[rrr, bend right =50, "\val_{\X}"']\ar[rr, "\Gb_{\X}", bend left=40]&\H\ar[r, "q_{\val_{\H}}", twoheadrightarrow]\ar[rr, bend right , "\val_{\H}"']&\B\ar[r, "a_{\val_{\H}}"]&A.
\end{tikzcd}
\]
The family of functions $\Gb=\{\Gb_{\X}\colon X \to \B\}_{\X=(X,\str)\textrm{: game}}$ is a game value, also due to \Cref{lem:GameValuesCharacterization} since it factors through $\rd=\{\rd_{\X}\colon X \to \H\}_{\X=(X,\str)\textrm{: game}}$.

We will prove that the triple of the game value $\Gb=\{\Gb_{\X}\colon X \to \B\}_{\X=(X,\str)\textrm{: game}}$, the function $a_{\val_{\H}}\colon \B\to A$, and the monoid structure $\circledast \colon \B \times \B \to \B$, all defined just above, provides the Bouton monoid for $\ast$ and $\val$.
Regarding the first condition \dq{(1) Informative}, the required equation $\val_{\X}=a_{\val_{\H}} \circ \Gb_{\X}$ obviously holds by definition.
Regarding the second condition \dq{(2) Decomposable}, we have
    \[\Gb_{\X\ast \Y}(x,y) = q_{\val_{\H}}(\rd_{\X\ast \Y}(x,y))=q_{\val_{\H}}(\rd_{\X}(x)\ast_{\H} \rd_{\Y}(y))= q_{\val_{\H}}(\rd_{\X}(x))\circledast q_{\val_{\H}}(\rd_{\Y}(y))=
\Gb_{\X}(x) \circledast \Gb_{\Y}(y),\]
where we use \Cref{lem:miniatureMonoidReflectsTheMonoidalStructure} at the second equation. 

Therefore, in order to prove the existence theorem (\Cref{thm:existenceTheoremOfBoutonMonoid}), it suffices to check that these data $(\B, \Gb,a_{\val_{\H}} ,\circledast)$ satisfies the third condition \dq{(3) Minimum.} Assuming that the same type of data $(\B', \Gb'=\{\Gb'_{\X}\colon X \to \B'\}, a'\colon \B' \to A, \circledast' \colon \B'\times \B' \to \B')$ satisfies the first two conditions in \Cref{def:BoutonMonoidAsUniversal}, we prove that there exsits a surjective monoid homomorphism $q\colon \bullet \twoheadrightarrow \B$ from a submonoid 
    $\bullet \subset\B'$
    such that the following diagram
    \[
    \begin{tikzcd}[row sep=10pt]
            &\bullet\ar[r, rightarrowtail, "\subset"]\ar[dd, twoheadrightarrow, "q"]&\B'\ar[rd, "a'"]&\\
        X\ar[ru, "\Gb'_{\X}"]\ar[rd, "\Gb_{\X}"']&&&A\\
            &\B\ar[rru, "a_{\val_{\H}}"', bend right =17pt]&&
    \end{tikzcd}
    \]
    commutes for any game $\X=(X,\str)$.

The first observation is that the function $\Gb'_{\H}\colon \H \to \B'$ is a monoid homomorphism since the following diagram commutes. The commutativity of the top part is the definition of $\ast_{\H}$, the commutativity of the left-hand trapezoid follows from the \dq{(2) Decomposable} assumption on $\circledast'$, and the commutativity of the right-hand triangle follows since $\Gb'$ is a game value. 
\[
    \begin{tikzcd}
        \H \times \H=U\H \times U\H\ar[r,"\mu_{\H,\H}"]\ar[rr, bend left , "\ast_{\H}"]\ar[d, "\Gb'_{\H} \times \Gb'_{\H}"]&U(\H\ast\H)\ar[r,"U(\rd_{\H\ast \H})"]\ar[rd, "\Gb'_{\H\ast \H}"]&U(\H)=\H\ar[d, "\Gb'_{\H}"]\\
        \B'\times\B'\ar[rr, "\circledast'"]&&\B'
    \end{tikzcd}
    \]
    Therefore, the universality of $\B$ as the minimum factorization monoid implies that there exists a surjective monoid homomorphism $q\colon \bullet \twoheadrightarrow \B$ from a submonoid $\bullet \subset \B'$ such that the following diagram commutes.
    \[
    \begin{tikzcd}[row sep=10pt]
            &\bullet\ar[r, rightarrowtail, "\subset"]\ar[dd, twoheadrightarrow, "q"]&\B'\ar[rd, "a'"]&\\
        \H\ar[ru, "\Gb'_{\H}"]\ar[rd, "\Gb_{\H}"']&&&A\\
            &\B.\ar[rru, "a_{\val_{\H}}"', bend right =17pt]&&
    \end{tikzcd}
    \]
    Therefore, \Cref{lem:GameValuesCharacterization} implies that
    \[
    \begin{tikzcd}[row sep=10pt]
            &&\bullet\ar[r, rightarrowtail, "\subset"]\ar[dd, twoheadrightarrow, "q"]&\B'\ar[rd, "a'"]&\\
        X\ar[rru, bend left=20, "\Gb'_{\X}"]\ar[rrd, bend right=20, "\Gb_{\X}"']\ar[r, "\rd_{\X}"]&\H\ar[ru, "\Gb'_{\H}"]\ar[rd, "\Gb_{\H}"']&&&A\\
            &&\B.\ar[rru, "a_{\val_{\H}}"', bend right =17pt]&&
    \end{tikzcd}
    \]
    commutes for any game $\X=(X,\str)$. This completes the proof of \Cref{thm:existenceTheoremOfBoutonMonoid}.

\subsubsection{Examples of Bouton monoids}
We conclude this section by giving several examples of the Bouton monoids.
\begin{proposition}[A categorical characterization of Nim-sum]\label{prop:CategoricalCharacterizationOfNim-Sum}
    The monoid of nim-sum $(\N, 0, \nsum)$ is isomorphic to the Bouton monoid $\B_{{+}, \oc}$ of the Conway addition ${+}$ and the outcome $\oc$.
\end{proposition}

\begin{example}[$\B_{\lor, \oc}$]
    The Bouton monoid $\B_{\lor, \oc}$ of the selective sum (\Cref{exmp:SelctiveSum}) and the outcome is given by $\B_{\lor, \oc}=\{P,N\}$ as essentially pointed out in \cite{smith1966graphs}. The binary operation $\B_{\lor, \oc} \times \B_{\lor, \oc} \to \B_{\lor, \oc}$ is given by $\min$ with respect to the order $N<P$.
\end{example}

\begin{example}[$\B_{\land, \oc}$]\label{exmp:RemotenessAsBoutonMonoid}
    The Bouton monoid of the conjunctive sum $\land$ (\Cref{exmp:ConjunctiveSum}) and the outcome $\oc$ is known to be remoteness (\cite{smith1966graphs}). 
    \[
    \B_{\land, \oc} =\{P_{2k}\mid k\in\mathbb{N}\} \cup \{N_{2k+1}\mid k\in\mathbb{N}\}
    \]
    Here, as the considered monoidal structure $\land$ is non-unital, the resulting Bouton \dq{monoid} $\B_{\land, \oc}$ is also non-unital, i.e., it is just a semigroup (\Cref{rmk:nonUnital}). In fact, the binary operation is given by $\min$ with respect to the order $P_0 <N_1<P_2<N_3<P_4< \dots$, which does not admit a neutral element. The minimality of the Bouton monoid can be checked by considering $\ElM$ (\Cref{exmp:EffeuillerLaMarguerite}, \Cref{exmpl:remotenessAndEffeuillerLaMarguerite}).
\end{example}

    
An example of a game-theoretically motivated Boouton monoid whose target game value is not $\oc$ is mentioned in \Cref{sssec:BoutonMonoidCircularNim} (without being able to give an explicit description).

\revmemo{misere quotient}

\section{Concluding remarks and open questions}\label{sec:OpenQuestions}
As there are still a lot of things to be calculated, we list some of the remaining questions with related observations.
\subsection{The Missing Piece: differential structures of games}\label{ssec:missingPiece}
While our paper defined the Bouton monoid and described an abstruct construction (\Cref{ssec:SynthesizingWithBoutonMonoid}), there remains the crucial part --- the concrete construction of the Bouton monoid. For example, the proof of \Cref{prop:CategoricalCharacterizationOfNim-Sum} relies on \Cref{thm:GeneralizedBoutonTheoremNimSumRule}. In that sense, it remains unsatisfactory as a true ``demystification'' of the nim-sum.

What we should demystify is the proof of \Cref{thm:GeneralizedBoutonTheoremNimSumRule}. 
The only non-trivial part of the proof is the following equation regarding the nim-sum $\nsum$ and the mex operator:\revmemo{cite}
\begin{equation}\label{eq:recursivedefinitionOfNimSum}
    \forall S,T\in \Pf(\N),\; \mex(S)\nsum \mex(T) = \mex(S\nsum \mex(T) \cup \mex(S)\nsum T),
\end{equation}
where we adopt the notion $n\nsum S\coloneqq\{n\nsum s\mid s\in S\}$ for $n\in \N$ and $S\in \Pf(\N)$. Furthermore, this equation recursively determines the value of the nim-sum since $n\nsum m$ is deteremined by choosing $S=\{0,1, \dots, n-1\}$ and $T=\{0,1, \dots, m-1\}$. \invmemo{This should be related to the section-technique since $S,T$ are chosen to be $v(n),v(m)$.}

The equation (\Cref{eq:recursivedefinitionOfNimSum}) is reminiscent of the integral version of the Leibniz rule, called the \demph{Rota-Baxter equation}
\[
\textstyle
\int f \times \int g =\int\left(f\times \int g + \int f \times g\right)
\]
by regarding the sets $S,T$ as functions, $\nsum$ as products, and $\mex$ as integration. This ``Rota-Baxter equation" of the nim-sum is a reflection of a more trivial ``Leibniz rule" of the Conway addition:
\[
\str_{\X+\Y}(x,y) = \str_{\X}(x)\times\{y\} \cup \{x\}\times\str_{\Y}(y),
\]
which resembles the genuine Leibniz rule $\partial(f\times g) = \partial f \times g + f\times \partial g$. Note that we see here a kind of the twistedness in Diagram \ref{eq:GrundyNumberDiagram} as the \dq{differential part} $\str$ goes in the opposite direction of the \dq{integral part} $\mex$.

Expecting the remaining equation (\Cref{eq:recursivedefinitionOfNimSum}) is actually a reflection of some differential structure on $(\Gs, +)$, we ask the first question.
\begin{question}
    How do these ``differential structures" appear in our formulation, especially for the Bouton monoids $\B_{\val,+}$ with respect to the Conway addition?
\end{question}
Let us describe a partial solution to this question:
\begin{proposition}[Rota-Baxter equation for the Bouton monoids w.r.t. the Conway addition]
    Let $\val$ be a game value. If the Bouton monoid $\{\Gb_{\X}\colon X\to\B_{\val, +}\}_{\X=(X, \str):\text{ game}}$ of $\val$ with respect to the Conway addition $+$ is induced by a $\Pf$-algebra structure $(\B_{\val, +}, \beta\colon\Pf(\B_{\val, +})\to \B_{\val, +})$, then the monoid structure $\circledast\colon \B_{\val, +}\times \B_{\val, +}\to \B_{\val, +}$ satisfies the following Rota-Baxter equation:
    \[
    \beta(S)\circledast \beta(T) = \beta(S\circledast \beta(T) \cup \beta(S)\circledast T)
    \]
    for any $S,T \in \Pf(\B_{\val, +})$.
\end{proposition}
\begin{proof}
In this proof, the image of a subset $S\subset X$ by a function $f\colon X \to Y$ is denoted by $f[S]\subset Y$.
    Since the function $\Gb_\H\colon \H \to \B_{\val, +}$ is surjective by the consturction, for any $S,T \in \Pf(\B_{\val, +})$, we can take $A,B \in \Pf(\H)$ such that $\Gb_\H[A]=S$ and $\Gb_\H[B]=T$. By the assmption that the game value $\Gb$ is induced by a $\Pf$-algebra structure $(\B_{\val, +}, \beta\colon\Pf(\B_{\val, +})\to \B_{\val, +})$, we have $\Gb_{\X}(x)=\beta(\Gb_{\X}[\str_{\X}(x)])$, in particular, we have $\Gb_{\H}(A)=\beta(\Gb_{\H}[A])=\beta(S)$ and $\Gb_{\H}(B) =\beta(T)$.
    As $\Gb_{\H}$ is a monoid homomorphism by the construction, we have
    \[
    \Gb_{\H}(A)\circledast\Gb_{\H}(B)= \Gb_{\H}(A+_{\H}B).
    \]
    The left-hand side is equal to $\beta(S)\circledast\beta(T)$.
    The right-hand side is calculated as
    \begin{align*}
        \Gb_{\H}(A+_{\H}B) =& \beta\left(\Gb_{\H}[\{a+_{\H}B\mid a\in A\}\cup \{A+_{\H} b\mid b\in B\}]\right)\\
        =& \beta(\{\Gb_{\H}(a+_{\H}B)\mid a\in A\}\cup \{\Gb_{\H}(A+_{\H} b)\mid b\in B\})\\
        =& \beta(\{\Gb_{\H}(a)\circledast\Gb_{\H}(B)\mid a\in A\}\cup \{\Gb_{\H}(A)\circledast \Gb_{\H}(b)\mid b\in B\})\\
        =& \beta(\Gb_{\H}[A]\circledast\Gb_{\H}(B)\cup \Gb_{\H}(A)\circledast \Gb_{\H}[B])\\
        =& \beta(S\circledast\beta(T)\cup \beta(S)\circledast T).
    \end{align*}
    This completes the proof.
\end{proof}

\revmemo{Is this related to the existing categorifications of differential structures?}

\subsection{Partisan games}\label{ssec:PartisanGames}
What we called games in this paper is usually called impartial games. We intentionally choose to concentrate on impartial games since our central motivation is to demystify the nim-sum. However, in many different contexts, partisam games, games in which a Left-player and a Right-player have different options, are of interest.

Instead of considering $\Pf\colon \Set \to \Set$, we can alternatively consider $\Pf\times \Pf\colon \Set \to \Set$ to obtain the parallel theory of ``partisan games as reursive coalgebras" and to mimick all the abstract nonsenses we used in this paper. (A similar idea of using coalgebras of $\Pow \times \Pow$ to study partisan games has already appeared in \cite{honsell2009conway}.)
\begin{question}
    To which extent the theory of partisan games are comprehended by replacing $\Pf$ by $\Pf\times \Pf$?
\end{question}

For example, we can capture the definition of (short and partisan) Conway games, which is recursively defined as a pair of finite sets of Conway games, as follows:
\begin{proposition}
    The terminal recursive $\Pf\times \Pf$-coalgebra is given by the set $\mathbb{G}$ of all (short and partisan) Conway games. 
\end{proposition}

\subsection{Classification problem of monoidal structures}\label{ssec:ClassificationProblemOfMonoidalStructures}
In order to fathom the applicativity of \Cref{thm:existenceTheoremOfBoutonMonoid}, we would like to classify monoidal structures on $\Gs$. The first question to be asked is the most restrictive one:
\begin{question}
    Classify all symmetric mononoidal closed structures on $\Gs$ such that $U\colon \Gs \to \Set$ is strictly monoidal.
\end{question}
Notice that the closedness assumption is redundant due to \Cref{Cor:closedness}. 
A similar research has been recently done for graphs \cite{kapulkin2024closed}. See also \cite{foltz1980algebraic}.

If there are infinitely many examples, we obtain infinitely many applications of \Cref{thm:existenceTheoremOfBoutonMonoid}, and if there are only finitely many examples, we obtain a finite list of all ``nicely behaved notions of game-sums," which may be of independent interest.

As \Cref{thm:existenceTheoremOfBoutonMonoid} works also for a non-unital monoidal structure, the loosest classification problem regarding  \Cref{thm:existenceTheoremOfBoutonMonoid} is the following one.
\begin{question}
    Classify all non-unital mononoidal structures on $\Gs$ such that $U\colon \Gs \to \Set$ is lax monoidal.
\end{question}

\subsection{Other questions}
In the rest of this section, we will list some problems, which might be of interest for some readers.
\subsubsection{The closed structure}\label{sssec:ClosedStructure}
One can check that the Conway addition structure is symmetric monoidal closed (\Cref{Cor:closedness}) by an anstract argument. 
\begin{question}[The closed structure]
    What is the internal hom with respect to the Conway addition?
\end{question}
In the classical context (including \cite{joyal1977remarques}), the closedness plays the key role of switching the two players. It makes their category of games compact closed, which categorifies the classical ordered abelian group of Conway games!
The closed structure might be better understood when we consider the partisan games $\RecCoalg{\Pf\times \Pf}$ (see \Cref{ssec:PartisanGames}).

\revmemo{
\begin{question}
    Are the Bouton functions always induced by a $\Pf$-algebra structure?
\end{question}
}

\subsubsection{Other types of Bouton monoid: towards circular nim}\label{sssec:BoutonMonoidCircularNim}
One of the motivating phenomena of this research is an open problem in combinatorial game theory: \demph{the circular nim} \cite{dufour2013circular}. In circular nim $\mathrm{CN}(n,k)$ for $0<k\leq n$, $n$ piles are arranged on a circle. At each move, a player selects $k$ consecutive piles and removes a total of at least one stone from them. 
The winning strategy for $\mathrm{CN}(6,2)$ remains an open question \cite[Section 3]{dufour2013circular}.
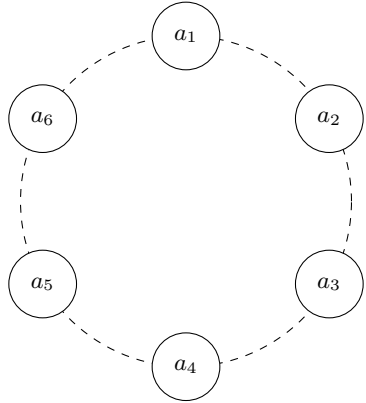
\begin{figure}[ht]
    \centering
    \begin{tikzpicture}[x=1cm,y=1cm, every node/.style={font=\small}]
  \def\r{2.2} 

  \draw[dashed] (0,0) circle (\r);

  \foreach \i [count=\k from 0] in {1,...,6}{
    \node[circle,draw,minimum size=9mm, fill=white] (P\i) at ({90-60*\k}:\r) {$a_{\i}$};
  }
\end{tikzpicture}
    \caption{Circular nim}
    \label{fig:circularNim}
\end{figure}

But how is $\mathrm{CN}(6,2)$ related to the Bouton monoid? By an easy argument, one can observe that $\mathrm{CN}(6,2)$ contains $\Nim{1}+(\Nim{1}\lor (\Nim{1}+\Nim{1}))$ as its subgame.
\begin{align*}
    \Nim{1}+(\Nim{1}\lor (\Nim{1}+\Nim{1})) &\rightarrowtail \mathrm{CN}(6,2)\\
    (a_1, a_4, a_3, a_5) &\mapsto (a_1, 0, a_3, a_4, a_5, 0)
\end{align*}
Therefore, as a subproblem of $\mathrm{CN}(6,2)$, we need to understand the winning strategy of $\Nim{1}+(\Nim{1}\lor (\Nim{1}+\Nim{1}))$, which also remains an open question, to the best of the author’s knowledge. In order to decompose the nested structure of the game, one na\"{i}ve idea is to understand the nested Bouton monoid
\[
\B_{\lor, \B_{+,\oc}} = \B_{\lor, \G}.
\]

\begin{question}
    What is the Bouton monoid $\B_{\lor, \G}$? Is it just the miniture monoid $(\H, \lor_{\H})$? 
\end{question}
If it is a proper quotient of $(\H, \lor_{\H})$, it might be a useful tool to tackle $\mathrm{CN}(6,2)$; otherwise, it at least captures one mathematical aspect of the difficulty of $\mathrm{CN}(6,2)$. Let us say a game value $\val$ is \demph{essentially indecomposable} by a monoidal structure $\ast$ if $\Gb_{\H}\colon \H \to \B_{\ast,\val}$ is bijective. Then, we can ask: is the Grundy number essentially indecomposable by the selective sum? This kind of question might serve as a passive approach to a difficult game in order to mathematically state one aspect of its difficulty.

\revmemo{It's quite natural to consider a multiplayer game, probability game, 2-turns/1-turn game, and other variants of games. And might be dealt with using some appropriate algebras. Furthermore, can we consider other graph data, like entropy?  I know there is a notion of the temperature of a game. Is it an example of this framework?}


\subsubsection{Other endofunctors}\label{sssec:OtherFunctors}
In this paper, we have concentrated on the endofunctor $\Pf\colon \Set \to \Set$ by restraining our categorical desire to generalize it.
As we used only simple abstract nonsense, our argument should be able to be categorically generalized.
\begin{question}
    To what extent is our framework categorically generalized? In particular, can it be done for any locally presentable category $\C$ and any accessible endofunctor $T\colon \C \to \C$?
\end{question}

This question is of not only categorical interest but also game-theoretic interest. As explained in \Cref{ssec:PartisanGames}, $\Pf\times \Pf\colon \Set \to \Set$ is another (possibly the most important) example. We are also interested in other variants like multiplayer games, probabilistic games, linear games, and transfinite games.


\revmemo{
\begin{question}
    Is there a nice way to calculate Bouton monoid? like Adamek construction? cf. Nim-sum is just the symmetric difference via the Ackerman interpretation.
\end{question}
}
\revmemo{
\begin{question}
    As a relative local state classifier.
\end{question}
}



\appendix

\section{Existence of the minimum monoid factorizations}\label{app:MinimumMonoidFactorization}
In this appendix, we will explain the context and the construction of the minimum monoid factorization (\Cref{def:minimumMonoidFactorization}).

\begin{proposition}[Unique existence and the construction]\label{prop:UniqueExistenceOfMinimumMonoidFactorization}
    For any monoid $M$, any set $A$, and any function $f\colon M \to A$, its minimum monoid factorization $f\colon M \xrightarrow{q_f} M_f \xrightarrow{a_f} A$ uniquely exists up to a canonical isomorphism. Furthermore, the monoid homomorphism $q_f\colon M \to M_f$ is surjective, and the corresponding monoid congruence ${\sim_f}=\{(m,m')\in M\times M\mid q_f(m)=q_f(m')\}\subset M\times M$ is given by
    \[
    m\sim_{f} m' \iff \textrm{for any }a,b\in M, f(amb)=f(am'b).
    \]
\end{proposition}
\begin{proof}
First, we need to prove that the equivalence relation $\sim_f$ is a monoid congruence. If $m\sim_f m'$ and $n\sim_f n'$, we have
\[
f(amnb)= f(am'nb)=f(am'n'b)
\]
for any $a,b\in M$,
and thus $mn\sim_f m'n'$. This proves that $\sim_f$ is a monoid congruence, and thus the monoid structure on $M/{\sim_f}$ and a surjective monoid homomorphism $q_f\colon M \twoheadrightarrow M_f \coloneqq M/{\sim_f}$ are well-defined. Furthermore, the function $a_f\colon M_f\to A: [m]_{\sim_f}\mapsto f(m)$ is also well-defined because we have $m\sim_f m' \implies f(m) = f(m')$ by taking the neutral element as $a,b$. 

So far, we have obtained a factorization $f\colon M \xrightarrow{q_f} M_f \xrightarrow{a_f} A$. 
We prove the minimality.
Let us take a monoid $M'_f$ equipped with a factorization $f= M \xrightarrow{q'_f} M'_f \xrightarrow{a'_f} A$ into a monoid homomorphism $q'_f \colon M \to M'_f$ followed by a function $a'_f\colon M'_f \to A$. Then we prove that there exsits a surjective monoid homomorphism $q\colon \Image(q'_f) \twoheadrightarrow M_f$ from the submonoid 
    $\Image(q'_f) \subset M'_f$
    such that the following diagram commutes.
    \[
    \begin{tikzcd}[row sep=10pt]
            &\Image(q'_f) \ar[r, rightarrowtail, "\subset"]\ar[dd, twoheadrightarrow, "q"]&M'_f\ar[rd, "a'_f"]&\\
        M\ar[rrr, "f", gray, dashed]\ar[ru, "q'_f", twoheadrightarrow]\ar[rd, "q_f"', twoheadrightarrow]&&&A\\
            &M_f\ar[rru, "a_f"', bend right =17pt]&&
    \end{tikzcd}
    \]
As both $q_f\colon M \twoheadrightarrow M_f$ and $q'_f\colon M\twoheadrightarrow \Image(q'_f)$ are surjective monoid homomorphisms, if a function $q\colon \Image(q'_f)\to M_f$ that makes the left triangle commutative exsits, then it automatically follows that the function $q$ is also a surjective monoid homomorphism, and that the right-hand side skew rectangle is also commutative. So it suffices to prove the implication $q'_f(m) = q'_f(m') \implies m\sim_f m'$ for any $m,m'\in M$.  If $q'_f(m) = q'_f(m')$, then for any $a,b\in M$, we have $q'_f(amb)=q'_f(a)q'_f(m)q'_f(b)= q'_f(a)q'_f(m')q'_f(b)=q'_f(am'b)$, hence $f(amb)=f(am'b)$.
This proves the implication $q'_f(m) = q'_f(m') \implies m\sim_f m'$ and completes the proof.
\end{proof}


The structure of the minimum quotient is by no means anything nontrivial nor advanced -- the idea of \dq{minimizing an algebraic structure without losing too much information} is quite mundane as exemplified by the following examples.
\begin{example}[Real numbers in complex numbers]
If you care whether the sum of two complex numbers is real, then it suffices to just keep track of the imaginary parts. This obvious fact can be described as follows.
    Consider a function $f\colon \comp \to \{\top, \bot\}$ defined by
    \[
    f(z)=
    \begin{cases}
        \top &(z\in \R)\\
        \bot &(z\notin\R ).
    \end{cases}
    \]
    Then, the minimum monoid factorization of $f$ (with respect to the additive structure of $\comp$) is given by taking the imaginary part $\comp \twoheadrightarrow \R i \to \{\top, \bot\}$.
\end{example}

\begin{example}[Syntactic monoid of a languege]\label{exmp:syntacticMonoid}
    For a finite set $\Sigma$, a \demph{language} $L$ is just a subset of the free monoid $L\subset \Sigma^{*}$. The \demph{syntactic monoid} of the language $L$ is defined as the minimum monoid factorization of the corresponding characteristic function $\chi_{L}\colon \Sigma^*\to \{\top, \bot\}$. See some textbooks on the formal language theory, for example \cite{pin2020mathematical}, for more details on the syntactic monoids.
\end{example}


\begin{remark}[Minimum $\mathbb{T}$-algebra factorization]\label{rem:minimumAlgebraFactorization}Actually, any equational theory $\mathbb{T}$ admits the property like \Cref{prop:UniqueExistenceOfMinimumMonoidFactorization}. In fact, for any $\mathbb{T}$-algebra $A$, the complete lattice of $\mathbb{T}$-algebra congruences is a complete sublattice of the complete lattice of equivalence relations on the underlying set $A$
\[
    \mathrm{Congruence}_{\mathbb{T}}(A) \hookrightarrow \mathrm{Equiv(A)},
    \]
    which is known in the universal algebra theory (\cite[II Theorem 5.3]{burris1981course}).
    Due to the adjoint functor theorem, the embedding admits both left and right adjoint, and the right adjoint part provides the desired minimum $\mathbb{T}$-algebra factorization.
    In particular, we have the minimum subgroup factorization, which we need in the context of \Cref{rmk:nonUnital}.
\end{remark}

\begin{remark}
    The construction of the minimum monoid factorization is ubiquitous in combinatorial game theory. For example, even the \dq{equality} of games is conventionally defined in this way!
    See \cite{siegel2013combinatorial}.
\end{remark}

\section{Categorical properties of \texorpdfstring{$\Gs$}{Games}}\label{app:CategoricalPropertiesOfGames}
In this appendix, we study the categorical properties of $\Gs$ and prove \Cref{thm:CategoricalPropertiesOfGames}:
\begin{description}
    \item[\Cref{thm:LocallyFinitePresentabilityOfTheGameCategory}] The category of games $\Gs$ is locally finitely presentable. In particular, it has all small limits and colimits.
    \item[\Cref{prop:Comonadicity}] The forgetful functor $U\colon \Gs \to \Set$ is comonadic. In particular, 
    it preserves, reflects, and strictly creates all small colimits.
    \item[\Cref{prop:SubobjectClassifier}]  The category of games $\Gs$ 
    has a subobject classifier.
\end{description}

We will also prove that
\begin{description}
        \item[\Cref{cor:Epi-MonoFactorization}] The category of games $\Gs$ has the epi-mono orthogonal factorization system.
        \item[\Cref{rem:GsisNotaTopos}] The category of games $\Gs$ is not cartesian closed. In particular, it is not a topos.
\end{description}


\begin{remark}[Comparison with locally finite Kripke frames]\label{rem:KripkeFrame}
In \cite{de2024profiniteness}, the categorical properties of the category of Kripke frames $\mathbf{KFr}\simeq \Coalg{\Pow}$ are studied, and then the category of locally finite Kripke frames $\mathbf{KFr}_{lf}$ is studied in the context of the duality in modal logic. As our category of games $\Gs$ is a full subcategory of $\mathbf{KFr}_{lf}$, some arguments in this appendix already appeared in \cite{de2024profiniteness}. For example, game morphisms correspond to \cite[Definition 3.3]{de2024profiniteness}, subgames (\Cref{def:Subgames}) to generated subframe \cite[Definition 3.7]{de2024profiniteness}, the descripton of colimits to \cite[Proposition 3.9]{de2024profiniteness}, the description of the equalizer to \cite[Lemma 3.11]{de2024profiniteness}, the local presentability to \cite[Theorem 4.3]{de2024profiniteness}, and the comonadicity to \cite[Theorem 5.4]{de2024profiniteness}.
    \[
    \RecCoalg{\Pf}\simeq \Gs \subsetneq \mathbf{KFr}_{lf} \subsetneq \mathbf{KFr}\simeq \Coalg{\Pow}
    \]
\end{remark}

\subsection{\texorpdfstring{$\Gs$}{Games} is locally finitely presentable}
We will prove that the category $\Gs$ is locally finitely presentable. For those who are not familiar with presentability, the standard textbook is \cite{adamek1994locally}. Our proof of the local presentability is very similar to that of \cite{de2024profiniteness}. As we will check the definition na\"{i}vely, we need to know the colimits in $\Gs$ first.
\begin{lemma}\label{lem:CocompletenessAndCreationOfColimitsConservative}
    The forgetful functor $U \colon \Gs \to \Set$ strictly creates all small colimits. In particular, $\Gs$ is cocomplete and $U$ preserves and reflects them. Furthermore, $U$ is conservative.
\end{lemma}
\begin{proof}
The forgetful functor is decomposed into the two forgetful functors
\[
\Gs \xrightarrow{U_0} \Coalg{\Pf}\xrightarrow{U_1} \Set.
\]
The first one $U_0\colon \Gs \to \Coalg{\Pf}$ is fully faithful, and admits a right adjoint of taking \textit{the well-founded part} (\cite[Proposition 4.19]{adamek2020well}). Therefore, $\Gs$ is coreflective subcategory of $\Coalg{\Pf}$ and the embedding functor $U_0$ is conservative and creates all small colimits. The other one $U_1\colon \Coalg{\Pf}\to \Set$ is conservative and creates all small colimits by the general theory of coalgebras of an endofunctor.
\end{proof}

\begin{definition}[Subgame]\label{def:Subgames}
    A subgame of a game $\X = (X, \rel)$ is a subset $S \subset X$ such that if $x\in S$ and $x\rel x'$ then $x' \in S$.
\end{definition}
As we will see later, subgames are almost the same as monomorphisms (\Cref{prop:SubgameAndSubobjectAndMono}).
\begin{lemma}[Image is a subgame]\label{lem:ImageisSubgame}
    For a game morphism $f\colon \X \to \Y$, its image $\Image{f}$ is a subgame of $\Y$.
\end{lemma}
\begin{proof}
    This is due to the path-lifting condition of \Cref{def:GameMorphism}.
\end{proof}
For any injective game morphism $f\colon \Y \to \X$, its image $\Image(f)$ is isomorphic to $\Y$ as games
since $U$ is conservative (\Cref{lem:CocompletenessAndCreationOfColimitsConservative}). Therefore, by abuse of terminology, we will call an injective game morphism a subgame.

\begin{definition}\label{def:SubgameGeneration}
    For a game $\X=(X, \rel)$ and a subset $S\subset X$, the minimum subgame of $\X$ that contains $S$ is called the subgame generated by $S$.
\end{definition}
See \Cref{lem:CogenerationOfSubgames} for its construction.
\begin{lemma}\label{lem:FinitelygeneratedSubgameisFinite}
    Any subgame generated by a finite subset is finite.
\end{lemma}
\begin{proof}
    This follows from the two finiteness conditions in the definition of games (\Cref{def:game}). (cf. \Cref{rem:finiteOptions}).
\end{proof}

\begin{lemma}\label{lem:GamesAreLocallyFinite}
    Every game is the filtered colimit of (the canonical diagram of) its finite subgames.
\end{lemma}
\begin{proof}
    For any finite subset $S\subset X$ of a game $\X=(X, \str)$, the minimum subgame containing $S$ is finite. Therefore the canonical diagram is in fact filtered, and \Cref{lem:CocompletenessAndCreationOfColimitsConservative} completes the proof.
\end{proof}

\begin{lemma}[Finitely presentable $=$ Finite]\label{lem:FinitePresentabilityOfGames}
    A game $\X=(X, \str)$ is finitely presentable if and only if its underlying set $X$ is finite.
\end{lemma}
\begin{proof}
Let us take an arbitrary game $\X=(X,\str)$.
    If a game $\X$ is finitely presentable, \Cref{lem:GamesAreLocallyFinite} implies that $\X$ coincides with its finite subgame, and hence $X$ itself is finite.

    Conversely, assuming that the underlying set $X$ is finite, we will prove that the hom functor
    \[
    \Gs (\X, -)\colon \Gs \to \Set
    \]
    preserves filtered colimits.

    Let $\C$ be a filtered category, $F\colon \C \to \Gs$ be a functor, and $\{\alpha_c \colon Fc \to \Y\}_{c\in \ob{\C}}$ be the colimit cocone. Take an arbitrary morphism $f\colon \X \to \Y$. Our goal is to prove that there exists $c\in \ob{\C}$ such that $f$ has a lift $g$ along $\alpha_c$
    \[
    \begin{tikzcd}
        &Fc\ar[d,"\alpha_c"]\\
        \X\ar[r,"f"']\ar[ru, dashed,"\exists g"]&\Y.
    \end{tikzcd}
    \]
    (The essential uniqueness of the factorization follows from the case of $\Set$ and \Cref{lem:CocompletenessAndCreationOfColimitsConservative}.)

    Since a finite set is finitely presentable in $\Set$, and $U$ preserves small colimits (in particuler, filtered colimits), there exists a function $h \colon U\X \to UFc$ such that the following diagram commutes
    \[
    \begin{tikzcd}
        &UFc\ar[d,"U\alpha_c"]\\
        U\X\ar[r,"Uf"']\ar[ru, dashed,"h"]&U\Y.
    \end{tikzcd}
    \]

    Let $\iota\colon \gS\to Fc$ be the subgame of $Fc$, generated by the image of $h$. Since $\X$ is finite, $\gS$ is also finite (Lemma \Cref{lem:FinitelygeneratedSubgameisFinite}). So far, we have onbtained the diagram
    \[
    \begin{tikzcd}
        &\gS\ar[r,\mono,"\iota"]&Fc\ar[d,"\alpha_c"]\\
        \X\ar[rr,"f"']\ar[ru, dashed,"h"]&&\Y,
    \end{tikzcd}
    \]
    where $h$ is denoted by a dashed arrow since it is a mere function. Since $\Y$ is a filtered colimit (preserved by $U$) and $\gS$ is finite, there exists $k\colon c \to c'$ in $\C$ such that $\alpha_{c'}$ is injective on the subgame $Fk(\gS)=\Image{(Fk \circ \iota)}$ of $Fc'$ (\Cref{lem:ImageisSubgame}). In other words, the morphism $\alpha_{c'}\circ m$ in the following diagram is injective. 
    \[
    \begin{tikzcd}
        &\gS\ar[r,\mono,"\iota"]\ar[d,\epi,"e"']&Fc\ar[dd,bend left, "\alpha_c"]\ar[d,"Fk"']\\
        &Fk(\gS)\ar[r,\mono,"m"]\ar[rd,"\alpha_{c'}\circ m"', \mono]&Fc'\ar[d,"\alpha_{c'}"']\\
        \X\ar[rr,"f"']\ar[ruu, dashed,"h", bend left]&&\Y,
    \end{tikzcd}
    \]
    Because the morphism $\alpha_{c'}\circ m$ is injective, the game $Fk(\gS)$ can be regarded as a subgame of $\Y$. Therefore, the composite $e\circ f$, which is the lift of $f$ along $\alpha_{c'}\circ m$, is a game morphism
    \[
    \begin{tikzcd}
        &\gS\ar[r,\mono,"\iota"]\ar[d,\epi,"e"']&Fc\ar[dd,bend left, "\alpha_c"]\ar[d,"Fk"']\\
        &Fk(\gS)\ar[r,\mono,"m"]\ar[rd,"\alpha_{c'}\circ m"', \mono]&Fc'\ar[d,"\alpha_{c'}"']\\
        \X\ar[rr,"f"']\ar[ru,"e\circ h"]\ar[ruu, dashed,"h", bend left]&&\Y.
    \end{tikzcd}
    \]
    This proves that $f$ has a lift along $\alpha_{c'}$ in the category $\Gs$.
\end{proof}

\begin{theorem}\label{thm:LocallyFinitePresentabilityOfTheGameCategory}
    The category of games $\Gs$ is locally fintiely presentble.
\end{theorem}
\begin{proof}
    We have observed that $\Gs$ is cocomplete (\Cref{lem:CocompletenessAndCreationOfColimitsConservative}) and admits a small number of finitely presentable objects (\Cref{lem:FinitePresentabilityOfGames}), and that every object is a filtered colimit of finitely presentable objects (\Cref{lem:GamesAreLocallyFinite}).
\end{proof}

Although we have spent a whole subsection to prove the existence of the terminal object in \Cref{ssec:terminalGame}, we can prove much more stronger statement now.
\begin{corollary}\label{Cor:CompletenessOfGames}
    The category of games $\Gs$ is complete.
\end{corollary}
\begin{proof}
    This follows since every locally presentable category is complete.
\end{proof}

\begin{corollary}\label{Cor:closedness}
    All the (non-unital or unital) monoidal structure that lifts the cartesian monoidal structure of $\Set$
    \[
    \begin{tikzcd}
        \Gs \times \Gs \ar[r,"\ast"]\ar[d, "U\times U"]&\Gs\ar[d,"U"]\\
        \Set \times \Set \ar[r, "\times"] & \Set
    \end{tikzcd}
    \]
    is monoidal closed.
\end{corollary}
\begin{proof}
    Due to the adjoint functor theorem for locally presentable categories,
    it suffices to prove that, for any game $\X=(X, {\rel})$, the functor $\X\ast {-}\colon \Gs \to \Gs$ is cocontinuous. This follows from the commutative daigram
    \[
    \begin{tikzcd}
        \Gs\ar[r, "\X\ast{-}"]\ar[d, "U"]&\Gs\ar[d, "U"]\\
        \Set\ar[r, "X\times {-}"]&\Set,
    \end{tikzcd}
    \]
     and the fact that the forgetful functor $U\colon \Gs \to \Set$ reflects colimits (\Cref{lem:CocompletenessAndCreationOfColimitsConservative}).
\end{proof}
This corollary implies that all of the Conway addition (\Cref{def:ConwayAddition}), the selective sum (\Cref{exmp:SelctiveSum}), and the conjunctive sum (\Cref{exmp:ConjunctiveSum}) are symmetric monoidal closed structure on $\Gs$. See also \Cref{sssec:ClosedStructure}.

\begin{corollary}\label{cor:rightAdjointofU}
    The forgetful functor $U \colon \Gs \to \Set$ admits a right adjoint.
\end{corollary}
\begin{proof}
    This also follows from the adjoint functor theorem for locally presentable categories and \Cref{lem:CocompletenessAndCreationOfColimitsConservative}.
\end{proof}
This adjunction is the subject of the next subsection.

\subsection{The forgetful \texorpdfstring{$\dashv$}{dashv} cofree game adjunction
}\label{ssec:ForgetfulCofreeAdjunction}
This subsection aims to describe the right adjoint of the forgetful functor $U \colon \Gs \to \Set$, which exsits due to \Cref{cor:rightAdjointofU}. We temporarily write $R\colon \Set \to \Gs$ for the right adjoint functor and call it \demph{the cofree games functor}. We already know that $R(1)$ should be isomorphic to the terminal game $\H$. So we write $RS=(\H_S, \str_S\colon \H_S \to \Pf(\H_S))$ for the cofree game $RS$. Let $\epsilon_{S}\colon UR(S)=\H_S \to S $ denote the counit at a set $S$.

We will observe that we can calculate $\H_S$ as the initial $\Pf({-})\times S$-algebra and hence by Adamek's fixed point theorem (cf. Diagram \ref{eq:TerminalGameHereditarilyFiniteAdamek})! In fact, we can show that the $\Pf({-})\times S$-coalgebra 
\[(\H_S, \langle \str_S, \epsilon_S\rangle\colon \H_S\to \Pf(\H_S)\times S)\]
is the terminal recursive $\Pf({-})\times  S$-coalgebra, and hence given by the initial $\Pf$-algebra (\Cref{prop:TerminalRecursiveCoalgebraIsInitialAlgebra}). 
Notice that a $\Pf({-})\times S$-coalgebra $\langle \str,\sigma \rangle\colon  X \to \Pf(X) \times S$ is recursive if and only if the $\Pf$-coalgebra $\str\colon X\to \Pf(X)$ is recursive\footnote{One can directly prove it. \revmemo{parametrizaed recursion?}}. 
Using this fact, for any recursive $\Pf({-})\times S$-coalgebra $\langle \str,\sigma \rangle\colon  X \to \Pf(X) \times S$,
the transpose of $\sigma$, which is denoted by $\sigma^{\flat}\colon \X=(X, \str) \to (\H_S,\str_S)=RS$, is the unique $\Pf({-})\times S$-coalgebra homomorphism:
\[
\begin{tikzcd}
    X\ar[r, "{\langle \str,\sigma \rangle}"]\ar[d, "\sigma^{\flat}"]& \Pf(X) \times S \ar[d, "\Pf(\sigma^{\flat}) \times \id_S"]\\
    \H_S \ar[r, "{\langle \str_S,\epsilon_S \rangle}"]& \Pf(\H_S) \times S.
\end{tikzcd}
\]
This proves that $(\H_S, \langle \str_S, \epsilon_S\rangle\colon \H_S\to \Pf(\H_S)\times S)$ is the terminal recursice $\Pf({-})\times S$-coalgebra, and hence given by Adamek's fixed point theorem:
\begin{equation}\label{eq:CofreeGamesHereditarilyFiniteAdamek}
    \begin{tikzcd}
            \emptyset \ar[r]& \Pf(\emptyset)\times S \ar[r]& \Pf(\Pf(\emptyset)\times S) \times S\ar[r]&\Pf(\Pf(\Pf(\emptyset)\times S) \times S)\times S \ar[r]&\cdots 
            \H_S 
        \end{tikzcd}
\end{equation}

By the above observations, we obtain the following set-theoretic realization of $\H_S$.
\begin{definition}[Labeled hereditarily finite sets]\label{def:LabeledHFS}
    For a set (of labels) $\Lambda$, a \demph{$\Lambda$-labeled hereditarily finite set} is recursively defined as a pair $(S,\lambda)$ of a finite set $S$ of $\Lambda$-labeled hereditarily finite sets and a label $\lambda \in \Lambda$.
    In other words, the set of all $\Lambda$-labeled hereditarily finite sets $\H_\Lambda$ is defined to be
    \[
    \H_\Lambda = \bigcup_{k=0}^{\infty} \H_{\Lambda,n},
    \]
    where $\H_{\Lambda,0}= \emptyset$ and $\H_{\Lambda,n+1}= \Pf(\H_{\Lambda,n}) \times \Lambda$.
\end{definition}

Just like a hereditarily finite set is a set with a finite number of parentheses, a $\Lambda$-labeled hereditarily finite set is a \dq{$\Lambda$-labeled set with finite number of $\Lambda$-labeled parentheses.} For example, if $\Lambda=\{{\color{red}R},{\color{green!60!black}G}, {\color{blue}B}\}$, a typical element of $\H_{\Lambda}$ looks like
\begin{equation}\label{eq:RGBhereditarilyfiniteset}
    \pg{\pg{}\pb{\pr{\pr{}\pg{}\pb{}}\pg{\pb{\pr{}}}}} \in \H_{\{{\color{red}R},{\color{green!60!black}G}, {\color{blue}B}\}}.
\end{equation}

\begin{proposition}[Forgetful-Cofree adjunction]\label{prop:ForgetfulCofreeadjunction}
The right adjoint $R \vdash U$ 
    \[\ADJ{\Set}{R}{\Gs}{U}\]
    is given by $R(\Lambda) =(\H_\Lambda, \str_{\Lambda}\colon \H_{\Lambda}\to \Pf{(\H_{\Lambda})})$, where the structure map $\str_{\Lambda}$ sends each element $x=(S, \lambda)\in \Pf(\H_{\Lambda})\times \Lambda = \H_{\Lambda}$ to $\str_{\Lambda}(x)=S\in \Pf(\H_{\Lambda})$. The counit $\epsilon_{\Lambda}\colon \H_{\Lambda} \to \Lambda$ sends each element $x=(S, \lambda)$ to its label $\lambda\in \Lambda$.
\end{proposition}

We will utilize this description of cofree games in the construction of the subobject classifier (\Cref{ssec:SubobjectClassifier}).

\subsection{Some limits in \texorpdfstring{$\Gs$}{Games}}
We have observed that the category $\Gs$ has all small limits (\Cref{Cor:CompletenessOfGames}). The only example of small limits we have seen so far is the terminal object (\Cref{prop:TerminalGameAndInitialAlgebraAreHereditarilyFiniteSets}).
In this subsection, we will calculate some other finite limits.

\subsubsection{Inverse image}
The first example is the inverse images of subgames.
\begin{lemma}\label{lem:InverseImageOfGames}
    For any game morphism $f\colon \X \to \Y$ and a subgame $S\subset Y$, its inverse image $f^{-1}(S)\subset X$ is also a subgame of $\X$, and the diagram
    \[
    \begin{tikzcd}
        f^{-1}(S) \ar[r] \ar[d, rightarrowtail]&S \ar[d, rightarrowtail]\\
        \X\ar[r, "f"] & \Y
    \end{tikzcd}
    \]
    is pullback in $\Gs$.
\end{lemma}
\begin{proof}
One can check it directly.
    \revmemo{write: taut comonad}
\end{proof}

\subsubsection{Monomorphisms}
The second example of limits is monomorphisms, which is a morphism $f$ whose pullback along itself $f$ is given by two identity maps. The observations of monomorphisms in this subsubsection will be utilized to study the subobject classifier (\Cref{ssec:SubobjectClassifier}).

Just to avoid the following argument becoming wordy, we introduce an accessibility relation in an obvious way:
\begin{definition}[Accessibility relation]\label{DefinitionAccessibility}
    For a game $\X=(X, \rel)$, we say that an element $x\in X$ is \demph{accessible} from $x'$ if there exists a non-negative integer $n \in \N$ and a sequence of elements $x_0, \dots x_n$ that satisfy
    \begin{itemize}
        \item $x_0=x'$,
        \item $x_i \rel x_{i+1}$ for $0\leq i < n$, and
        \item $x_n=x$.
    \end{itemize}
    This accessibility relation is denoted by $x' \acc x$.
\end{definition}
We write $x' \accneq x$ for $(x'\acc x) \land (x'\neq x)$.

\begin{remark}[Subgames are downward closed subset]
    This accessibility relation is just the reflective and transitive closure of $\rel$, and defines a preorder on the underlying set $X$. Furthermore, due to the \dq{finite time} condition in \Cref{def:game}, it is a partial order.
    The notion of subgames (\Cref{def:Subgames}) coincides with the notion of downward closed subsets of the poset.
\end{remark}

\begin{proposition}[Subgames $=$ Subobjects]\label{prop:SubgameAndSubobjectAndMono}
    For any game morphism $f\colon \X \to \Y$, the following conditions are equivalent:
    \begin{enumerate}
        \item $f$ is monic in $\Gs$. \label{ConditionMonic}
        \item $f$ is injective. \label{ConditionInjective}
        \item $f$ is (canonically isomorphic to) a subgame inclusion.\label{ConditionSubobject}
    \end{enumerate}
\end{proposition}
\begin{proof}
    For the equivalence between $(\ref{ConditionInjective})$ and $(\ref{ConditionSubobject})$, see the comment below \Cref{lem:ImageisSubgame}.
    The implication
    $(\ref{ConditionInjective}) \implies (\ref{ConditionMonic})$ immediately follows since the forgetful fucntor $U\colon \Gs \to \Set$ is faithful.

    We prove $(\ref{ConditionMonic})\implies (\ref{ConditionInjective})$. Suppose $f$ is monic. We prove that $\# f^{-1}(y) \leq 1$ by induction on the well-founded order structure $(Y,\acc)$. Assuming that $\# f^{-1}(y') \leq 1$ holds for any $y' \prec y$, we prove $\# f^{-1}(y) \leq 1$. 
    We can assume $\# f^{-1}(y) \geq 1$ since otherwise the target inequality $\# f^{-1}(y) \leq 1$ trivially holds. Assuming $\# f^{-1}(y) \geq 1$, the image of $f$ includes $y$ and hence all of $y'\prec y$ (\Cref{lem:ImageisSubgame}), so the induction hypothesis implies 
    \begin{equation}\label{eq:uniqueYdash}
        \# f^{-1}(y') =1\text{ for any }y'\prec y.
    \end{equation}
     Let us define $T\subset X$ as
    \[
    T \coloneqq \{x\in X\mid f(x)\prec y\}.
    \]
    Since $\{y'\in Y\mid y'\prec y\}$ is a subgame of $\Y$, \Cref{lem:InverseImageOfGames} implies that $T$ is also a subgame of $\X$. Furthermore, $T$ is a finite game due to \Cref{eq:uniqueYdash}. 
    
    Adding a new element $\ast$ to $T$, we extend the finite game $T$ into a new game $\W=(T \coprod \{\ast\}, \rel_{\W})$. In addition to the relations in $T$, we add
    a relation $\ast \rel_{\W} t$ for each $t\in T$ such that $y \rel_{\Y} f(t)$.
    \[
    \ast \rel_{\W} t \iff y\rel_{\Y} f(t)
    \]
    One can prove $\W$ is actually a game, in particular satisfies the finite option condition, due to the finiteness of $T$. Then, for any $x\in f^{-1}(y)$, the function $g_{x} \colon \W \to \X$ defined by
    \[
    g_{x}(w) =
    \begin{cases}
        t &(w=t\in T)\\
        x & (w= \ast)
    \end{cases}
    \]
    is a game morphism. It is because, for any $x'\in \X$,
    \begin{align*}
        x \rel_{\X} x' &\iff y= f(x) \rel_{\Y} f(x')\\
        &\iff x'\in T \textrm{ and }y \rel_{\Y} f(x')\\
        &\iff x'\in T \textrm{ and }\ast \rel_{\W} x',
    \end{align*}
    where the first inverse implication $x\rel_{\X} x' \impliedby f(x)\rel_{\Y}f(x')$ holds due to the following argument: Assuming $f(x)\rel_{\Y}f(x')$, there exists $x\rel_{\X}x''\in X$ such that $f(x'')=f(x')$ since $f$ satisfies the path-lifting property (\Cref{def:GameMorphism}). Due to \Cref{eq:uniqueYdash}, we have $\#f^{-1}(f(x'))=1$, and hence $x''=x'$.
    
    
    For any (possibly non-distinct) $x_0, x_1 \in f^{-1}(y)$, we have a diagram
    \[
    \begin{tikzcd}
        \W\ar[r,"g_{x_0}", shift left]\ar[r,"g_{x_1}"',shift right]&\X\ar[r,"f"]&\Y,
    \end{tikzcd}
    \]
    with the same composite.
    Since we assumed that $f$ is monic, we obtain $g_{x_0} = g_{x_1}$ and $x_0 = x_1$. This completes the proof.
\end{proof}

As $U$ is a faithful left adjoint, a game morphism $f$ is an epimorphism if and only if $f$ is surjective. As a corollary, the decomposition of a game morphism $f\colon \X \to \Y$ through the image game $\X \twoheadrightarrow \Image{f}\rightarrowtail \Y$ provides the epi-mono factorization system of $\Gs$.

\begin{corollary}[Epi-mono orthogonal factorization]\label{cor:Epi-MonoFactorization}
    Any game morphism is uniquely factored into a composition of an epimorphism followed by a monomorphism. This data defines an orthogonal factorization system on $\Gs$.
\end{corollary}

\begin{remark}
    In fact, this factorization system coincides with the (Epi, Regular mono) factorization system on the coregular category $\Gs$. \revmemo{write more details}
\end{remark}

\subsubsection{Equalizer}\label{SubsubsectionEqualizer}
The third example is equalizers.
\begin{lemma}[Generation and Cogeneration of subgames]\label{lem:CogenerationOfSubgames}
    For a game $\X=(X, \rel)$ and a subset $S\subset X$, 
    \begin{itemize}
        \item there exists the minimum subgame of $\X$ that contains $S$.
        \item there exists the maximum subgame of $\X$ that is contained by $S$.
    \end{itemize} 
\end{lemma}
\begin{proof}
    This is an immediate corollary of the general adjoint functor theorem applied to the complete lattice inclusion from the lattice of subgames into the lattice of subsets.

    Explicitly, the former subgame is constructed as
    \[\{x\in X\mid \exists s \in S, \ s\acc x\},\]
    and the latter is 
    \[\{x\in X\mid \forall y \in X,\  (x\acc y \implies y \in S)\}\]
\end{proof}

The next lemma is essentially the same as \cite[Lemma3.11]{de2024profiniteness}.
\begin{lemma}[Equalizer in $\Gs$]\label{lem:equalizer}
    For any parallel pair of game morphisms $f,g\colon \X \rightrightarrows \Y$, its equalizer $E_{f,g}\rightarrowtail \X \rightrightarrows \Y$ is given by 
    \[
    E_{f,g}=\{x\in X\mid \forall x'\preceq x, f(x')=g(x')\}.
    \]
\end{lemma}
\begin{proof}
    For any game morphism $h\colon \W \to \X$ such that $f\circ h =g \circ h$, the image $\Image{h}$ is a subgame contained by the set-theoretic equalizer $E_0=\{x\in X\mid f(x) =g(x)\}$. Since $E_{f,g}$ is the maximum subgame contained by $E_0$ in the sense of \Cref{lem:CogenerationOfSubgames}, we have $\Image{h} \subset E_{f,g}$, and $h$ factors through the subgame inclusion $E_{f,g}\rightarrowtail \X$. The uniqueness part of the universality trivially holds since $E_{f,g}\rightarrowtail \X$ is injective.
\end{proof}

\subsubsection{Digression: General limits and Stirling numbers}
As $U$ is not continuous, it is not easy to calculate genral limits in $\Gs$. However, there is a (non-efficient) way to calculate limits using the cofree games functor (\Cref{ssec:ForgetfulCofreeAdjunction}):
Take an arbitrary small diagram $F\colon  J \to \Gs$. Due to \Cref{Cor:CompletenessOfGames}, we know that $\lim F$ exists. So we have
\begin{align*}
    \Gs(\X, \lim F) \cong  & \mathbf{Cone}_{\Gs}(\X, F)\\
    \subset & \mathbf{Cone}_{\Set}(U\X, UF)\\
    \cong  & \Set(U\X, \lim UF)\\
    \cong &\Gs(\X, R\lim UF),
\end{align*}
where $R$ denotes the cofree game functor, which is the right adjoint to $U$ (\Cref{prop:ForgetfulCofreeadjunction}). Therefore, we conclude that $\lim F$ is a subgame of $R\lim U F$, whose underlying set is given by the $\lim UF$-labeled hereditarily finite set $\H_{\lim UF}$. As this observation works for any faithful left adjoint functor $U$, it loses a lot of information and it is not always easy to figure out which subgame is the limit.

Let us demonstrate how general limits are difficult to calculate. For a positive integer $n>0$, let $\Star_n$ denote the $(n+1)$-state game that looks like \Cref{fig:TheStarGame}.
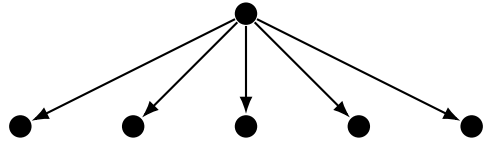
\begin{figure}[ht]
\centering
\begin{tikzpicture}[>=Latex, thick, scale=1.0]

  \def\n{5} 
  \def\dx{1.5} 
  \def\yA{1.5} 
  \def\yB{0}   

  \node[circle, fill=black, inner sep=3pt] (A) at (0,\yA) {};

  \foreach \i in {1,...,\n} {
    \pgfmathsetmacro{\x}{(\i - (\n+1)/2)*\dx}
    \node[circle, fill=black, inner sep=3pt] (B\i) at (\x,\yB) {};
    \draw[->] (A) -- (B\i);
  }
\end{tikzpicture}
\caption{The game $\Star_n$ for $n=5$}
\label{fig:TheStarGame}
\end{figure}

Then, the product game $\Star_2 \times \Star_3$ looks like \Cref{fig:TheProductGame}.
\begin{figure}[ht]
\centering
\begin{tikzpicture}[>=Latex, thick, scale=0.9]
  \tikzset{
    edgeThree/.style={->, draw=red!70, thin},
    edgeFour/.style={->, draw=green!60!black, thin},
    edgeFive/.style={->, draw=blue!70, thin},
    edgeSix/.style={->, draw=black, thin}
  }


  \node[circle, fill=black, inner sep=2pt] (B00) at (-5.00, -4) {};
  \node[circle, fill=black, inner sep=2pt] (B01) at (-3.00, -4) {};
  \node[circle, fill=black, inner sep=2pt] (B02) at (-1.00, -4) {};
  \node[circle, fill=black, inner sep=2pt] (B10)  at (1.00, -4) {};
  \node[circle, fill=black, inner sep=2pt] (B11)  at (3.00, -4) {};
  \node[circle, fill=black, inner sep=2pt] (B12)  at (5.00, -4) {};

  \node[circle, fill=black, inner sep=2pt] (A31) at (-4.00, 2.20) {};
  \node[circle, fill=black, inner sep=2pt] (A32) at (-2.40, 2.20) {};
  \node[circle, fill=black, inner sep=2pt] (A33) at (-0.80, 2.20) {};
  \node[circle, fill=black, inner sep=2pt] (A34) at (0.80, 2.20) {};
  \node[circle, fill=black, inner sep=2pt] (A35) at (2.40, 2.20) {};
  \node[circle, fill=black, inner sep=2pt] (A36) at (4.00, 2.20) {};
  \node[circle, fill=black, inner sep=2pt] (A41) at (-6.05, 3.70) {};
  \node[circle, fill=black, inner sep=2pt] (A42) at (-4.95, 3.70) {};
  \node[circle, fill=black, inner sep=2pt] (A43) at (-3.85, 3.70) {};
  \node[circle, fill=black, inner sep=2pt] (A44) at (-2.75, 3.70) {};
  \node[circle, fill=black, inner sep=2pt] (A45) at (-1.65, 3.70) {};
  \node[circle, fill=black, inner sep=2pt] (A46) at (-0.55, 3.70) {};
  \node[circle, fill=black, inner sep=2pt] (A47) at (0.55, 3.70) {};
  \node[circle, fill=black, inner sep=2pt] (A48) at (1.65, 3.70) {};
  \node[circle, fill=black, inner sep=2pt] (A49) at (2.75, 3.70) {};
  \node[circle, fill=black, inner sep=2pt] (A410) at (3.85, 3.70) {};
  \node[circle, fill=black, inner sep=2pt] (A411) at (4.95, 3.70) {};
  \node[circle, fill=black, inner sep=2pt] (A412) at (6.05, 3.70) {};
  \node[circle, fill=black, inner sep=2pt] (A51) at (-4.00, 5.20) {};
  \node[circle, fill=black, inner sep=2pt] (A52) at (-2.40, 5.20) {};
  \node[circle, fill=black, inner sep=2pt] (A53) at (-0.80, 5.20) {};
  \node[circle, fill=black, inner sep=2pt] (A54) at (0.80, 5.20) {};
  \node[circle, fill=black, inner sep=2pt] (A55) at (2.40, 5.20) {};
  \node[circle, fill=black, inner sep=2pt] (A56) at (4.00, 5.20) {};
  \node[circle, fill=black, inner sep=2pt] (A61) at (0.00, 6.70) {};

  \draw[edgeThree] (A31) -- (B00);
  \draw[edgeThree] (A31) -- (B01);
  \draw[edgeThree] (A31) -- (B12);
  \draw[edgeThree] (A32) -- (B00);
  \draw[edgeThree] (A32) -- (B02);
  \draw[edgeThree] (A32) -- (B11);
  \draw[edgeThree] (A33) -- (B00);
  \draw[edgeThree] (A33) -- (B11);
  \draw[edgeThree] (A33) -- (B12);
  \draw[edgeThree] (A34) -- (B01);
  \draw[edgeThree] (A34) -- (B02);
  \draw[edgeThree] (A34) -- (B10);
  \draw[edgeThree] (A35) -- (B01);
  \draw[edgeThree] (A35) -- (B10);
  \draw[edgeThree] (A35) -- (B12);
  \draw[edgeThree] (A36) -- (B02);
  \draw[edgeThree] (A36) -- (B10);
  \draw[edgeThree] (A36) -- (B11);
  \draw[edgeFour] (A41) -- (B00);
  \draw[edgeFour] (A41) -- (B01);
  \draw[edgeFour] (A41) -- (B02);
  \draw[edgeFour] (A41) -- (B10);
  \draw[edgeFour] (A42) -- (B00);
  \draw[edgeFour] (A42) -- (B01);
  \draw[edgeFour] (A42) -- (B02);
  \draw[edgeFour] (A42) -- (B11);
  \draw[edgeFour] (A43) -- (B00);
  \draw[edgeFour] (A43) -- (B01);
  \draw[edgeFour] (A43) -- (B02);
  \draw[edgeFour] (A43) -- (B12);
  \draw[edgeFour] (A44) -- (B00);
  \draw[edgeFour] (A44) -- (B01);
  \draw[edgeFour] (A44) -- (B10);
  \draw[edgeFour] (A44) -- (B12);
  \draw[edgeFour] (A45) -- (B00);
  \draw[edgeFour] (A45) -- (B01);
  \draw[edgeFour] (A45) -- (B11);
  \draw[edgeFour] (A45) -- (B12);
  \draw[edgeFour] (A46) -- (B00);
  \draw[edgeFour] (A46) -- (B02);
  \draw[edgeFour] (A46) -- (B10);
  \draw[edgeFour] (A46) -- (B11);
  \draw[edgeFour] (A47) -- (B00);
  \draw[edgeFour] (A47) -- (B02);
  \draw[edgeFour] (A47) -- (B11);
  \draw[edgeFour] (A47) -- (B12);
  \draw[edgeFour] (A48) -- (B00);
  \draw[edgeFour] (A48) -- (B10);
  \draw[edgeFour] (A48) -- (B11);
  \draw[edgeFour] (A48) -- (B12);
  \draw[edgeFour] (A49) -- (B01);
  \draw[edgeFour] (A49) -- (B02);
  \draw[edgeFour] (A49) -- (B10);
  \draw[edgeFour] (A49) -- (B11);
  \draw[edgeFour] (A410) -- (B01);
  \draw[edgeFour] (A410) -- (B02);
  \draw[edgeFour] (A410) -- (B10);
  \draw[edgeFour] (A410) -- (B12);
  \draw[edgeFour] (A411) -- (B01);
  \draw[edgeFour] (A411) -- (B10);
  \draw[edgeFour] (A411) -- (B11);
  \draw[edgeFour] (A411) -- (B12);
  \draw[edgeFour] (A412) -- (B02);
  \draw[edgeFour] (A412) -- (B10);
  \draw[edgeFour] (A412) -- (B11);
  \draw[edgeFour] (A412) -- (B12);
  \draw[edgeFive] (A51) -- (B00);
  \draw[edgeFive] (A51) -- (B01);
  \draw[edgeFive] (A51) -- (B02);
  \draw[edgeFive] (A51) -- (B10);
  \draw[edgeFive] (A51) -- (B11);
  \draw[edgeFive] (A52) -- (B00);
  \draw[edgeFive] (A52) -- (B01);
  \draw[edgeFive] (A52) -- (B02);
  \draw[edgeFive] (A52) -- (B10);
  \draw[edgeFive] (A52) -- (B12);
  \draw[edgeFive] (A53) -- (B00);
  \draw[edgeFive] (A53) -- (B01);
  \draw[edgeFive] (A53) -- (B02);
  \draw[edgeFive] (A53) -- (B11);
  \draw[edgeFive] (A53) -- (B12);
  \draw[edgeFive] (A54) -- (B00);
  \draw[edgeFive] (A54) -- (B01);
  \draw[edgeFive] (A54) -- (B10);
  \draw[edgeFive] (A54) -- (B11);
  \draw[edgeFive] (A54) -- (B12);
  \draw[edgeFive] (A55) -- (B00);
  \draw[edgeFive] (A55) -- (B02);
  \draw[edgeFive] (A55) -- (B10);
  \draw[edgeFive] (A55) -- (B11);
  \draw[edgeFive] (A55) -- (B12);
  \draw[edgeFive] (A56) -- (B01);
  \draw[edgeFive] (A56) -- (B02);
  \draw[edgeFive] (A56) -- (B10);
  \draw[edgeFive] (A56) -- (B11);
  \draw[edgeFive] (A56) -- (B12);
  \draw[edgeSix] (A61) -- (B00);
  \draw[edgeSix] (A61) -- (B01);
  \draw[edgeSix] (A61) -- (B02);
  \draw[edgeSix] (A61) -- (B10);
  \draw[edgeSix] (A61) -- (B11);
  \draw[edgeSix] (A61) -- (B12);
\end{tikzpicture}
\caption{The pruduct game $\Star_2\times \Star_3$, which categorifies\\ $2!S(n,2)\times 3!S(n,3) = 1\cdot 6!S(n,6)+6\cdot 5!S(n,5)+ 12\cdot4!S(n,4)+6\cdot 3!S(n,3)$.}
\label{fig:TheProductGame}
\end{figure}
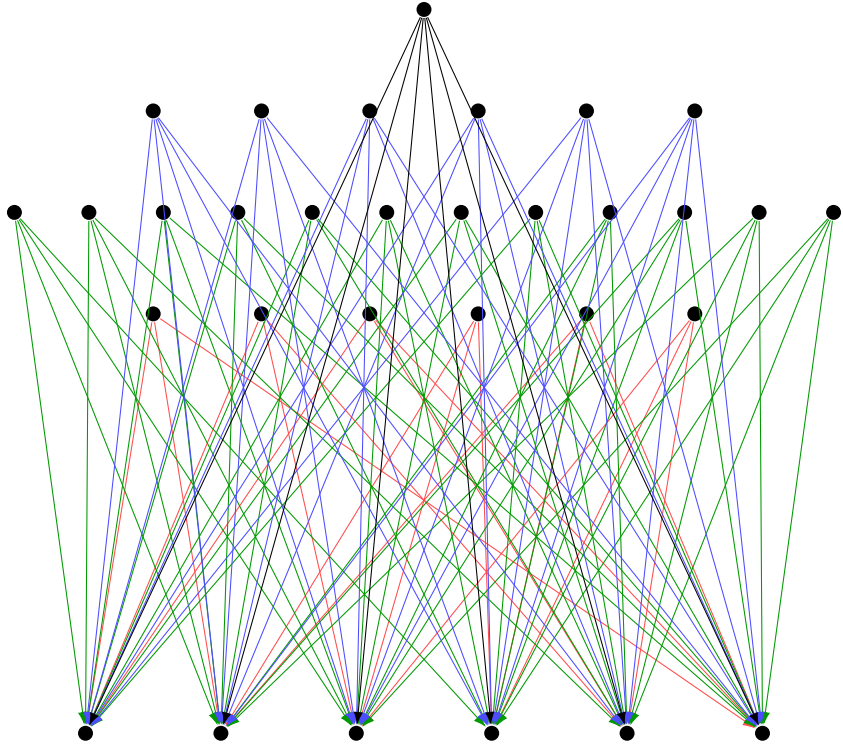 What is happning there? One interpretation is given by regarding $\Star_k$ as an categorification of \demph{the Stirling number} of the second kind! 

One can easily prove that $\#\Gs(\Star_n, \Star_k)$ is the number of surjections from an $n$-element set to an $m$-element set, which is known to be equal to $k! S(n,k)$, where $S(n,k)$ is the Stirling number of the second kind. 
\[
\#\Gs(\Star_n, \Star_k)=k!S(n,k)
\]
Therefore, the product game $\Star_2\times \Star_3$ should have information about $2!S(n,2)\times 3! S(n, 3)$ since the following equality holds.
\[
\Gs(\Star_n, \Star_2\times \Star_3) \cong \Gs(\Star_n, \Star_2)\times \Gs(\Star_n, \Star_3)
\]
In fact, the product game (or, more precisely, the isomorphism to its concrete description \Cref{fig:TheProductGame}) is a categorification of the multiplicative identity of the Stirling number:
\[
2!S(n,2)\times 3!S(n,3) = 1\cdot 6!S(n,6)+6\cdot 5!S(n,5)+ 12\cdot4!S(n,4)+6\cdot 3!S(n,3).
\]
Notice that these coefficients $1,6,12,6$ appear in \Cref{fig:TheProductGame}.

Even the inocent-looking example $\Star_n\times \Star_m$ is interwinded with the combinatorics of the Stirling numbers.
By this example, we would like to emphasize the difficulty of the calculation of general limits in $\Gs$.

\subsection{\texorpdfstring{$\Gs$ is comonadic over $\Set$}{Games is comonadic over Set}}

The following proof is essentially the same as the comonadicity of locally finite Kriple frames \cite[Theorem 5.4]{de2024profiniteness}.
\begin{proposition}\label{prop:Comonadicity}
    The forgetful functor $U\colon \Gs \to \Set$ is comonadic.
\end{proposition}
\begin{proof}
    We use the reflexive tripleability theorem (see \cite[Propotition 5.5.8]{riehl2017category}). As we have already proven that $U$ is a consevative left adjoint functor (\Cref{lem:CocompletenessAndCreationOfColimitsConservative} and \Cref{cor:rightAdjointofU}) and $\Gs$ is complete (\Cref{Cor:CompletenessOfGames}), it suffices to prove that $U$ preserves equalizers of coreflexive pairs.

    Take an arbitrary coreflexive pair $f,g\colon \X\rightrightarrows\Y$ with a common retract $r\colon \Y \to \X$ such that $r\circ f = r\circ g = \id_{\X}$. Then, due to \Cref{lem:equalizer}, the equalizer is given by $E_{f,g} =\{x\in X\mid \forall x' \preceq x,\; f(x') =g(x')\}$. We should prove that this coincides with the set-theoretic equalizer $\{x\in X\mid f(x)=g(x)\}$.

    Therefore, it suffices to prove that, for any $x\rel_{\X}x'$, $f(x) =g(x)$ implies $f(x')=g(x')$. Assuming $f(x)=g(x)$ and  $x\rel_{\X}x'$, we can take $x\rel_{\X} x''$ such that $f(x'')=g(x')$ by applying the path-lifting property to $f(x)=g(x)\rel_{\Y} g(x')$. By applying the common retract $r$ to the equation $f(x'')=g(x')$, we obtain $x''=x'$, which implies $f(x')=f(x'')=g(x')$. This completes the proof.
\end{proof}

\subsection{\texorpdfstring{$\Gs$}{Games} has a subobject classifier}\label{ssec:SubobjectClassifier}
In this subsection, we will give an explicit description of the subobject classifier of the category of games.

Our idea of the construction is simple: utilize the cofree-forgetful adjunction (\Cref{prop:ForgetfulCofreeadjunction}). As subobjects are just subsets satisfying some properties (\Cref{prop:SubgameAndSubobjectAndMono}), we have the following canonical injection 
\begin{equation}\label{eq:subsetClassification}
    \mathrm{Sub}_{\Gs}(\X) \subset \Pow(X) \cong \Set(X,\{\top,\bot\}) \cong \Gs(\X,\H_{\{\top,\bot\}}).
\end{equation}
Therefore, by the Yoneda lemma, if a subobject classifier $\Omega$ exists, then it should be a subobject of the game of truth-values-labeled hereditarily finite sets $\Omega \subset \H_{\{\top,\bot\}}$.

Let us recall how the correspondence $\Pow(X) \cong \Set(X,\{\top,\bot\}) \cong \Gs(\X,\H_{\{\top,\bot\}})$ in (\ref{eq:subsetClassification}) works. Let $\X=(X, \str)$ be a game and $S\subset X$ be a subset. Then, the corresponding game morphism $\chi_{S}\colon \X \to \H_{\{\top, \bot\}}$ sends an element $x\in X$ to
\[
\chi_{S}(x) = \left(\{\chi_{S}(x')\mid x\rel_{\X} x'\}, x\in S\right) \in \Pf\left(\H_{\{\top, \bot\}}\right)\times \{\top, \bot\} =\H_{\{\top, \bot\}},
\]
where ``$x\in S$" denotes the element of $\{\top, \bot\}$ corresponding to the truth value of the proposition ``$x\in S$."
Conversely, for a given game morphism $f\colon \X \to \H_{\{\top, \bot\}}$, the corresponding subset is recovered by taking $\{x\in X\mid \text{the label of $f(x)$ is $\top$}\}$.
An example of such $\chi_{S}$ (in the case where $S$ is a subgame of $\X$) is visualized in \Cref{fig:SubobjectClassification}, where ${\color{red}\top}$ and $\bot$ are colored like (\ref{eq:RGBhereditarilyfiniteset}).
\begin{figure}[ht]
\centering
\begin{tikzpicture}[>=Latex, thick, scale=1.0]

  \tikzset{
    v/.style   ={circle, fill=black, inner sep=3pt},    
    vred/.style={circle, fill=red, inner sep=3pt},
    e/.style   ={->, draw=black, line cap=round, very thick},
    ered/.style={->, draw=red, line cap=round, very thick},
    lab/.style ={font=\small, inner sep=1pt, outer sep=2pt, text=black},
    labr/.style={lab, text=red}
  }

  \node[v] (r)   at (0,3)  {};
  \node[above=1mm of r] {$\pd{\pr{\pr{\pr{}}}\pd{\pr{\pr{}}\pd{\pr{}\pd{}}}}$};
  \node[v] (c)   at (2.5,1) {};
  \node[right=1mm of c] {$\pd{\pr{\pr{}}\pd{\pr{}\pd{}}}$};
  \node[v] (g)    at (3.75,-1) {};
  \node[right=1mm of g] {$\pd{\pr{}\pd{}}$};
  \node[v] (f1)  at (5,-3) {};
  \node[right=1mm of f1] {$\pd{}$};

  \node[vred] (a)   at (-2.5,1) {};
  \node[left=1mm of a] {$\pr{\pr{\pr{}}}$};
  \node[vred] (l2) at (-3.75,-1) {};
  \node[left=1mm of l2] {$\pr{\pr{}}$};
  \node[vred] (l3) at (-5,-3) {};
  \node[left=1mm of l3] {$\pr{}$};
  \node[vred] (l4) at (-2.5,-3) {};
  \node[right=1mm of l4] {$\pr{}$};
  \node[vred] (m1) at (0,-1)  {};
  \node[left=1mm of m1] {$\pr{\pr{}}$};
  \node[vred] (m2) at (0,-3) {};
  \node[left=1mm of m2] {$\pr{}$};
  \node[vred] (rr) at (2.5,-3) {};
  \node[right=1mm of rr] {$\pr{}$};

  \draw[e] (r) -- (a);
  \draw[e] (r) -- (c);
    \draw[e] (c) -- (m1);
    \draw[e]   (c) -- (g);
        \draw[e]   (g) -- (f1);
        \draw[e]   (g) -- (rr);
  \draw[e] (r) -- (m1);

  \draw[ered] (a) -- (l2);
    \draw[ered] (l2) -- (l3);
    \draw[ered] (l2) -- (l4);  
  \draw[ered] (m1) -- (m2);
  \draw[ered] (m1) -- (rr);
\end{tikzpicture}
\caption{The characteristic map $\chi_{S}$ for {\color{red} a subgame $S$}}
\label{fig:SubobjectClassification}
\end{figure}
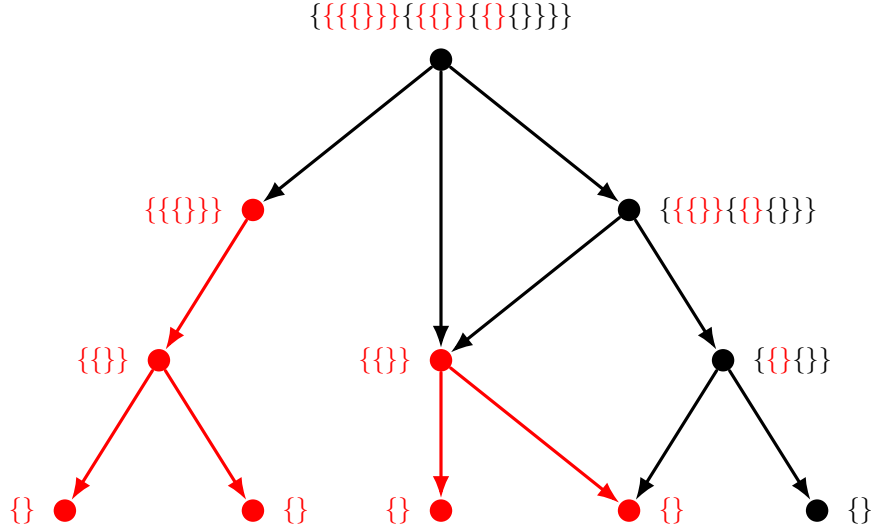

\begin{definition}\label{def:TruthClosed}
    A $\{\top, \bot\}$-labeled hereditarily finite set $A\in \H_{\{\top,\bot\}}$ is \demph{truth-closed} if 
    \begin{enumerate}
        \item if $A$ itself is labeled by $\top$, then every element of $A$ is also labeled by $\top$, and
        \item every element of $A$ is truth-closed.
    \end{enumerate}
\end{definition}
Due to condition (2), truth-closed $\{\top, \bot\}$-labeled hereditarily finite sets form a subgame of $\H_{\{\top,\bot\}}$. By coloring like (\ref{eq:RGBhereditarilyfiniteset}), the typical examples of a truth-closed $\{{\color{red}\top}, \bot\}$-labeled set include $\pr{}$, $\pd{}$, $\pd{\pd{\pd{\pr{\pr{}}}}}$, and $\pd{\pd{\pd{\pr{}}}\pr{\pr{}\pr{\pr{}}}\pd{\pd{}}}$. Non-examples include $\pr{\pd{}}$ and $\pd{\pd{}\pr{\pd{\pr{}}}}$. Intuitively, a $\{{\color{red}\top}, \bot\}$-labeled hereditarily finite set is truth-closed if and only if no black parenthesis appears within any red parenthesis.

\begin{notation}
The subgame of $\H_{\{\top,\bot\}}$ consisting of all truth-closed $\{\top,\bot\}$-labeled hereditarily finite sets is denoted by $\Omega$.
\end{notation}
The game $\Omega$ is visualized in \Cref{fig:SubobjectClassifier}.

\begin{figure}[ht]
\centering
\begin{tikzpicture}[>=Latex, thick, node distance=20mm, scale=0.16]

  \usetikzlibrary{fadings}
  \tikzfading[name=fade down, top color=transparent!0, bottom color=transparent!100]
  \tikzset{
    thfs/.style={circle, draw=black, inner sep=2pt, minimum size=8mm, align=center, very thick},
    Rthfs/.style={circle, draw=red, inner sep=2pt, minimum size=8mm, align=center, very thick},
    edge/.style={->, draw=black!70, line cap=round, very thin},
    fedge/.style={->, draw=black!70, line cap=round, very thin, path fading=fade down}
  }
  
  \newcommand{\XA}{\pd{}}%
  \newcommand{\XB}{\pd{\pd{}}}%
  \newcommand{\XC}{\pd{\pr{}}}%
  \newcommand{\XD}{\pd{\pd{}\pr{}}}%
  \newcommand{\XE}{\pr{}}%
  \newcommand{\XF}{\pr{\pr{}}}%

  \def\dx{5.0}
  \def\dy{20.0}
  \def\yBzero{\dy}
  \def\yBone{2*\dy}
  \def\yBtwo{3*\dy}
  \def\yBthr{4*\dy}
  \def\yBfor{5*\dy}
  \def\yBfiv{6*\dy}
  \def\yBsix{7*\dy}
  \def\yRed{-1*\dy}
  \def\fsize{\tiny}

  \newcommand{\Snode}[5]{%
    \node[thfs] (#3) at (#1,#2) {\tiny $#4$};
    \foreach \m in {#5} { \draw[edge, black] (#3) -- (\m); }%
  }
  \newcommand{\Fnode}[5]{%
    \node[thfs] (#3) at (#1,#2) {\tiny $#4$};
    \foreach \m in {#5} { \draw[fedge, black] (#3) -- (\m); }%
  }
  \newcommand{\Rnode}[5]{%
    \node[Rthfs] (#3) at (#1,#2) {\tiny $#4$};
    \foreach \m in {#5} { \draw[edge, red] (#3) -- (\m); }%
  }

  \Rnode{8*\dx}{0}{R}{\pr{}}{}
  \Snode{-8*\dx}{0}{P}{\pd{}}{}

  \Snode{-8.1*\dx}{\dy}{PP}{\pd{\pd{}}}{P}
  \Snode{-3*\dx}{\dy}{PR}{\pd{\pr{}}}{R}
  \Snode{3*\dx}{\dy}{PPR}{\pd{\pd{}\pr{}}}{R, P}
  \Rnode{8.1*\dx}{\dy}{RR}{\pr{\pr{}}}{R}

  \Snode{-8.2*\dx}{2*\dy}{BPP}{\pd{\XB}}{PP}
  \Snode{-4*\dx}{2*\dy}{BPR}{\pd{\XC}}{PR}
  \Snode{0*\dx}{2*\dy}{BPPR}{\pd{\XD}}{PPR}
  \Snode{4*\dx}{2*\dy}{BRR}{\pd{\XF}}{RR}
  \Rnode{8.2*\dx}{2*\dy}{RRR}{\pr{\XF}}{RR}

  \Snode{ -10*\dx}{3*\dy}{B2AB}{\pd{\XA\XB}}{P,PP}
  \Fnode{-7.4*\dx}{3*\dy}{B2AC}{\pd{\XA\XC}}{P,PR}
  \Fnode{  -4*\dx}{3*\dy}{B2AD}{\pd{\XA\XD}}{P,PPR}
  \Fnode{-1.3*\dx}{3*\dy}{B2AE}{\pd{\XA\XE}}{P,R}
  \Fnode{ 1.3*\dx}{3*\dy}{B2AF}{\pd{\XA\XF}}{P,RR}
  \Fnode{   4*\dx}{3*\dy}{B2BC}{\pd{\XB\XC}}{PP,PR}
  \Fnode{ 7.4*\dx}{3*\dy}{B2BD}{\pd{\XB\XD}}{PP,PPR}
  \Rnode{  10*\dx}{3*\dy}{R3}{\pr{\XE\XF}}{R,RR}

  \Fnode{ -10*\dx}{4*\dy}{B2BE}{\pd{\XB\XE}}{PP,R}
  \Fnode{-7.4*\dx}{4*\dy}{B2BF}{\pd{\XB\XF}}{PP,RR}
  \Fnode{  -4*\dx}{4*\dy}{B2CD}{\pd{\XC\XD}}{PR,PPR}
  \Fnode{-1.3*\dx}{4*\dy}{B2CE}{\pd{\XC\XE}}{PR,R}
  \Fnode{ 1.3*\dx}{4*\dy}{B2CF}{\pd{\XC\XF}}{PR,RR}
  \Fnode{   4*\dx}{4*\dy}{B2DE}{\pd{\XD\XE}}{PPR,R}
  \Fnode{ 7.4*\dx}{4*\dy}{B2DF}{\pd{\XD\XF}}{PPR,RR}
  \Fnode{  10*\dx}{4*\dy}{B2EF}{\pd{\XE\XF}}{R,RR}

  \node[] (DOTS) at (0,4.4*\dy) {$\vdots$};

\end{tikzpicture}
\caption{A partial sketch of the subobject classifier $\Omega$}
\label{fig:SubobjectClassifier}
\end{figure}
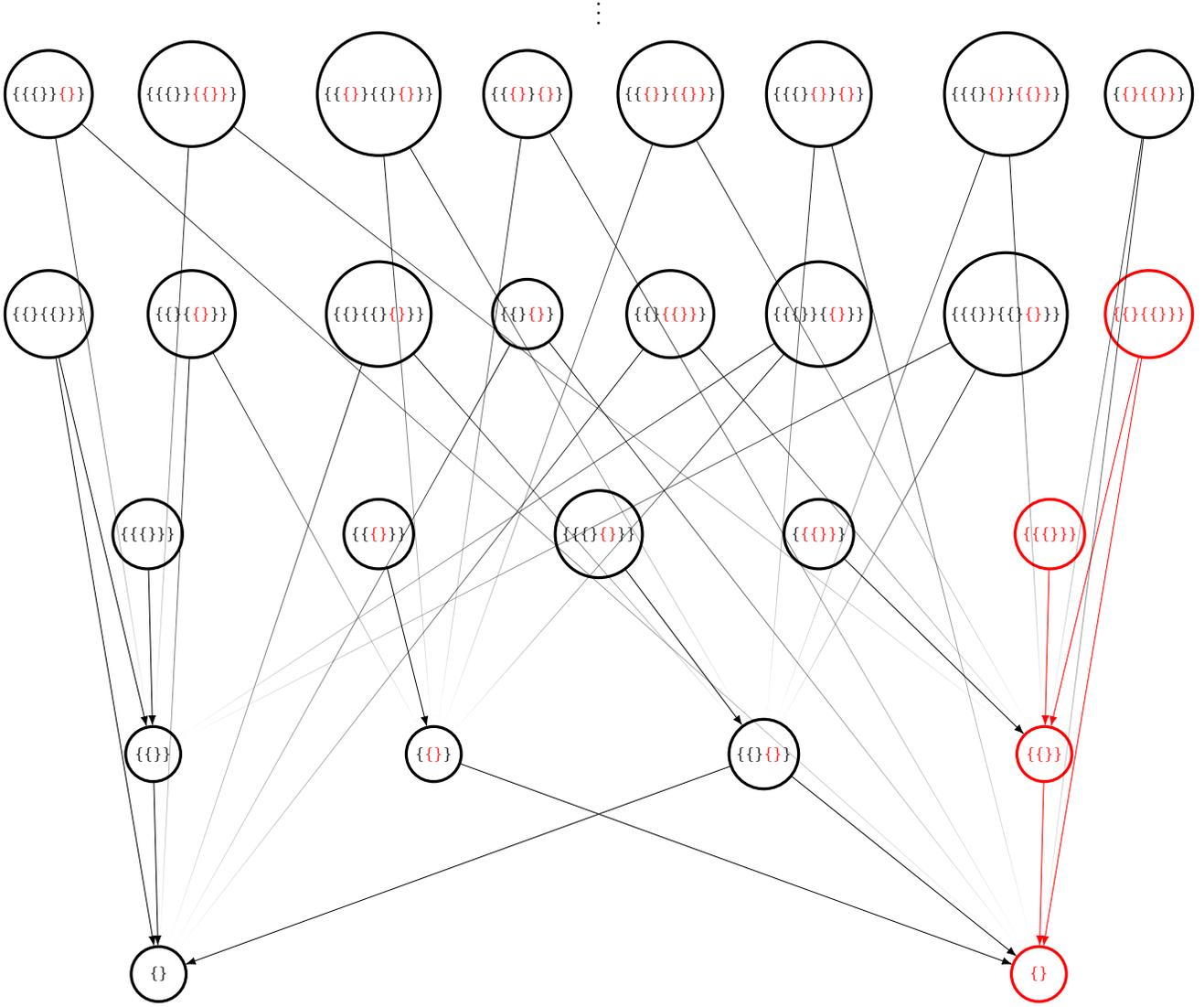

\begin{proposition}[Subobject classifier]\label{prop:SubobjectClassifier}
    The game $\Omega$ is the subobject classifier of the category of games $\Gs$.
\end{proposition}
\begin{proof}
    We prove that, for any subset $S\subset X$ of any game $\X=(X, \str)$, $S$ is a subgame if and only if the characteristic map $\chi_{S}\colon \X \to \H_{\{\top, \bot\}}$ factors through the subgame $\Omega \rightarrowtail \H_{\{\top, \bot\}}$.

    Assuming that $S$ is a subgame, we prove that $\chi_{S}(x)$ is truth-closed for any $x\in X$ by induction. At the step of $x\in X$, the condition (2) in \Cref{def:TruthClosed} is verified by the induction hypothesis. Regarding condition (1), if $\chi_{S}(x)$ is labeled by $\top$, i.e., $x\in S$, then every $\chi_{S}(x')$ is also labeled by $\top$ for any $x\rel_{\X} x'$, since $S$ is a subgame.

Conversely, assume that $\chi_{S}$ factors through $\Omega$. Then, for any $x\in S$, $\chi_{S}(x)$ is truth-closed and labeled by $\top$. Therefore, for every $x\rel_{\X} x'$, $\chi_{S}(x')$ is also labeled by $\top$, which implies that $S$ is a subgame.
\end{proof}

\begin{remark}[$\Gs$ is not a topos]\label{rem:GsisNotaTopos}
    The reader may wonder if $\Gs$ is a Grothendieck topos, since it satisfies a lot of nice properties like topos, including the existence of a subobject classifier, epi-mono factorization system, the frame structure of the subobject lattice, local finite presentability, and so on. Unfortunately, $\Gs$ is not even cartesian closed. In fact, the product functor $\Star_2\times {-} \colon \Gs \to \Gs$ does not preserve colimits (cf. \Cref{fig:TheStarGame}). For example, the colimit of the canonical $\Z/2\Z$-action on $\Star_2$ is $\Star_1$. However, the colimit of induced $\Z/2\Z$-action on $\Star_2\times \Star_2$ has $6$ elements. 
    \[
    \colim_{\Z/2\Z} \left(\Star_2\times \Star_2 \right)\xrightarrow{\text{non-iso}}   \Star_2\times \colim_{\Z/2\Z}\Star_2 = \Star_2 \times \Star_1 = \Star_2
    \]
    This shows that the functor $\Star_2 \times {-}$ is not cocontinuous; in particular, $\Gs$ is neither cartesian closed nor a topos.
\end{remark}



\section{\texorpdfstring{Birthday $\dashv$ Nim $\dashv$ Grundy number}{Birthday Nim Grundy number}}\label{app:GrundyNumberisLeftAdjointToNim}
The Bouton theorem (\Cref{thm:Bouton}) looks particularly simple due to the fact that the Grundy number map $\G_{\Nim{1}}\colon \N \to \N$ turns out to be the identity map $\id_{\N}$ (cf. \Cref{exmp:AnalysisOfNim}). The aim of this appendix is to explain this coincidence $\G_{\Nim{1}}=\id_{\N}$ by a \dq{categorical origin} of the mex function $\mex\colon \Pf(\N)\to \N$.

We start with a simple observation.
\begin{proposition}\label{prop:SectionIdentity}
For an endofunctor $T\colon \C \to \C$ on a category $\C$, a $T$-algebra $\A=(A,\alpha)$, and a recursive $T$-coalgebra $\X=(A, \str)$ on the same object $A\in \ob(\C)$, the following two condtions are equivalent.
    \begin{enumerate}
        \item $\str$ is the section of $\alpha$, i.e., $\alpha\circ \str = \id_A$.
        \item The hylomorphism $\hylo_{\A,\X}\colon A \to A$ coincides with the identity map $\id_A$.
    \end{enumerate}
%
\end{proposition}
\begin{proof}
Both conditions are equivalent to the commutativity of 
    \[
    \begin{tikzcd}
        TA\ar[r, "T(\id_{A})"]&TA\ar[d,"\alpha"']\\
        A\ar[u,"\str"]\ar[r,"\id_{A}"]&A.
    \end{tikzcd}
    \]
\end{proof}

Therefore, unsurprisingly, the coincidence $\G_{\Nim{1}}(=\hylo_{\Mex,\Nim{1}})=\id_{\N}$ comes from the fact that the structure map of nim $\nu$ is a section of $\mex$, i.e., $\mex\circ \nu = \id_\N$
\[
\mex\circ\nu(n)=\mex\{0,1, \dots, n-1\}=n.
\]
What might be a bit surprising is that $\mex$ is actually the canonical choice of a retraction of $\nu$ in a categorical sence!
When a canonical \dq{reverse} of a map $\nu\colon \N \to \Pf(\N)$ is needed, we category theorists think like \dq{ok, then consider its adjoint!}
\[
    \begin{tikzcd}
        \N \ar[r, shift left= 5pt,"\nu" name=A]&\Pf(\N)\ar[l, shift left= 5pt, "?"name=B]\ar[phantom, from= A, to=B, "\dashv" rotate=-90]
    \end{tikzcd}
\]
And the adjoint turns out to be mex! Here, we consider the usual order on each set, namely, the usual order $\leq$ on $\N$ and the inclusion relation on $\Pf(\N)$. 

\begin{proposition}[A characterization of mex]
    The right adjoint of $\nu$ is $\mex$.
    \[
    \begin{tikzcd}
        \N \ar[r, shift left= 5pt,"\nu" name=A]&\Pf(\N)\ar[l, shift left= 5pt, "\mex"name=B]\ar[phantom, from= A, to=B, "\dashv" rotate=-90]
    \end{tikzcd}
\]
\end{proposition}
\begin{proof}
    Take any $n\in \N$ and $S\in \Pf(\N)$. Then, we have
    \begin{align*}
                n\leq \mex{S} 
        &\iff   n\leq \min {S^{\mathrm{c}}}\\
        &\iff   \nu (n)^{\mathrm{c}} \supset S^{\mathrm{c}}\\
        &\iff   \nu (n) \subset S.
    \end{align*}
\end{proof}

Since $\nu$ is fully faithful, its right adjoint $\mex$ should be a retract of $\nu$ thanks to the general theory of adjoint functors (and the skeletality of the posets). So we can apply \Cref{prop:SectionIdentity}.






Similar phenomena are ubiquitous in this paper. For example, $\nu\colon \N \to \Pf(\N)$ also admits a right adjoint, which turns out to be $\xem$!
\[
    \begin{tikzcd}
        \N \ar[r, shift left= 5pt,"\nu" name=A]&\Pf(\N)\ar[l, shift left= 5pt, "\xem"name=B]\ar[phantom, from= A, to=B, "\dashv" rotate=90]
    \end{tikzcd}
\]
Therefore, \Cref{prop:SectionIdentity} implies that $\BirthDay_{\Nim{1}}\colon \N \to \N$ is also the identity map. This is the meaning of the title of this appendix: Birthday $\dashv$ Nim $\dashv$ Grundy number, which actually means
\[
\xem \dashv \nu \dashv \mex.
\]

    

\revmemo{
\invmemo{remark: LSC}
\invmemo{"almost all states are N-state"}
}
Another example is the remoteness.
\begin{example}\label{exmpl:remotenessAndEffeuillerLaMarguerite}
    The game $\ElM$ (\Cref{exmp:EffeuillerLaMarguerite}) is a section of the remoteness (\Cref{exmp:remoteness}) up to a canonical bijection.
\end{example}

\printbibliography

\end{document}